\documentclass[a4paper,11pt,twoside,reqno]{amsart}

\usepackage{graphicx}
\usepackage{subfig}
\usepackage[utf8]{inputenc}
\usepackage[plainpages=false,pdfpagelabels=true]{hyperref}
\usepackage{amssymb,amsthm}
\usepackage[margin=1in]{geometry}
\usepackage{slashed}
\usepackage{enumitem}
\usepackage{empheq}
\usepackage{tensor}
\usepackage{mathrsfs}
\usepackage{amsrefs}
\usepackage{todonotes} 

\setenumerate{font=\upshape,label=(\roman*),itemsep=0.1cm}

\renewcommand{\epsilon}{\varepsilon}

\def\restr#1#2{{
  \left.\kern-\nulldelimiterspace 
  #1 
  \vphantom{\big|} 
  \right|_{#2} 
  }}

\newcommand{\Z}{\mathbb{Z}}
\newcommand{\R}{\mathbb{R}}
\newcommand{\C}{\mathbb{C}}
\newcommand{\Q}{\mathbb{Q}}
\newcommand{\Quat}{\mathbb{H}}
\newcommand{\Sph}{\mathbb{S}}

\newcommand{\GL}{\mathbf{GL}}
\newcommand{\SL}{\mathbf{SL}}
\newcommand{\U}{\mathbf{U}}
\newcommand{\SU}{\mathbf{SU}}
\newcommand{\SO}{\mathbf{SO}}
\newcommand{\Spin}{\mathbf{Spin}}

\newcommand{\g}{\mathfrak{g}}
\newcommand{\h}{\mathfrak{h}}
\renewcommand{\k}{\mathfrak{k}}
\renewcommand{\u}{\mathfrak{u}}
\newcommand{\su}{\mathfrak{su}}
\newcommand{\spin}{\mathfrak{spin}}
\newcommand{\so}{\mathfrak{so}}
\renewcommand{\sl}{\mathfrak{sl}}
\newcommand{\e}{\mathfrak{e}}

\renewcommand{\v}{\mathfrak{v}}
\newcommand{\w}{\mathfrak{w}}
\newcommand{\crit}{\mathrm{crit}}

\DeclareMathOperator{\tr}{\mathrm{tr}}
\DeclareMathOperator{\artanh}{artanh}
\DeclareMathOperator{\arsinh}{arsinh}
\DeclareMathOperator{\arcosh}{arcosh}
\DeclareMathOperator{\arcoth}{arcoth}
\DeclareMathOperator{\sech}{sech}
\DeclareMathOperator{\csch}{csch}
\newcommand{\dirac}{\slashed{\mathrm{D}}}

\newcommand{\Ad}{\mathrm{Ad}}
\newcommand{\ad}{\mathrm{ad}}
\newcommand{\diff}{\mathrm{d}}
\newcommand{\QuatRe}{\mathfrak{C}}
\newcommand{\QuatIm}{\mathfrak{H}}

\usepackage[dvipsnames]{xcolor}

\newtheorem{theorem}{Theorem}[section]
\newtheorem{proposition}[theorem]{Proposition}
\newtheorem{lemma}[theorem]{Lemma}
\newtheorem{corollary}[theorem]{Corollary}
\theoremstyle{definition}

\newtheorem{remark}[theorem]{Remark}

\numberwithin{equation}{section}

\makeatletter
\newsavebox\myboxA
\newsavebox\myboxB
\newlength\mylenA
\newcommand*\xoverline[2][0.75]{%
    \sbox{\myboxA}{$\m@th#2$}%
    \setbox\myboxB\null
    \ht\myboxB=\ht\myboxA%
    \dp\myboxB=\dp\myboxA%
    \wd\myboxB=#1\wd\myboxA
    \sbox\myboxB{$\m@th\overline{\copy\myboxB}$}
    \setlength\mylenA{\the\wd\myboxA}
    \addtolength\mylenA{-\the\wd\myboxB}%
    \ifdim\wd\myboxB<\wd\myboxA%
       \rlap{\hskip 0.5\mylenA\usebox\myboxB}{\usebox\myboxA}%
    \else
        \hskip -0.5\mylenA\rlap{\usebox\myboxA}{\hskip 0.5\mylenA\usebox\myboxB}%
    \fi}
\makeatother
\newcommand{\conj}[1]{\xoverline{#1}}

\usepackage{pgfplots}
\usepackage{pgfplotstable}
\pgfplotsset{compat = 1.18}

\usetikzlibrary{decorations.markings}
 
\tikzset{decorated arrows/.style={
    postaction={
        decorate,
        decoration={
            markings,
            mark=between positions 0 and 1 step 15mm with {\arrow[black]{stealth};}
            }
        },
    }
}
\pgfmathdeclarefunction{artanh}{1}{%
  \pgfmathparse{1/2*ln((1+#1)/(1-#1))}%
}

\title{Spherically symmetric Dirac-Yang-Mills pairs on Riemannian manifolds}
\date{\today}

\author{Adam Lindstr\"om}
\address{University of Vienna, Faculty of Mathematics\\
Oskar-Morgenstern-Platz 1, 1090 Vienna, Austria\\}
\email{adam.lindstroem@univie.ac.at}

\author{Marko Sobak}
\address{University of Vienna, Faculty of Mathematics\\
Oskar-Morgenstern-Platz 1, 1090 Vienna, Austria\\}
\email{marko.sobak@univie.ac.at}


\thanks{The authors gratefully acknowledge the support of the Austrian Science Fund (FWF) through the project "The Standard Model as a Geometric Variational Problem" (DOI: 10.55776/P36862).
}

\begin{document}
\begin{abstract}
    In this paper we construct examples of spherically symmetric Dirac-Yang-Mills pairs on Riemannian 3-manifolds with the structure group $\SU(2)$. This approach yields coupled solutions (i.e.\ the connection is not a Yang-Mills connection) and among them are solutions on $\Sph^1(r_1)\times \Sph^2(r_2)$ for certain radii $r_1$ and $r_2$. We further show how to use such pairs to induce Dirac-Yang-Mills pairs on Riemannian products of arbitrary dimension. These are, to the authors' best knowledge, the first examples of coupled Dirac-Yang-Mills pairs on a closed Riemannian spin manifold.
\end{abstract}

\maketitle

\section{Introduction}

Let $(M,g)$ be a Riemannian spin manifold, $G$ a compact Lie group with Lie algebra $\g$, and $G \to P \to M$ a principal fiber bundle.
We study the \emph{Dirac-Yang-Mills Lagrangian density}
\begin{equation}\label{eq-DYM density}
    (\omega, \Psi) \mapsto \frac12 |F_\omega|^2 + \Re \langle \Psi, \dirac_\omega\Psi\rangle,
\end{equation}
where
\begin{itemize}[itemsep=0.1cm]
    \item $\omega \in \Omega^1(P,\g)$ is a connection on $P$ with curvature form $F_\omega \in \Omega^2(M, \Ad(P))$,
    \item $\Psi \in \Gamma(\Sigma M \otimes E)$ is a spinor field twisted by an associated bundle $E = P \times_\chi V$ for some representation $\chi : G \to \GL(V)$, and $\dirac_\omega\Psi = e^j \cdot (\nabla_\omega\Psi)(e_j)$ denotes the twisted Dirac operator for the connection $\nabla_\omega$ on $\Sigma M \otimes E$ induced by the spin Levi-Civita connection and the connection on $E$ associated to $\omega$. Here, Clifford multiplication acts only on the $\Sigma M$-factor.
\end{itemize}
The associated Euler-Lagrange equations are given by 
\begin{equation}\label{eq-dym}
    \begin{cases}
        \diff_\omega^\ast F_\omega = \mathfrak{J}[\Psi], \\[0.1cm]
        \dirac_\omega \Psi = 0,
    \end{cases}
\end{equation}
where the current is explicitly given by
\begin{equation}\label{eq-dym-current}
    \mathfrak{J}[\Psi] = \langle e_j \cdot \Psi, \,\chi_*(\tau_a)\Psi \rangle\, e^j \otimes \tau_a,
\end{equation}
for a local frame $\tau_a$ of $\Ad(P)$, where 
$\Ad(P)$ acts only on the $E$-factor of $\Psi$. 
We will refer to \eqref{eq-dym} as the \emph{Dirac-Yang-Mills equations}, and its solutions as \emph{Dirac-Yang-Mills pairs}.
For brevity, we henceforth abbreviate Dirac-Yang-Mills by DYM in the text.

The first term in the DYM Lagrangian density \eqref{eq-DYM density} is the Lagrangian density of the pure Yang-Mills theory which on its own is of great interest from a differential geometry perspective. In particular the moduli space of (anti-)self-dual solutions has played a key role in understanding the topology of 4-manifolds thanks to the seminal work of Donaldson \cite{Donaldson} (see the book \cite{DoKr} for a detailed exposition). The celebrated Atiyah-Singer index theorem \cite{AtSi} as well as results such as the alternative proof of completeness of the ADHM construction \cite{ADHM} found in \cite{GodCor} show a connection between Yang-Mills theory and the kernel of the associated twisted Dirac operators.

In quantum field theory, twisted spinors are employed to model fermions, while connections on principal bundles are used to model bosons. The density \eqref{eq-DYM density} appears as part of the full Lagrangian density of the standard model of particle physics and describes the interaction betwen a massless fermionic matter field and a (massless) bosonic field given by the connection. For details concerning the standard model or gauge theory in general we encourage the reader to consult \cite{Hamilton}.

The model resulting from the Lagrangian density \eqref{eq-DYM density} is therefore a natural object of study from both a mathematics and a physics perspective.
The study of the DYM equations on Riemannian manifolds was initiated by Parker in \cite{Parker}, and continued by Otway in \cite{Otway}, Li in \cite{Li} and Isobe in \cite{Isobe}. A similar system was also studied by Jost, Ke{\ss}ler, Wu and Zhu in \cite{Jost}. The focus in the articles \cites{Isobe,Li,Otway,Parker} has been mostly on regularilty theory and on extending removability of singularities results from Yang-Mills theory to DYM.

Note that if $\omega$ is a Yang-Mills connection (i.e.\ satisfies the pure Yang-Mills equation $\diff_\omega^*F_\omega = 0$), then a DYM pair $(\omega, \Psi)$ necessarily has vanishing current, i.e.\ $\mathfrak{J}[\Psi] \equiv 0$.
DYM pairs with vanishing current are said to be \emph{uncoupled}, and otherwise we say that they are coupled.
In \cite{A} the first named author uses index theory to establish existence results for such uncoupled solutions when the base manifold satisfies $\dim M \equiv 0 \ (\mathrm{mod} \ 4)$. On the other hand it is shown, for closed base manifolds of any dimension, that the set of connection forms $\omega$ for which $\mathfrak{J}[\Psi] = 0$ for all $\Psi \in \ker (\dirac_\omega)$ contains a dense open subset of the affine space of smooth connection forms (with respect to the $C^\infty$-topology).
This suggests that global coupled DYM pairs on closed Riemannian manifolds are scarce, and therefore it is of interest to construct examples in more specific settings.
To this end, in this work we construct spherically symmetric solutions of the DYM equations with structure group $\SU(2)$ on spherically symmetric three-dimensional Riemannian manifolds.

\begin{theorem}\label{thm-existence-3-dim}
    Let $N = \R$ or $N = \Sph^1$.
    Then in both cases, there exist families of spherically symmetric Riemannian metrics $g$ on $N\times \Sph^2$ admitting globally defined coupled spherically symmetric Dirac-Yang-Mills pairs with structure group $\SU(2)$.
\end{theorem}

These solutions can be lifted to yield coupled DYM pairs on spherically symmetric manifolds of any dimension.

\begin{theorem}\label{thm-existence-any-dim}
    Let $(M,g)$ be a $3$-manifold such that there exists a globally defined coupled DYM pair $(\omega, \Psi)$ on $M$ with structure group $\SU(2)$. Then for any $n \in \mathbb{Z}_{>0}$, there exists an $n$-dimensional Riemannian spin manifold $B$ such that $(\omega,\Psi)$ lifts to a family of coupled DYM pairs  $(\widetilde{\omega},\widetilde{\Psi}_c)$ on $M \times B$, parametrized by $c \in \C^k \setminus \{0\}$ for some $k \geq 1$ depending on $B$. Moreover, we can always take $B$ to be closed.
\end{theorem}

As both the constructions leading to Theorem \ref{thm-existence-3-dim} and to Theorem \ref{thm-existence-any-dim} are explicit this provides explicit examples of coupled DYM pairs on (closed) Riemannian manifolds which are, to the authors' best knowledge, the first such examples to be found.

Let us briefly sketch out a few details about the construction.
The appropriate spherically symmetric ansatz for the DYM equations involves (see \S \ref{sec-su(2)-sph-sym} for the details):
\begin{itemize}[itemsep=0.1cm]
    \item The background manifold $M = N \times \Sph^2$, where $N$ is a one-dimensional manifold, i.e.\ either $N=\R$ or $N=\Sph^1$, endowed with the metric $g = r(s)^2 (\diff s^2 + g_{\Sph^2})$ for some metric coefficient $r : N \to \R$. We will denote this Riemannian manifold by $(N \times \Sph^2)_r$.
    \item A complex-valued function $z : N \to \C$ parametrizing the connection $\omega$.
    \item A quaternion-valued%
    \footnote{
    \label{foot-quaternions}
    Here and throughout the rest of this work, quaternions are represented as pairs of complex numbers, in the sense that a quaternion $q \in \Quat$ can uniquely be written as $q = z+jw$ for $z,w \in \C$, where $j$ is the quaternionic unit with $j^2=-1$ and $ij = -ji$. We will refer to $z$ and $w$ respectively as the complex and quaternionic part of $q$, and denote this by $\QuatRe(q)=z$ and $\QuatIm(q)=w$.
    }
    function $\xi : N \to \Quat$ parametrizing the spinor field $\Psi$, and depending also on a non-negative odd integer $n$ that determines the representation of $\SU(2)$ used for the twisting bundle.
\end{itemize}
Then the DYM equations \eqref{eq-dym} can be written as
\begin{subequations}
    \begin{empheq}[left=\empheqlbrace]{align}
        \label{eq-constraint-intro}
        0&= \Im(\conj{z}Y) + \frac14 |\xi|^2,\\
        \label{eq-z-intro}
        0&= z' - r Y,\\
        \label{eq-Y-intro}
        0&= Y' + \tfrac{1}{r} z(1-|z|^2) + \lambda \QuatRe(\xi) \QuatIm(\xi),\\
        \label{eq-xi-intro}
        0&= \xi' +  i\lambda \xi \conj{z} j.
    \end{empheq}
\end{subequations}
where $\lambda := (n+1)/2$ and $Y$ is an auxiliary variable used for the first-order reduction.
Here, \eqref{eq-constraint-intro} can be viewed as a constraint equation, since it is preserved by the evolution equations (\ref{eq-z-intro}--\ref{eq-xi-intro}) assuming that it holds initially.

One of the main difficulties with studying the system (\ref{eq-constraint-intro}--\ref{eq-xi-intro}) in general is that it is non-autonomous, due to the presence of the metric coefficient $r$, which should be viewed as a fixed externally given parameter. 
In particular, much of the standard theory of dynamical systems does not apply directly. Furthermore, there are no conditions that this coefficient should satisfy a priori (except in the case $N=\Sph^1$, where $r$ should be periodic in $s$), which often makes it difficult to make claims about the monotonicity of quantities, and consequently the analysis also becomes difficult in general. 
One way of approaching this issue would be to impose certain monotonicity and/or asymptotic conditions on $r$, and then to study the system with respect to the different cases.

In this paper, as a first step, we take a different approach by instead fixing certain quantities in a way that makes the system simpler to study (e.g.\ $|z|\equiv\text{const}$ considerably simplifies \eqref{eq-Y-intro}). 
This then also reduces the number of degrees of freedom, and in particular fixes the metric coefficient $r$.
It may not be come as too much of a surprise that this approach yields some global solutions of the system, so that they are defined on the manifold $(\R \times \Sph^2)_r$. However, it is surprising that this also approach also provides solutions that are periodic, and thus induce DYM pairs on \emph{closed} Riemannian manifolds of the form $(\Sph^1 \times \Sph^2)_r$.

For the lifting procedure resulting in Theorem \ref{thm-existence-any-dim} we first note that, with $\dim M = 2m-1$, $\dim B = n$ and $f \in C^\infty(B; \R_{>0})$, the spinor bundle $\Sigma (B \times_f M)$ of the warped product $(B\times M, h + f^2g)$ is isomorphic to
\begin{itemize}
    \item $\Sigma B \otimes (\Sigma M \oplus \Sigma M)$ for $n$ odd,
    \item $\Sigma B \otimes \Sigma M$ for $n$ even.
\end{itemize}
By starting at a spherically symmetric $3$-manifold $(M,g)$ as in Theorem \ref{thm-existence-3-dim} admitting a coupled DYM pair $(\omega, \Psi)$ and iterating the constructions detailed in \S \ref{sec-lifting} we then construct lifts to product manifolds of any dimension. More explicitly, let $(B,h)$ be a Riemannian spin manifold with $\dim B = n$ and $f \in C^\infty(B; \R_{>0})$ satisfying:
    \begin{enumerate}[label = (\roman*)]
        \item  if $n = 1$, $B$ is an open interval, $\R$ or $\Sph^1$ equiped with the trivial spin structure and $f\in C^\infty(B;\R_{>0})$ is arbitrary,
        \item  if $n$ is even, $(B,h)$ admits a nontrivial globally defined parallel spinor and $f \equiv 1$,
        \item if $n>1$ and odd, $B = B^1 \times B^{n-1}$ where $B^1$ satsifies the conditions in (i) and $B^{n-1}$ those in (ii) and $f \equiv 1$.
    \end{enumerate}
Then $(\omega,\Psi)$ induces a coupled DYM pair $(\widetilde{\omega},\widetilde{\Psi})$ on the warped product $(B\times_f M)$ where in each case $\widetilde{\omega} = \mathrm{pr}_2^*\omega$ and $\widetilde\Psi$ is given by:
\begin{enumerate}[label = (\arabic*)]
        \item for $B$ as in (i), $\widetilde\Psi = cf^{-\frac{3}{2}}(\Psi \oplus \frac{-i}{2|c|^2}\Psi)$ where $c \in \C\setminus\{0\}$,
        \item for $B$ as in (ii), $\widetilde\Psi  = \Phi \otimes \Psi$ where $\Phi \in \Gamma(\Sigma^+B)$ is a positive chirality parallel spinor satisfying $|\Phi| \equiv 1$,
        \item for $B = B^1 \times B^{n-1}$ as in (iii), $\widetilde{\Psi} = c\Phi \otimes(\Psi \oplus \frac{-i}{2|c|^2}\Psi)$ where $c \in \C\setminus\{0\}$ and $\Phi \in \Gamma(\Sigma^+M^{n-1})$ is a positive chirality parallel spinor satisfying $|\Phi| \equiv 1$.
\end{enumerate}
There are closed examples of admissible spin manifolds $(B,h)$ for each of the cases (i--iii).

\section{\texorpdfstring{$\SU(2)$}{SU(2)}-ansatz in spherical symmetry}
\label{sec-su(2)-sph-sym}

\subsection{General theory of invariant bundles, connections, and sections}
\label{subsec-K-symmetry}
Let $K$ be a compact Lie group.
We say that a (semi-)Riemannian manifold $(M,g)$ is \emph{$K$-symmetric} if $K$ acts on $M$ by isometries, and $K$ is then referred to as the \emph{symmetry group}.
Provided that $K$ is compact, there exists an open and dense submanifold of $M$ that is diffeomorphic to $N \times K/H$, for some connected manifold $N$ and a closed subgroup $H \subset K$.
With a minor loss of generality (i.e.\ by removing the singular orbits), we will assume therefore that
\begin{equation*}
    M \cong N \times K/H,
\end{equation*}
where $K$ acts by left multiplying the cosets in the second factor. 

Let $G$ be another compact Lie group, and let $P\to M$ be a principal $G$-bundle. 
If $M$ is $K$-symmetric and the natural left action of $K$ lifts to a faithful left action on $P$ by bundle automorphisms (i.e.\ the left action of $K$ and the right action of $G$ commute), then $P$ is also said to be \emph{$K$-symmetric}.
Two $K$-symmetric principal $G$-bundles $P\to M$ and $Q\to M$ are said to be \emph{equivalent} if there exists a bundle diffeomorphism $P \to Q$ which is equivariant with respect to both the right action of $G$ and the left action of $K$.

A map $\phi : P \to W$ for a vector space $W$ (or more generally a $W$-valued tensor field on $P$) is said to be \emph{$K$-invariant} if $\ell_k^\ast \phi = \phi$.
In particular, this definition applies to connections $\omega \in \Omega^1(P,\g)$.
Moreover, if $\rho : G \to \GL(W)$ is a representation and $\phi : P \to W$ is $G$-equivariant in the sense that $r_g^*\phi = \rho(g^{-1})\phi$ for all $g \in G$, then $\phi$ can be identified with a section $\Phi$ of the associated vector bundle $E = P \times_\rho W$ via $\Phi_{\pi(p)} = [p, \phi(p)]$. 
Then $\Phi \in \Gamma(E)$ is said to be $K$-invariant if and only if $\phi$ is $K$-invariant, or equivalently if and only if $k \cdot \Phi_{\pi(p)} = \Phi_{\pi(kp)}$ for all $k \in K$ and $p \in P$, with respect to the natural induced left action of $K$ on $E$.

There is a well-known theory classifying invariant principal bundles and associated objects, which we briefly recall, and refer to \cites{Kunzle, Brodbeck, Harnad} for the details. 
Let $M = N \times K/H$ be a $K$-symmetric manifold and $G$ a Lie group.
Denote the Lie algebras of $K,H,G$ respectively by $\k, \h, \g$.
\begin{enumerate}
    \item The equivalence classes of $K$-symmetric principal $G$-bundles over $M$ are classified by pairs $([\lambda], Q)$ where $[\lambda]$ is a conjugacy class of Lie group homomorphisms $\lambda : H \to G$, and $Q \to N$ is a principal $Z$-bundle, where $Z$ is the center of $\lambda(H)\subset G$. Explicitly, $P = P([\lambda], Q) \to M$ is constructed as the quotient $(Q \times K \times G) / (Z \times H)$ with respect to the equivalence relation
    \begin{equation*}
        (q,k,g) \sim (qz, kh, z^{-1}\lambda(h^{-1})g), \qquad (z,h) \in Z \times H,
    \end{equation*}
    equipped with the projection $\pi([q,k,g]) = (\pi_Q(q),kH)$, and left (resp.\ right) action of $K$ (resp.\ $G$) given by
    \begin{equation*}
        k_0 \cdot [q,k,g] = [q,k_0k, g], \qquad [q,k,g]\cdot g_0 = [q,k,gg_0],
    \end{equation*}
    for $k_0 \in K$ and $g_0 \in G$.
    \item The $K$-invariant connections $\omega$ on $P([\lambda], Q)$ are classified by pairs $(\Lambda, \eta)$, where $\eta$ is a connection on $Q$, and $\Lambda : Q \to \text{Hom}(\k, \g)$ is a map satisfying the Wang conditions
    \begin{equation}\label{eq-wang-conditions}
        \Lambda \circ \Ad_h = \Ad_{\lambda(h)} \circ \Lambda, \qquad \Lambda|_{\h} = \diff\lambda,
    \end{equation}
    for all $h \in H$. 
    Explicitly, if we denote by $\Pi : Q \times K \times G \to P([\lambda], Q)$ the natural projection, then the invariant connection $\omega$ induced by $(\Lambda, \eta)$ is given by
    \begin{equation*}
        \omega_{[q,k,g]}\circ \diff \Pi(\alpha \oplus \beta \oplus \gamma) = \Ad_{g^{-1}} \left( \Lambda_q \circ \mu^K_k(\beta) + \eta_q(\alpha) \right) + \mu^G_g(\gamma),
    \end{equation*}
    where $\mu^K$ and $\mu^G$ are the left-invariant Maurer-Cartan forms of $K$ and $G$ respectively, and $\alpha \in T_q Q, \, \beta \in T_k K,\, \gamma \in T_g G$.
    \item The $K$-invariant sections $\Phi \in \Gamma(P([\lambda], Q) \times_\rho W)$ for a representation $\rho : G \to \GL(W)$ are characterized by $Z$-equivariant maps $\phi : Q \to W$ (i.e.\ sections of the associated bundle $Q \times_\rho W$) such that 
    \begin{equation}\label{eq-section-invariance-condition}
        \rho(\lambda(h)) \phi = \phi
    \end{equation}
    for all $h \in H$, explicitly
    \begin{equation*}
        \Phi_{(\pi_Q(q), kH)} = [[q,k,g], \rho(g^{-1})\phi(q)] = [[q,k,e], \phi(q)].
    \end{equation*}
\end{enumerate}

A choice of (local) sections $\xi$ of $Q \to N$ and $\sigma$ of $K \to K/H$ induces the section $p(s,kH) = [\xi(s), \sigma(kH), e]$ of $P([\lambda], Q)$, with respect to which the invariant connection $\omega$ corresponding to $(\Lambda, \eta)$ and section $\Phi$ corresponding to $\phi$ respectively take the form
\begin{equation*}
    (p^\ast \omega)_{(s,kH)} = (\xi^*\Lambda)_s \circ (\sigma^\ast \mu^K)_{kH} + (\xi^*\eta)_s,
    \quad
    \Phi_{(s,kH)} = [p(s,kH), (\xi^*\phi)(s)].
\end{equation*}

Let us note that if the classifying bundle $Q$ is trivial%
\footnote{
In our setting, we will have $Z=\U(1)$ and either $N = \R$ or $N = \Sph^1$, so that any principal $Z$-bundle $Q$ over $N$ is necessarily trivial. More generally, this is also the case if $N$ is contractible for arbitrary $\dim N$ and $Z$.
}
then we can identify $Q \cong N \times Z$ and elements of $P([\lambda], Q)$ can be written as $[(s,z),k,g]$ for $s \in N$ and $z\in Z$.
In this case, we can identify $P([\lambda], Q)$ with a simpler principal bundle $P[\lambda]$ defined as the quotient $N \times (K \times G) / H$ with respect to the equivalence relation
\begin{equation*}
    (k,g) \sim (kh, \lambda(h^{-1})g), \qquad (k,g)\in K\times G, \quad h \in H,
\end{equation*}
equipped with the natural projection, left action of $K$, and right action of $G$.
The isomorphism $P([\lambda], Q) \to P[\lambda]$ is explicitly given by $[(s,z),k,g] \mapsto (s, [k,zg])$.
In this case, the maps $(\Lambda, \eta)$, $\phi$ used to classify invariant connections/sections can be globally viewed as smooth maps $\Lambda : N \to \text{Hom}(\k, \g)$, $\eta : N \to \mathfrak{z}$ (here $\mathfrak{z}$ denotes the Lie algebra of $Z$), and $\phi : N\to W$, satisfying the respective invariance conditions \eqref{eq-wang-conditions} and \eqref{eq-section-invariance-condition}.

Throughout the rest of this work, we will apply this general theory to the case when the structure group $G=\SU(2)$.
We will also assume the symmetry group is $\SU(2)$, and that $H=\U(1)$. The corresponding objects will be said to be \emph{spherically symmetric}.

\subsection{Some important matrices and maps}

In this subsection we fix the notation and recall some basic facts for certain important objects.
We define a basis for $\su(2)$ by
\begin{equation}\label{eq-su(2)-basis}
    \tau_k = -\frac{i}{2} \sigma_k,
\end{equation}
where
\begin{equation*}
    \sigma_1 = \begin{bmatrix}
        0 & 1\\
        1 & 0
    \end{bmatrix},
    \qquad
    \sigma_2 = \begin{bmatrix}
        0 & -i\\
        i & 0
    \end{bmatrix},
    \qquad
    \sigma_3 = \begin{bmatrix}
        1 & 0\\
        0 & -1
    \end{bmatrix},
\end{equation*}
are the Pauli matrices.
Here, $\tau_3$ spans the subalgebra $\mathfrak{u}(1)$.
If we equip $\mathfrak{su}(2)$ with the usual Ad-invariant inner product
\begin{equation}\label{eq-su(2)-metric}
    \langle Z, W \rangle_{\mathfrak{su}(2)} = 2\tr(Z^\dagger W) = -2\tr(ZW),
\end{equation}
then the basis elements $\tau_k$ are orthonormal.

Note that the basis elements satisfy
\begin{equation*}
    \tau_j\tau_k = \frac12\epsilon_{jk\ell} \tau_\ell - \frac{1}{4}\delta_{jk}I_2
\end{equation*}
where $\epsilon_{jk\ell}$ is the sign of the permutation of the indices, $\delta_{jk}$ is the Kronecker delta, and $I_2$ is the two-dimensional identity matrix.
In particular,
\begin{equation*}
    [\tau_j, \tau_k] = \epsilon_{jk\ell} \tau_\ell, \qquad \{\tau_j, \tau_k\} = -\frac{1}{2} \delta_{jk} I_2.
\end{equation*}
Thus, we also see that the matrices
\begin{equation}\label{eq-clifford-basis}
    \gamma_k = 2\tau_k = -i\sigma_k    
\end{equation}
form a representation of the Clifford algebra of $\R^3$,
i.e.\ $\{\gamma_j, \gamma_k\} = -2\delta_{jk} I_2$.
The representation $\gamma_k$ of the Clifford algebra induces the spinor representation of $\Spin(3)$.
The Lie algebra $\spin(3)$ has a basis of Clifford products of the form $\gamma_i \gamma_j$ for $i<j$.
Since
\begin{equation*}
    \frac12 \gamma_2\gamma_3 = \tau_1, \qquad \frac12 \gamma_3\gamma_1 = \tau_2, \qquad \frac12 \gamma_1\gamma_2 = \tau_3.
\end{equation*}
we see that $\spin(3) = \su(2)$ as Lie algebras, and in particular $\Spin(3) = \SU(2)$ since both groups are simply connected.
The spinor representation of $\Spin(3)$ coincides to the fundamental representation of $\SU(2)$, explicitly
\begin{equation*}
    \begin{bmatrix}
        z & -\conj{w} \\ w & \conj{z}        
    \end{bmatrix}
    =
    \Re(z) I_2 - \Im(w) \gamma_1 + \Re(w) \gamma_2 - \Im(z) \gamma_3,
\end{equation*}
where $|z|^2+|w|^2 = 1$.

It will also be convenient to define
\begin{equation}\label{eq-so(3)-basis}
    L_1 = \begin{bmatrix}
        0 & 0 & 0\\
        0 & 0 & -1\\
        0 & 1 & 0
    \end{bmatrix},
    \qquad
    L_2 = \begin{bmatrix}
        0 & 0 & 1\\
        0 & 0 & 0\\
        -1 & 0 & 0
    \end{bmatrix},
    \qquad
    L_3 = \begin{bmatrix}
        0 & -1 & 0\\
        1 & 0 & 0\\
        0 & 0 & 0
    \end{bmatrix}
\end{equation}
as a basis for the Lie algebra $\so(3)$, satisfying
\begin{equation*}
    [L_j, L_k] = \epsilon_{jk\ell} L_\ell.
\end{equation*}
In fact, the natural double covering map
\begin{equation}\label{eq-su(2)-so(3)-double-cover}
    \SU(2) \to \SO(3), \qquad g \mapsto \Ad_g \in \SO(\su(2)) \cong \SO(3)
\end{equation}
induces the Lie algebra isomorphism
\begin{equation*}
    \su(2) \to \so(3), \qquad X \mapsto \ad(X) \in \so(\su(2)) \cong \so(3),
\end{equation*}
and one can explicitly verify that this isomorphism maps
$\tau_k \mapsto L_k$,
where the $L_k$-matrices are expressed with respect to the basis $(\tau_1,\tau_2,\tau_3)$.

Next, we recall the Hopf fibration
\begin{equation*}
    \pi_{\text{Hopf}} : \Sph^3 \cong \SU(2) \to \Sph^2 \cong \SU(2)/\U(1), \qquad (z,w) \mapsto (2\conj{z} w, |z|^2 - |w|^2)
\end{equation*}
where one identifies $\mathbb{S}^3 \subset \C^2 \cong \Quat$ with $\SU(2)$ via
\begin{equation*}
    z + jw \leftrightarrow (z,w) \leftrightarrow \begin{bmatrix}
        z & -\conj{w}\\
        w & \conj{z}
    \end{bmatrix},
    \qquad
    |z|^2+|w|^2 = 1,
\end{equation*}
where $j$ is the quaternionic imaginary unit.
One can show that the Hopf fibration is equivariant in the sense that if $k \in \SU(2)$, then
\begin{equation}\label{eq-hopf-equivariant}
    \pi_{\text{Hopf}} \circ \ell_k = \Ad_k \circ \pi_{\text{Hopf}},
\end{equation}
where $\ell_k$ denotes left multiplication by $k$, and on the right-hand side it is understood that $\R^3 \cong \su(2)$ by identifying $\tau_1,\tau_2,\tau_3$ with the standard Cartesian unit vectors.
In particular, with this identification we can also write $\pi_{\text{Hopf}}(k) = \Ad_k \circ \pi_{\text{Hopf}}(I_2) = \Ad_k(\tau_3)$.

The map
\begin{equation}\label{eq-hopf-section}
    \sigma : (\theta,\varphi) \mapsto \exp(\varphi \tau_3) \exp(\theta \tau_2)
    = 
    \begin{bmatrix}
    e^{-\frac{i\varphi}{2}} \cos\frac{\theta}{2} & -e^{-\frac{i\varphi}{2}}\sin\frac{\theta}{2}  \\ 
    e^{\frac{i\varphi}{2}}\sin\frac{\theta}{2} & e^{\frac{i\varphi}{2}} \cos\frac{\theta}{2} 
    \end{bmatrix} 
\end{equation}
defines a local section of the bundle $\mathbb{S}^3 \to \mathbb{S}^2$ with respect to the standard spherical coordinates $(\theta,\varphi)$ defined on $\mathbb{S}^2$ minus the poles.
We observe that
\begin{align}
    \nonumber
    \Ad_{\sigma(\theta,\varphi)} (\tau_1)
    &= \cos\theta\cos\varphi \, \tau_1 + \cos\theta\sin\varphi \, \tau_2 - \sin\theta \, \tau_3 = \tau_\theta,\\
    \label{eq-Ad-cartesian-spherical}
    \Ad_{\sigma(\theta,\varphi)} (\tau_2)
    &= -\sin\varphi \,\tau_1 + \cos\varphi \,\tau_2 = \tau_\varphi,\\
    \nonumber
    \Ad_{\sigma(\theta,\varphi)} (\tau_3)
    &= \sin\theta\cos\varphi \, \tau_1 + \sin\theta\sin\varphi \, \tau_2 + \cos\theta \, \tau_3 = \tau_r,
\end{align}
where $\tau_r, \tau_\theta, \tau_\varphi$ are the "spherical" $\tau$-matrices.
In particular, the image of $\sigma(\theta,\varphi)$ under the natural two-fold covering $\SU(2) \to \SO(3)$ provides the relation between the Cartesian basis elements $\tau_1,\tau_2,\tau_3$ and the spherical basis elements $\tau_\theta, \tau_\varphi, \tau_r$.

Finally, we note that if $\SU(2)/\U(1) \cong \Sph^2$ is equipped with the canonical round metric, then the equivariance condition \eqref{eq-hopf-equivariant} together with the identities \eqref{eq-Ad-cartesian-spherical} shows that the Hopf map is a Riemannian submersion.

\subsection{Base manifold}
\label{subsec-base}

Let $N$ be a smooth connected one-dimensional manifold (i.e.\ either $N\cong \R$ or $N\cong \Sph^1$). We will consider manifolds of the form
\begin{equation*}
    M = N \times \Sph^2.
\end{equation*}
In particular, such manifolds are $\SU(2)$-symmetric in the sense of \S \ref{subsec-K-symmetry}, as $\Sph^2 \cong \SU(2) / \U(1)$ via the Hopf fibration. 
The most general spherically symmetric (i.e.\ $\SU(2)$-invariant) metric on $M$ then takes the form
\begin{equation}\label{eq-sph-sym-metric}
    g = \alpha(s)^2 \, \diff s^2 + r(s)^2\, g_{\Sph^2},
\end{equation}
where $\alpha, r : N \to \R$ are smooth positive functions, $s$ is a coordinate on $N$, 
and $g_{\Sph^2}$ is the standard round metric on $\Sph^2$, i.e.\
$g_{\Sph^2} = \diff\theta^2 + \sin^2\theta \, \diff\varphi^2$
with respect to the standard spherical coordinates $(\theta,\varphi)$. 

\begin{remark}
    \label{remark-s-coordinate}
    In the case $N = \Sph^1$, one can also view $M = \Sph^1\times \Sph^2$ locally as a submanifold of $\R \times \Sph^2$ equipped with the metric \eqref{eq-sph-sym-metric} where $\alpha$ and $r$ are extended periodically, so we can without loss of generality take $N$ to be an interval.
    Only one of the functions $\alpha$, $r$ is then truly a free parameter, since we can always transform the coordinate $s$ to fix one of the parameters as desired (while possibly also shrinking $\R$ to some smaller open interval).%
    \footnote{
    Note that this may destroy the original periodicity, but such properties are best studied manually in the end anyway. 
    }
    Some natural choices here include:
    \begin{itemize}[itemsep=0.1cm]
        \item The \emph{radial coordinate} $s \equiv r$, which describes the radius of the spheres. Then one views $\alpha$ as a function of $r$. This choice has the disadvantage of breaking down at points where $r$ is stationary, leading to artificial coordinate singularities, and therefore will not be preferred in this paper.
        \item The \emph{arc-length coordinate} $s$ such that $\alpha \equiv 1$, so that $(M,g)$ can be viewed as a warped product $I \times_r \Sph^2$ for some interval $I\subset \R$.
        \item The \emph{log-polar coordinate} $s$ such that $\alpha \equiv r$. This choice may appear less natural at first but is known to be convenient for studying the Yang-Mills equation in spherically symmetric contexts, so we will generally prefer this choice later in the work. For convenience, we use the notation $(N \times \Sph^2)_r$ to denote the manifold $N \times \Sph^2$ equipped with the Riemannian metric $r^2(\diff s^2 + g_{\Sph^2})$.
    \end{itemize}
\end{remark}

A natural "spherical" frame for $(M,g)$ is given by
\begin{equation}\label{eq-sph-frame}
    e_s = \frac{1}{\alpha} \, \partial_s, \qquad 
    e_\theta = \frac{1}{r} \, \partial_\theta, \qquad
    e_\varphi = \frac{1}{r\sin\theta} \, \partial_\varphi.
\end{equation}
The Levi-Civita connection symbols in this frame are given in Table \ref{tab:levi-civita}.
We can view the Levi-Civita connection also as a $\so(3)$-valued 1-form $\omega_{\scriptscriptstyle\mathrm{LC}}$ on the frame bundle $\SO(TM)$, and write
\begin{align*}
    e^\ast\omega_{\scriptscriptstyle\mathrm{LC}} &= \frac{r'}{\alpha r} (e^\theta \otimes L_2 - e^\varphi \otimes L_1)  + \frac{\cot\theta}{r} \, e^\varphi \otimes L_3  \\
    &= \frac{r'}{\alpha} \left(\diff\theta \otimes L_2 - \sin\theta\,\diff\varphi \otimes L_1\right)  + \cos\theta \, \diff\varphi \otimes L_3
\end{align*}
where prime denotes differentiation by $s$, the frame $e$ is the local section of $\SO(TM)$ given by (\ref{eq-sph-frame}) ordered as $(e_s,e_\theta,e_\varphi) = (e_3, e_1, e_2)$ (this choice of ordering is essentially due to \eqref{eq-Ad-cartesian-spherical}, which will become more apparent in the next section), and $L_i$ are the standard skew-symmetric matrices forming a basis for $\so(3)$, cf.\ \eqref{eq-so(3)-basis}.
Thus, after choosing a spin structure, the corresponding spin Levi-Civita connection will be given by
\begin{equation}\label{eq-levi-civita-form}
    \epsilon^\ast \omega_{\scriptscriptstyle\mathrm{spin}}
    = (\rho_*)^{-1} \circ e^\ast\omega_{\scriptscriptstyle\mathrm{LC}}
    = \frac{r'}{\alpha} \left(\diff\theta \otimes \tau_2 - \sin\theta\,\diff\varphi \otimes \tau_1\right)  + \cos\theta \, \diff\varphi \otimes \tau_3
\end{equation}
in a spin frame $\epsilon$ corresponding to $e$, where $\rho : \Spin(3) \cong \SU(2) \to \SO(3)$ is the two-fold covering, $\tau_i$ are the skew-Hermitian matrices forming a basis for $\su(2)$, cf.\ \eqref{eq-su(2)-basis}.

\begin{table}
    \centering
    \renewcommand{\arraystretch}{2}
    \begin{tabular}{|c||c|c|c|}
        \hline
        $\nabla_{e_j} e_k$  & $k=s$ & $k=\theta$ & $k=\varphi$ \\\hline\hline
        $j=s$ & 0 & 0 & 0  \\\hline
        $j=\theta$ & $\frac{r'}{\alpha r} e_\theta$ & $- \frac{r'}{\alpha r} e_s$ & 0  \\\hline
        $j=\varphi$ & $\frac{r'}{\alpha r} e_\varphi$ & $\frac{\cot\theta}{r} e_\varphi$ & $- \frac{r'}{\alpha r} e_s - \frac{\cot\theta}{r} e_\theta $ \\\hline
    \end{tabular}
    \vskip0.5cm
    \caption{The covariant derivatives
    $\nabla_{e_j} e_k$ for the frame (\ref{eq-sph-frame}).}
    \label{tab:levi-civita}
\end{table}

\subsection{Invariant principal bundles}

We now wish determine the $\SU(2)$-symmetric principal fiber bundles over $M = N \times (\SU(2)/\U(1))$. 
In our case, we will work with the structure group $G = \SU(2)$, which due to the isomorphism $\SU(2) \cong \Spin(3)$ also covers principal $\Spin(3)$-bundles.
In view of the discussion in \S \ref{subsec-K-symmetry}, to determine $\SU(2)$-symmetric principal $\SU(2)$-bundles, we need to classify the conjugacy classes of homomorphisms $\lambda : \U(1) \to \SU(2)$,
as well as principal $Z$-bundles $Q$ over $N$, where $Z\cong \U(1)$ is the center of $\lambda(\U(1))$.
Since $N=\R$ or $N=\Sph^1$, we see that such a bundle $Q$ is necessarily trivial, so it suffices to classify the conjugacy classes homomorphisms $\U(1)\to\SU(2)$.
In fact, it is well-known, and simple to show, that the conjugacy classes of homomorphisms $\U(1) \to \SU(2)$ are given by integral powers
\begin{equation*}
    \lambda_m : \exp(t\tau_3) \mapsto \exp(mt\tau_3)
\end{equation*}
for $m\in\Z_{\geq 0}$.
We recall from \S \ref{subsec-K-symmetry} that we can then view the corresponding bundle 
$P_m = P[\lambda_m] \to M$ as the space $N \times (\SU(2)\times\SU(2))/\U(1)$ with respect to the equivalence relation
\begin{equation*}
    (k,g) \sim \left(k\exp(t\tau_3), \exp(-mt\tau_3)g\right), \quad (k,g) \in \SU(2) \times \SU(2),
\end{equation*}
and where the projection and actions of $K$ and $G$ are naturally defined via
\begin{equation*}
    \pi(s, [k,g]) = (s, kH), \quad k_0 \cdot (s, [k, g]) = (s, [k_0k, g]), \quad (s, [k, g]) \cdot g_0 = (s, [k, gg_0]),
\end{equation*}
for $s \in N$, $k, k_0 \in K$ and $g, g_0 \in G$.
Note that the bundle $P_m$ is trivial if (and only if) $m=0$ and $m=1$, with a global trivialization defined by
\begin{equation*}\label{eq-global-trivialization-P}
    M \times \SU(2) \to P_m, \qquad ((s,kH), g) \mapsto \begin{cases}
        (s, [k, g]), & m=0,\\ 
        (s, [k, k^{-1}g]), & m=1.
    \end{cases}
\end{equation*}

\subsection{Invariant connections}
\label{sec-invariant-connections}

Next, we determine the invariant connections on $P_m$, i.e.\ we determine the pairs $(\Lambda, \eta)$ satisfying the Wang conditions \eqref{eq-wang-conditions}.
This derivation is well-known in literature but let us recall a few details.
The Wang conditions associated with the homomorphism $\lambda_m$ work out to
\begin{equation*}
    \Lambda(\tau_1) = -m[\tau_3, \Lambda(\tau_2)],\qquad
    \Lambda(\tau_2) = m[\tau_3, \Lambda(\tau_1)],\qquad
    \Lambda(\tau_3) = m\tau_3.
\end{equation*}
Writing out $\Lambda(\tau_1), \Lambda(\tau_2)$ in the basis $\tau_i$, the first two identities imply that 
\begin{equation*}
    \Lambda(\tau_1) = 
    \begin{cases}
        0, & m\not=1,\\
        v(s)\tau_1 - w(s)\tau_2, & m=1
    \end{cases},
    \qquad
    \Lambda(\tau_2) = 
    \begin{cases}
        0, & m\not=1,\\
        w(s)\tau_1 + v(s)\tau_2, & m=1
    \end{cases},
\end{equation*}
for real-valued functions $w,v$.
On the other hand, for $\eta$, we note that $\mathfrak{z} = \u(1) = \text{span}(\tau_3)$, so that $\eta = a(s)\diff s \otimes \tau_3$ for a real-valued $a$. 
It follows that a general invariant connection $\omega$ on $P_m$ is determined by three real-valued functions $a,w,v$, such that $v, w \equiv 0$ if $m\not= 1$.

Now consider the section $\sigma$ (\ref{eq-hopf-section}) of the Hopf bundle, and the corresponding section $p(s,\theta,\varphi) = (s, [\sigma(\theta,\varphi), e])$ of $P_m$, where $(\theta,\varphi)$ are the standard spherical coordinates on $\Sph^2$.
One calculates
\begin{equation*}
    \sigma^\ast \mu_{\SU(2)} = \sigma^{-1}\diff\sigma 
    = \diff\theta \otimes \tau_2 - \sin\theta\,\diff\varphi \otimes \tau_1 + \cos\theta \, \diff\varphi \otimes \tau_3,
\end{equation*}
and therefore
\begin{align*}
    \Lambda \circ \sigma^\ast \mu_{\SU(2)}
    =&\,
    v(s)(\diff\theta \otimes \tau_2 - \sin\theta\,\diff\varphi \otimes \tau_1)
    + w(s)(\diff\theta \otimes \tau_1 + \sin\theta \,\diff\varphi \otimes \tau_2) + m\cos\theta \, \diff\varphi \otimes \tau_3,
\end{align*}
where it is understood that $v,w \equiv 0$ if $m\not=1$.
Thus the connection $\omega$ can be written, with respect to the section $p$, as
\begin{align*}
    p^*\omega =&\, a(s)\, \diff s \otimes \tau_3
    + m\cos\theta \,\diff\varphi \otimes \tau_3 
    \\[0.1cm]
    &+ v(s)(\diff\theta \otimes \tau_2 - \sin\theta\, \diff\varphi \otimes \tau_1)
    + w(s)(\diff\theta \otimes \tau_1 + \sin\theta\, \diff\varphi \otimes \tau_2),
\end{align*}
for some functions $a,w,v$, such that $v, w \equiv 0$ if $m\not= 1$.


Gauge transformations of the form $p\mapsto p\exp\left(-f(s)\tau_3\right)$ do not change the overall form of the connection but send
\begin{equation*}
    a \mapsto a-f', \quad w \mapsto w\cos f - v\sin f, \quad v \mapsto w\sin f + w\cos f. 
\end{equation*}
Thus we may assume $a \equiv 0$ (i.e.\ $\eta \equiv 0$) without losing generality.%
\footnote{Alternatively, one could choose $f$ to make any of the coefficients $w,v,a$ zero, while keeping the other two as variables.}
For the rest of the work, we take $m=1$ and $a = 0$, so that the connection has the form
\begin{equation}\label{eq-connection-ansatz}
    p^*\omega = v(s)(\diff\theta \otimes \tau_2 - \sin\theta\, \diff\varphi \otimes \tau_1)
    + w(s)(\diff\theta \otimes \tau_1 + \sin\theta\, \diff\varphi \otimes \tau_2) + \cos\theta \,\diff\varphi \otimes \tau_3 .
\end{equation}
The curvature form $F_\omega$ of $\omega$ is given by (with respect to the section $p$ and standard spherical coordinates $(\theta, \varphi)$)
\begin{align*}
    F_\omega =&\,
    - (1-v^2-w^2)\, \diff\theta \wedge \sin\theta\, \diff\varphi \otimes \tau_3
    \\[0.1cm]
    &
    + w'\, \diff s\wedge (\diff\theta \otimes \tau_1 + \sin\theta\, \diff\varphi \otimes \tau_2)
    \\[0.1cm]
    &
    + v' \, \diff s \wedge (\diff\theta \otimes \tau_2 - \sin\theta\, \diff\varphi \otimes \tau_1)
\end{align*}
and the Yang-Mills operator becomes
\begin{align}
    \nonumber
    \diff_\omega^\ast F_\omega =
    & \; \frac{2}{r^2}(v'w  - vw') \,\diff s \otimes \tau_3\\[0.2cm]
    \nonumber
    &- \frac{1}{\alpha^2}\left( 
    w'' - \frac{\alpha'}{\alpha}w' + \frac{\alpha^2}{r^2}w(1-v^2-w^2) \right)(\diff\theta \otimes \tau_1 + \sin\theta\, \diff\varphi \otimes \tau_2)\\[0.2cm]
    \label{eq-ym-operator-ansatz}
    &- \frac{1}{\alpha^2}\left( 
    v'' - \frac{\alpha'}{\alpha}v' + \frac{\alpha^2}{r^2}v(1-v^2-w^2) \right)(\diff\theta \otimes \tau_2 - \sin\theta\, \diff\varphi \otimes \tau_1).
\end{align}

\subsection{Spin structure}
\label{subsec-spin-structure}

In view of \S \ref{sec-invariant-connections} and \eqref{eq-levi-civita-form}, we see that the spin Levi-Civita connection fits into the spherically symmetric framework if and only if $m=1$.
In particular, this suggests that we should utilize the principal $\SU(2) \cong \Spin(3)$-bundle $P_1$ to define the spin structure on the base manifold $M$.

To this end, we note that we can trivialize the tangent bundle $TM \cong TN \oplus T\Sph^2$ via the map 
\begin{equation*}
    S = S_{(s,kH)} : \su(2) \cong\R^3 \to T_{(s,kH)} M
\end{equation*}
given by (with respect to standard spherical coordinates $(\theta,\varphi)$)
\begin{align*}
    \sin\theta\cos\varphi \, \tau_1 + \sin\theta\sin\varphi \, \tau_2 + \cos\theta \, \tau_3 &\mapsto e_s = \frac{1}{\alpha(s)} \partial_s,
    \\
    \cos\theta\cos\varphi \, \tau_1 + \cos\theta\sin\varphi \, \tau_2 - \sin\theta \, \tau_3 &\mapsto e_\theta = \frac{1}{r(s)} \partial_\theta,
    \\
    -\sin\varphi \, \tau_1 + \cos\varphi \, \tau_2 &\mapsto e_\varphi = \frac{1}{r(s)\sin\theta} \partial_\varphi,
\end{align*}
and extended linearly.
Here, $\tau_1$, $\tau_2$, $\tau_3$ are identified with the standard unit Cartesian vectors, so that the vectors on the left-hand side correspond to the unit vectors in the standard spherical coordinates on $\R^3$.
The mapping $S$ thus identifies the spherical frame elements $e_s, e_\theta, e_\varphi$ with the standard unit vectors or $\R^3$ in spherical coordinates, which makes the identification rather natural and in particular independent of the coordinates $(\theta,\varphi)$.
The map $S$ defines a pointwise isometry $\su(2) \to T_{(s,kH)} M$ and thus induces also a global trivialization of the orthonormal frame bundle $\SO(TM)$.

By composing $E$ with the double covering $\Ad : \SU(2) \to \SO(3) \cong \SO(\su(2))$ and the global trivialization \eqref{eq-global-trivialization-P} of $P_1$, we can then define the spin structure on $M$ by setting $\Spin(TM)=P_1$ and endowing it with the equivariant double covering
\begin{equation*}
    \mathscr{S} : \Spin(TM) \to \SO(TM), \qquad [s, k, g] \mapsto ((s,kH), S_{(s,kH)} \circ \Ad_{kg}).
\end{equation*}

The natural spherical section%
\footnote{In principle, $\epsilon$ is given by the same formula as the section $p$ from \S \ref{sec-invariant-connections}, but we prefer to use different notations for the two since they should be considered as sections of separate bundles, and one could do gauge transformations on each of them separately. E.g.\ changing the orthonormal frame of the manifold would change the spin frame $\varepsilon$, but not the section $p$ of the coefficient bundle.}
$\varepsilon : (s,\theta,\varphi) \mapsto[s,\sigma(\theta,\varphi),e]$ induced by the Hopf section \eqref{eq-hopf-section} then defines a spin frame on $M$, which induces also an orthonormal frame $\mathscr{S} \circ \varepsilon$ for the tangent bundle. In fact, from the identities (\ref{eq-Ad-cartesian-spherical}), we see that $\mathscr{S} \circ \varepsilon$ is given explicitly by the (ordered) spherical frame $(e_\theta, e_\varphi, e_s)$.

If we let $\chi$ denote the spinor representation (i.e.\ the fundamental representation of $\SU(2)$), then the spinor bundle is the associated vector bundle $\Sigma M = \Spin(TM) \times_\chi \C^2$. A spinor is then a section $\Psi \in \Gamma(\Sigma M)$, and Clifford multiplication between vector fields and spinors is given by
\begin{equation*}
    X \cdot \Psi = [\epsilon, \psi] = [\epsilon,\, X^i \gamma_i \psi],
\end{equation*}
where $X^i$ are the components of $X$ with respect to the ordered frame $\mathscr{S}\circ\varepsilon$, and $\gamma_i$ are the matrices \eqref{eq-clifford-basis} representing the Clifford algebra.
In particular, if we use the spherical spin frame $\varepsilon : (s,\theta,\varphi) \mapsto[s,\sigma(\theta,\varphi),e]$ induced by the Hopf section, then 
\begin{equation*}
    e_\theta \cdot \Psi = [\varepsilon, \, \gamma_1\psi], \quad
    e_\varphi \cdot \Psi = [\varepsilon, \, \gamma_2\psi], \quad
    e_s \cdot \Psi = [\varepsilon, \, \gamma_3\psi].
\end{equation*}

\subsection{Invariant twisted spinors}
\label{subsec-twisted-spinors}

Next, we determine invariant twisted spinors.
The representations of $\SU(2)$ are well-known to be classified by integers $n\geq 0$,%
\footnote{The classifying integer $n$ should not be confused with the classifying integer $m$ for the principal bundle from the previous sections. At this point we are already fixing the latter to be $m=1$ both for the principal spin bundle and the principal $G$-bundle.}
i.e.\ there exist non-isomorphic irreducible representations $\chi_n : \SU(2) \to \U(V_n)$ with $\dim V_n = n+1$ (see below).
We define the associated bundles $E_n = P_1 \times_{\chi_n} V_n$, and we are interested in finding $\SU(2)$-invariant sections of $\Sigma M \otimes E_n$.
In the classification, $\chi_1$ corresponds to the fundamental representation of $\SU(2)$, which also coincides with the spinor representation $\kappa : \Spin(3) \to \U(\Sigma)$, so that we in fact have $\Sigma M \cong E_1$.
Thus the problem reduces to finding $n$ for which there exist invariant sections of the bundle
\begin{equation*}
    \Sigma M \otimes E_n \cong (P_1\times P_1) \times_{(\chi_1\otimes \chi_n)} (V_1\otimes V_n),
\end{equation*}
where we view $P_1 \times P_1$ as a principal $(\SU(2)\times\SU(2))$-bundle over $M$, and $\chi_1\otimes \chi_n$ is the  representation of $\SU(2) \times \SU(2)$ on $V_1 \otimes V_n$ given by
\begin{equation*}
    (\chi_1\otimes\chi_n)(g,l)(v\otimes w) = \chi_1(g)v \otimes \chi_n(l)w,
\end{equation*}
and extended linearly.
Note that $P_1 \times P_1$ is trivially an $\SU(2)$-symmetric principal bundle, and the associated classifying homomorphism $\U(1) \to \SU(2) \times \SU(2)$ is given simply by $\lambda_1 \times \lambda_1$.
Thus, invariant twisted spinors are determined by elements $\psi \in V_1 \otimes V_n$ satisfying the invariance condition \eqref{eq-section-invariance-condition}, i.e.\
\begin{equation*}
    (\chi_1\otimes\chi_n)(\lambda_1(h),\lambda_1(h))\psi = \psi, \qquad h\in H.
\end{equation*}
Infinitesimally, this means that $\psi \in \ker L_n$, where
\begin{equation}\label{eq-L_n}
    L_n = (\chi_1)_*(\tau_3) \otimes \mathrm{id} + \mathrm{id} \otimes (\chi_n)*(\tau_3).
\end{equation}
The goal is thus to find $n$ for which $\ker L_n$ is non-empty.

To this end, we need an explicit description of the representations $\chi_n$, which we briefly recall.
Set $h = 2i\tau_3$ and $\tau^{\pm} = i\tau_1 \pm \tau_2$, which form a basis of $\sl_2(\C) \cong \su(2) \otimes \C$ with $h$ spanning the Cartan subalgebra of $\sl_2(\C)$. The irreducible representations of $\sl_2(\C)$ when restricted to $\su(2)$ also integrate to yield all the irreducible representations of $\SU(2)$, so it suffices to understand $\sl_2(\C)$.
In fact, for each integer $n\geq 0$, there is a representation $\chi_n: \SL_2(\C) \to \GL(V_n)$ of dimension $n+1$ such that there is a basis $\v_0, \dots, \v_n$ of $V_n$ satisfying
\begin{itemize}
    \item $(\chi_n)_*(h)(\v_k) = (n-2k)\v_k$,
    \item $(\chi_n)_*(\tau^+)(\v_n) = (\chi_n)_*(\tau_-)(\v_0) = 0$,
    \item $(\chi_n)_*(\tau^+)(\v_k) = \v_{k+1}$ for $k < n$
    \item $(\chi_n)_*(\tau^-)(\v_k) = k(n-k+1)\v_{k-1}$ for $k > 0$,
\end{itemize}
and any irreducible representation of $\sl_2(\C)$ of dimension $n+1$ is isomorphic to $\chi_n$, see e.g.\ \cite{KnappLieGroups}*{\S I.9.}.
Here, the induced representation of $\SU(2)$ corresponds to the fundamental representation for $n=1$.
Defining a Hermitian metric on $V_n$ by
\begin{equation*}
    \langle \v_k , \v_\ell\rangle := (k!)^2\binom{n}{k}\delta_{k\ell} = \, \frac{n!k!}{(n-k)!} \delta_{k\ell}
\end{equation*}
makes the induced representation $\chi_n$ of $\SU(2)$ unitary and an orthonormal basis is given by $\u_k = \frac{1}{k!}\binom{n}{k}^{-\frac{1}{2}}\v_k = \sqrt{\frac{(n-k)!}{n!k!}}\,\v_k$.
We have $\tau_3 = -\frac{i}{2}h$, $\tau_1 = -\frac{i}{2}(\tau^+ + \tau^-)$ and $\tau_2 = \frac{1}{2}(\tau^+ - \tau^-)$, so that explicitly
\begin{align*}
    (\chi_n)_*(\tau_3)\u_k &= -\frac{i}{2}(n-2k)\u_k\\
    (\chi_n)_*(\tau_1)\u_k &= -\frac{i}{2}\left(\sqrt{k(n-k+1)}\u_{k-1} + \sqrt{(k+1)(n-k)}\u_{k+1}\right)\\
    (\chi_n)_*(\tau_2)\u_{k} &= \frac{1}{2}\left(-\sqrt{k(n-k+1)}\u_{k-1} + \sqrt{(k+1)(n-k)}\u_{k+1}\right).
\end{align*}

Now we can finally determine the kernel of $L_n$, cf.\ \eqref{eq-L_n}.
For easier distinction between the factors, we denote the orthonormal basis corresponding to $V_n$ by $\u_k^{(n)}$.
Then any $\psi \in V_1 \otimes V_n$ can be written as
\begin{equation*}
    \psi = \sum_{j=0}^1 \sum_{k=0}^n \psi_{jk} \, \u^{(1)}_j \otimes \u^{(n)}_k,
\end{equation*}
and a simple calculation shows that $\psi \in \ker L_n$ if and only if
\begin{equation*}
    \psi_{jk} (1-2j+n-2k) = 0,
\end{equation*}
for all $j,k$. Thus $\psi_{jk}=0$ unless $1-2j+n-2k=0$, and the latter is satisfied if and only if $n$ is odd and
\begin{equation*}
    (j,k) = \left(0, \,\frac{n+1}{2}\right) \quad\text{or}\quad (j,k) = \left(1, \,\frac{n-1}{2}\right).
\end{equation*}
It follows that if we denote
\begin{equation*}
    \e_1 = \u_0^{(1)}, \qquad \e_2 = \u_1^{(1)}, \qquad 
    \w_1 = \u^{(n)}_{\frac{n-1}{2}}, \qquad \w_2 = \u^{(n)}_{\frac{n+1}{2}},
\end{equation*}
for odd $n$,
then the invariant subspace of $V_1 \otimes V_n$ is spanned by $\e_1\otimes \w_2$ and $\e_2\otimes \w_1$, whereas the invariant subspace is empty for even $n$.

Thus, we see that invariant twisted spinors $\Psi = \Psi^{(n)} \in \Gamma(\Sigma M \otimes E_n)$ for odd $n$ are of the form
\begin{equation}\label{eq-spinor-ansatz}
    \Psi = \Psi^{(n)} = \left[\epsilon \times p, \, \psi_1(s) \,\e_1 \otimes \w_2 + \psi_2(s) \, \e_2 \otimes \w_1 \right].
\end{equation}
Note that gauge transformations of the form $p\mapsto p\exp\left(-f(s)\tau_3\right)$ employed in \S \ref{sec-invariant-connections} do not change the overall form of the spinor ansatz, but only send
\begin{equation*}
    \psi_1 \mapsto e^{-\frac{if}{2}} \psi_1, \qquad \psi_2 \mapsto e^{\frac{if}{2}}\psi_2,
\end{equation*}
and thus we do not lose generality by making such transformations also in the context of invariant twisted spinors.

\begin{remark}
    For $n=1$, one may be inclined to view $\Psi$ as a spinor-valued spinor, since both factors are sections of essentially the same bundle.
    While this is perfectly reasonable from a topological perspective, it is less so reasonable from a geometric one, since the factors are equipped with different connections.
    Indeed, the first factor comes equipped with the connection associated to the spin Levi-Civita connection (\ref{eq-levi-civita-form}), while the second factor is endowed with the connection associated to a connection $\omega$ that solves the Yang-Mills equation (with Dirac current). Therefore it is more appropriate to only view the first factor as a spinor in this case.
\end{remark}

\begin{table}
    \centering
    \renewcommand{\arraystretch}{2}
    \begin{tabular}{|c||c|c|c|}
        \hline
        $(\chi_n)_*(X)w$  & $X=\tau_1$ & $X=\tau_2$ & $X=\tau_3$ \\\hline\hline
        $w=\w_1$ & $-\frac{i}{2}\left(\lambda\w_{2} + \mu \u_{\frac{n-3}{2}}\right)$ & $\frac{1}{2}\left(\lambda\w_2 - \mu\u_{\frac{n-3}{2}}\right)$ & $-\frac{i}{2}\w_1$  \\\hline
        $w=\w_2$ & $-\frac{i}{2}\left(\lambda\w_1 + \mu\u_{\frac{n+3}{2}}\right)$ & $\frac{1}{2}\left(-\lambda\w_1 + \mu\u_{\frac{n+3}{2}}\right)$ & $\frac{i}{2} \w_2$  \\\hline
    \end{tabular}
    \vskip0.5cm
    \caption{The representation $(\chi_n)_*$. Here, $\lambda = \frac{n+1}{2}$ and $\mu = \frac{\sqrt{(n-1)(n+3)}}{2}$.}
    \label{tab:invariant-actions}
\end{table}

Next, we calculate the current \eqref{eq-dym-current}.
We recall that we are using the representation $\gamma_j = -i\sigma_j$ for the Clifford algebra, and correspondingly Clifford multiplication by $e_j$ acts on $\Psi$ as multiplication by $\gamma_j$ on the first factor (in a section). 
From Table \ref{tab:invariant-actions} we then compute the current in the section $\epsilon \times p$ to be
\begin{align}
    \nonumber
    \mathfrak{J}[\Psi] &= \langle e_j \cdot \Psi, \chi_*(\tau_k)\Psi\rangle \, e^j \otimes \tau_k\\[0.2cm]
    \nonumber
    &= \Big\langle \psi_1 \,(\gamma_j\e_1) \otimes \w_2 + \psi_2 \, (\gamma_j \e_2) \otimes \w_2,\; \psi_1\, \e_1 \otimes (\chi_*(\tau_k)\w_2) + \psi_2\, \e_2 \otimes (\chi_*(\tau_k)\w_1) \Big\rangle \, e^j \otimes \tau_k \\[0.2cm]
    %
    \nonumber
    &=-\frac12 \alpha(|\psi_1|^2 + |\psi_2|^2) \, \diff s \otimes \tau_3 \\[0.2cm]
    \label{eq-current-ansatz}
    &\quad + \lambda r\Re(\psi_1\conj{\psi}_2)(\diff \theta \otimes\tau_1 + \sin \theta \diff \varphi \otimes \tau_2)  + \lambda r\Im(\psi_1\conj{\psi}_2)(\diff \theta \otimes \tau_2- \sin\theta \diff \varphi \otimes \tau_1). 
\end{align}

For the Dirac operator we first calculate the derivatives of $\Psi$ with respect to the connection on $\Sigma M \otimes E_n$ induced by the spin Levi-Civita connection and $\omega$.
In the section $\epsilon \times p$, they work out to be
\begin{align*}
    (\nabla_\omega \Psi)(e_\theta) &= -\frac{1}{2r}\left(\frac{r'}{\alpha}\psi_2 + \lambda(v+iw)\psi_1 \right)\e_1\otimes \w_1 + \frac{1}{2r}\left(\frac{r'}{\alpha}\psi_1 + \lambda(v-iw)\psi_2\right)\e_2 \otimes \w_2\\[0.1cm]
    &\quad + \frac{\mu}{2r}(v-iw)\psi_1 \, \e_1\otimes \u_{\frac{n+3}{2}} - \frac{\mu}{2r}(v+iw)\psi_2 \, \e_2 \otimes \u_{\frac{n-3}{2}},\\[0.2cm]
    (\nabla_\omega \Psi)(e_\varphi) &= \frac{i}{2r}\left(\frac{r'}{\alpha}\psi_2 + \lambda(v+iw)\psi_1\right)\e_1 \otimes \w_1 + \frac{i}{2r}\left(\frac{r'}{\alpha}\psi_1 + \lambda(v-iw)\psi_2\right)\e_2\otimes\w_2\\[0.1cm]
    &\quad +\frac{\mu}{2r}(iv+w)\psi_1 \, \e_1 \otimes \u_{\frac{n+3}{2}} + \frac{\mu}{2r}(iv - w)\psi_2 \, \e_2 \otimes \u_{\frac{n-3}{2}},\\[0.2cm]
    (\nabla_\omega \Psi)(e_s) &= \frac{\psi'_1}{\alpha} \, \e_1\otimes \w_2 + \frac{\psi'_2}{\alpha} \, \e_2\otimes \w_1.
\end{align*}
The twisted Dirac operator on invariant twisted spinors is then
\begin{align}
    \nonumber
    \dirac_\omega\Psi 
    =&\,
    \gamma^k\, (\nabla_\omega \Psi)(e_k)
    \\[0.2cm]
    \nonumber
    =&\, -\frac{i}{\alpha}\left( \psi_1' + \frac{r'}{r}\psi_1 + \frac{\alpha}{r}\,\lambda(v-iw)\psi_2 \right) \e_1 \otimes \w_2\\[0.1cm]
    \label{eq-dirac-ansatz}
    &\, + \frac{i}{\alpha}\left( \psi_2' + \frac{r'}{r}\psi_2 + \frac{\alpha}{r}\,\lambda(v+iw)\psi_1 \right) \e_2 \otimes \w_1,
\end{align}
with respect to the section $\epsilon \times p$.

\section{Dirac-Yang-Mills equations}

\subsection{Summary of the ansatz}
We are now ready to write down the Dirac-Yang-Mills system \eqref{eq-dym} in spherical symmetry.
For the convenience of the reader, we briefly summarize the ansatz from the previous section.
We work with three-dimensional Riemannian manifolds $(M,g)$ of the form
\begin{equation*}
    M = N \times \Sph^2, \qquad g = \alpha(s)^2\diff s^2 + r(s)^2 g_{\Sph^2},
\end{equation*}
where $N$ is a one-dimensional manifold and $\alpha, r : N \to \R$ are positive functions. 

We set $\Spin(TM) = P_1$ and $P = P_1$ so that we get a total principal bundle 
\begin{equation*}
    \SU(2) \times \SU(2) \to \Spin(TM) \times P \to M.
\end{equation*}
The Hopf section \eqref{eq-hopf-section} with respect to standard spherical coordinates induces natural sections $\epsilon$ of $\Spin(TM)$ and $p$ of $P$.
The spin frame $\epsilon$ is then related to the natural spherical frame $(e_\theta, e_\varphi, e_s)$, cf.\ \eqref{eq-sph-frame}.

The total principal bundle $\Spin(TM) \times P$ is equipped with the connection 
\begin{equation*}
    \omega_{\mathrm{spin}} \oplus \omega \in \Omega^1(\Spin(TM) \times P, \, \su(2)\oplus\su(2)),    
\end{equation*}
where $\omega_{\mathrm{spin}}$ is the spin Levi-Civita connection, and $\omega$ is a spherically symmetric connection on $P$.
With respect to the section $\epsilon \times p$, we have $(\epsilon\times p)^*(\omega_{\mathrm{spin}} \oplus \omega) = \epsilon^*\omega_{\mathrm{spin}} \oplus p^*\omega$, where
\begin{align*}
    \epsilon^*\omega_{\mathrm{spin}} &= \frac{r'}{\alpha} \left(\diff\theta \otimes \tau_2 - \sin\theta\,\diff\varphi \otimes \tau_1\right)  + \cos\theta \, \diff\varphi \otimes \tau_3\\[0.1cm]
    p^*\omega &= v(s)(\diff\theta \otimes \tau_2 - \sin\theta\, \diff\varphi \otimes \tau_1)
    + w(s)(\diff\theta \otimes \tau_1 + \sin\theta\, \diff\varphi \otimes \tau_2) + \cos\theta \,\diff\varphi \otimes \tau_3,
\end{align*}
for two real-valued unknown functions $w,v : N \to \R$.

Twisted spinors $\Psi \in \Gamma(\Sigma M \otimes E_n)$ are sections of the associated bundle
\begin{equation*}
    \Sigma M \otimes E_n = (\Spin(TM) \times P) \times_{\chi_1 \otimes \chi_n} (V_1 \otimes V_n),
\end{equation*}
where $\chi_n : \SU(2) \to \U(V_n)$ is the $n$-representation of $\SU(2)$.
With respect to the section $\epsilon\times p$, spherically symmetric twisted spinors are given (for odd $n$) by
\begin{equation*}
    \Psi = \Psi^{(n)} = \left[\epsilon \times p, \, \psi_1(s) \,\e_1 \otimes \w_2 + \psi_2(s) \, \e_2 \otimes \w_1 \right],
\end{equation*}
where $\e_i \in V_1$, $\w_j \in V_n$ are appropriately chosen vectors, and $\psi_1,\psi_2 : N \to \C$ are complex-valued unknown functions.

\subsection{The ODE system}
\label{subsec-dym-ode-system}

Using (\ref{eq-ym-operator-ansatz}, \ref{eq-current-ansatz}, \ref{eq-dirac-ansatz}), we see that $(\omega, \Psi)$ forms a Dirac-Yang-Mills pair on $(M,g)$ if and only if
\begin{subequations}
\begin{empheq}[left=\empheqlbrace]{align}
    \label{eq-constraint-vw}
    0&= v'w-vw'+\frac{\alpha r^2}{4}(|\psi_1|^2 + |\psi_2|^2),
    \\[0.1cm]
    \label{eq-ym-w}
    0&= w'' - \frac{\alpha'}{\alpha}w' + \frac{\alpha^2}{r^2}w(1-v^2-w^2) + \lambda \alpha^2 r \,\Re(\psi_1\conj{\psi}_2),
    \\[0.1cm]
    \label{eq-ym-v}
    0&= v'' - \frac{\alpha'}{\alpha}v' + \frac{\alpha^2}{r^2}v(1-v^2-w^2) + \lambda \alpha^2 r \,\Im(\psi_1\conj{\psi}_2),
    \\[0.1cm]
    \label{eq-psi-1}
    0&= \psi_1' + \frac{r'}{r}\psi_1 + \frac{\alpha}{r}\,\lambda (v-iw)\psi_2,
    \\[0.1cm]
    \label{eq-psi-2}
    0&= \psi_2' + \frac{r'}{r}\psi_2 + \frac{\alpha}{r}\,\lambda(v+iw)\psi_1,
\end{empheq}
\end{subequations}
where $\lambda = \frac{n+1}{2}$ for a positive odd integer $n$.

We choose the log-polar coordinate $s$ so that $\alpha = r$ (cf.\ Remark \ref{remark-s-coordinate}), which simplifies the appearance of the system considerably.
It is also more convenient to write down the sytem by defining the complex/quaternionic variables (cf.\ also Footnote \footref{foot-quaternions} on p.\ \pageref{foot-quaternions})
\begin{equation*}
    z = w+iv, \qquad Y = \frac{z'}{r}, \qquad \xi = r(\psi_1+j\conj{\psi}_2) =: \QuatRe(\xi) + j\QuatIm(\xi).
\end{equation*}
Note that then one can write
\begin{equation*}
    r^2\psi_1\conj{\psi}_2 = \QuatRe(\xi) \QuatIm(\xi) = \frac{1}{4}(\xi - i\xi i)j( \conj{\xi} + i\conj{\xi} i),
\end{equation*}
and the equations become
\begin{subequations}
    \begin{empheq}[left=\empheqlbrace]{align}
        \label{eq-constraint}
        0&= \Im(\conj{z}Y) + \frac14 |\xi|^2,\\
        \label{eq-z}
        0&= z' - r Y,\\
        \label{eq-Y}
        0&= Y' + \tfrac{1}{r} z(1-|z|^2) + \lambda \QuatRe(\xi) \QuatIm(\xi),\\
        \label{eq-xi}
        0&= \xi' +  i\lambda \xi \conj{z} j.
    \end{empheq}
\end{subequations}
For the geometric setting, $\lambda$ is a positive integer since $n\geq0$ is odd, although in principle the analysis carries through if $\lambda \geq 1$ is real.

\subsection{Invariances, constraints, and initial data}
\label{subsec-initial-invariances}

Note that the equations (\ref{eq-constraint}--\ref{eq-xi}) have:
\begin{itemize}
    \item the conjugate transposition symmetry
    \begin{equation*}
        (\psi_1,\conj{\psi}_2) \mapsto (\conj{\psi}_2, \psi_1), \quad\text{i.e.}\quad 
        (\QuatRe(\xi), \QuatIm(\xi)) \mapsto (\QuatIm(\xi), \QuatRe(\xi)),
    \end{equation*}
    \item the $\U(1)$-symmetry
    \begin{equation}\label{eq-u(1)-transf}
        (z,Y,\xi) \mapsto (e^{iT}z, e^{iT}Y, \xi e^{\frac{iT}{2}}) ,
    \end{equation}
    for any fixed $T\in\R$.
    Zooming out, this is just a consequence of the gauge invariance of the theory and a gauge transformation by $\exp(T\tau_3)$.
\end{itemize}

\begin{remark}
    In pure Yang-Mills theory (i.e.\ when $\xi=0$), the constraint equation \eqref{eq-constraint} implies that $w=\Re(z)$ and $v=\Im(z)$ are constant multiples of one another, and then one can use the $\U(1)$-invariance to set either $v=0$ or $w=0$. This is often taken for granted in $\SU(2)$ Yang-Mills literature, but the Dirac-Yang-Mills setting has the interesting feature that one really needs both $v$ and $w$ to not be identically zero if one wants to have non-trivial solutions.
\end{remark}

Next, we observe the following.

\begin{lemma}
    If the differential equations {\normalfont (\ref{eq-z}--\ref{eq-xi})} are satisfied and {\normalfont \eqref{eq-constraint}} holds at one point, then {\normalfont \eqref{eq-constraint}} holds everywhere.
\end{lemma}

\begin{proof}
    A calculation gives
    \begin{equation*}
        \frac{\diff}{\diff s}(\conj{z}Y)
        =
        r|Y|^2 + \frac{1}{r}|z|^2(|z|^2-1) - \lambda \conj{z} \QuatRe(\xi) \QuatIm(\xi),
    \end{equation*}
    and
    \begin{equation*}
        \frac{\diff}{\diff s}|\xi|^2 
        = -\lambda jz \conj{\xi} i\xi - \lambda \conj{\xi} i \xi \conj{z}j 
        = 4\lambda \Im(\conj{z}  \QuatRe(\xi) \QuatIm(\xi)).
    \end{equation*}
    Thus the function $\Im(\conj{z}Y)+\frac14|\xi|^2$ has vanishing derivative and is consequently constantly equal to zero by assumption.
\end{proof}

Therefore we can view \eqref{eq-constraint} as an initial data constraint.
In particular we are interested in solving (\ref{eq-z}--\ref{eq-xi}) with initial data in the space
\begin{equation*}
    \{(z_0, Y_0, \xi_0) \in \C \times \C \times \Quat \, \mid \, \Im(\conj{z}_0 Y_0) + \tfrac14|\xi_0|^2 = 0 \}.
\end{equation*}

Using the $\U(1)$-invariance of the equations, we can without loss of generality assume that $\Im (z_0) = 0$ and $\Re (z_0) > 0$. Then the constraint reduces to $\rho_0 \Im(Y_0) + \tfrac14|\xi_0|^2 = 0$ which fixes the imaginary part of $Y_0$, whereas the real part $\Re(Y_0) =: U_0$ can be chosen freely.
Using the invariance $(\QuatRe(\xi), \QuatIm(\xi))\mapsto(\QuatIm(\xi), \QuatRe(\xi))$, we can without loss of generality assume $|\QuatRe(\xi_0)| \geq |\QuatIm(\xi_0)|$.
Finally, using the $\U(1)$-invariance once again, we can assume without loss of generality that $|\arg\QuatRe(\xi_0) + \arg\QuatIm(\xi_0)| < \pi$.
Thus we can write general initial data as
\begin{equation*}
    z(0) = \rho_0, \quad Y(0) = U_0 - \frac{i|\xi_0|^2}{4\rho_0}, \quad \xi(0) = \xi_0,
\end{equation*}
where the initial parameters $(\rho_0, U_0, \xi_0)$ belong to the set
\begin{equation*}
    \mathscr{I}_0 = \left\{(\rho_0, U_0, \xi_0) \in (0,\infty)\times \R \times \Quat, \;\mid\;  |\QuatRe(\xi_0)| \geq |\QuatIm(\xi_0)|, \; -\pi < \arg\QuatRe(\xi_0) + \arg\QuatIm(\xi_0) \leq \pi \right\}.
\end{equation*}
We note that the case $\rho_0 = 0$ is excluded here since then $\xi_0 = 0$ and the theory degenerates to a pure Yang-Mills theory. 

Finally, we would like to note that we could assume without loss of generality that $|\xi_0|=1$ in view of the scale-invariance
\begin{equation*}
    Y \mapsto a^2Y, \qquad \xi \mapsto a\xi, \qquad r \mapsto r/a^2, \qquad a\in\R.
\end{equation*}
However, this clearly requires the rescaling of the metric coefficient $r$, which should be given as a fixed external parameter, so we do not exploit this scale-invariance to avoid confusion.

\subsection{Polar form}
\label{subsec-polar}

The appearance of the norms $|z|$ and $|\xi|$ in the equations (\ref{eq-constraint}--\ref{eq-xi}) suggests that it may be beneficial to write the complex variables in polar form. 
To this end, we let
\begin{equation*}
    z = w+iv = \rho e^{iZ}, \qquad \QuatRe(\xi) = r\psi_1 = R_1 e^{iX_1}, \qquad \overline{\QuatIm(\xi)} = r\psi_2 = R_2 e^{iX_2}.
\end{equation*}
As noted in \S \ref{subsec-initial-invariances}, we can assume $z_0 \not= 0$ so that the polar form for $z$ is well-defined near the initial data.
On the other hand, there exist distinguished solutions for which one of $\QuatRe(\xi)$ or $\QuatIm(\xi)$ vanishes at a point, while the other is non-zero.%
\footnote{We do not consider such solutions in this paper.}
Since we are assuming $|\QuatRe(\xi_0)| \geq |\QuatIm(\xi_0)|$, it suffices to assume that $\QuatIm(\xi_0) \not=0$ for the polar form of $\xi$ to be well-defined locally near the initial data.
In terms of the initial data $(\rho_0, U_0, \xi_0) \in \mathscr{I}_0$ from \S \ref{subsec-initial-invariances}, the initial data for the polar variables are then
\begin{equation*}
    \begin{array}{lll}
        \rho(0) = \rho_0, & R_1(0) = |\QuatRe(\xi_0)|, & R_2(0) = |\QuatIm(\xi_0)|,\\[0.2cm]
        Z(0)=0, & X_1(0) = \arg\QuatRe(\xi_0), & X_2(0) = -\arg\QuatIm(\xi_0).
    \end{array}
\end{equation*}

In terms of the polar variables, we can express
\begin{equation*}
    Y = \frac{z'}{r} = \left(\frac{\rho'}{\rho} + iZ'\right)\frac{z}{r},
\end{equation*}
so that the constraint equation \eqref{eq-constraint} reads
\begin{equation}
\label{eq-Z-general}
    \frac{\rho^2}{r} Z' + \frac{1}{4}(R_1^2 + R_2^2) = 0,
    \quad\text{i.e.}\quad
    Z(s) = -\frac14 \int_0^s \frac{r(R_1^2 + R_2^2)}{\rho^2},
\end{equation}
since we are assuming $Z_0 = 0$.

The Yang-Mills equation \eqref{eq-Y} is
\begin{equation*}
    0 = \left[ \frac{\rho''}{\rho}-\frac{r'\rho'}{r\rho} + 1 - \rho^2  - (Z')^2 + i \left(Z'' + \left(\frac{2\rho'}{\rho} - \frac{r'}{r}\right)Z'  \right)\right]\frac{\rho}{r} e^{iZ} + \lambda R_1R_2e^{i(X_1-X_2)},
\end{equation*}
or equivalently
\begin{align}
    \nonumber
    \lambda e^{i(X_1-X_2-Z)} &= -\frac{\rho}{rR_1R_2} \left[ \frac{\rho''}{\rho}-\frac{r'\rho'}{r\rho} + 1 - \rho^2 - (Z')^2 + i \left(Z'' + \left(\frac{2\rho'}{\rho} - \frac{r'}{r}\right)Z'  \right)\right]
    \\[0.2cm]
    \label{eq-ym-polar-general}
    &= -\frac{\rho}{rR_1R_2} \left[ \frac{\rho''}{\rho}-\frac{r'\rho'}{r\rho} + 1 - \rho^2  - \frac{r^2(R_1^2+R_2^2)^2}{16\rho^4} \right] + \frac{i}{2\rho} \left( \frac{R_1'}{R_2} + \frac{R_2'}{R_1} \right),
\end{align}
where we have also inserted \eqref{eq-Z-general} in the last line.
The Dirac equation \eqref{eq-xi} reduces to the system
\begin{equation*}
    \begin{cases}
    0 = \left(\frac{R_1'}{R_1} + iX_1'\right)R_1e^{iX_1} - i\lambda\rho R_2 e^{i(Z+X_2)}\\[0.2cm]
    0 = \left(\frac{R_2'}{R_2} + iX_2'\right)R_2e^{iX_2} + i\lambda\rho R_1 e^{i(X_1-Z)}
    \end{cases}
\end{equation*}
or equivalently
\begin{equation}\label{eq-dirac-polar-general}
    \lambda e^{i(X_1-X_2-Z)} = \frac{R_1}{\rho R_2} \left(X_1' + i\,\frac{R_1'}{R_1}\right) = \frac{R_2}{\rho R_1} \left(- X_2' + i\,\frac{R_2'}{R_2}\right).
\end{equation}
The Dirac-Yang-Mills system thus reduces to (\ref{eq-ym-polar-general}, \ref{eq-dirac-polar-general}), which can equivalently be rewritten as the following set of equalities:
\begin{equation*}
    \begin{cases}
        \lambda \rho \sin(X_1-X_2-Z) = \frac{R_1'}{R_2} = \frac{R_2'}{R_1} = \frac{1}{2}\left(\frac{R_1'}{R_2} + \frac{R_2'}{R_1}\right),
        \\[0.2cm]
        \lambda\rho \cos(X_1-X_2-Z) = \frac{R_1X_1'}{R_2} = -\frac{R_2 X_2'}{R_1} = -\frac{\rho}{rR_1R_2} \left[ \rho'' -\frac{r'}{r}\,\rho' + \rho(1 - \rho^2)\right]  + \frac{r(R_1^2+R_2^2)^2}{16\rho^2 R_1R_2},
    \end{cases}
\end{equation*}
In fact, the last equality in the first line is implied by the other two.

Setting
\begin{equation*}
    W = X_1 - X_2 - Z,
\end{equation*}
we see that
\begin{equation*}
    W' = \frac{R_1^2+R_2^2}{R_1R_2} \, \lambda \rho \cos W + \frac{r(R_1^2+R_2^2)}{4\rho^2}.
\end{equation*}
Setting also $\rho' = rH$ we thus see that the system reduces to
\begin{subequations}
    \begin{empheq}[left=\empheqlbrace]{align}
        \nonumber
        \rho' - rH = 0,
        \\[0.1cm]
        \nonumber
        H' + \frac{1}{r}\rho(1-\rho^2) - \frac{r(R_1^2+R_2^2)^2}{16 \rho^3} + \lambda R_1R_2 \cos W = 0,
        \\[0.1cm]
        \nonumber
        R_1' - \lambda \rho R_2 \sin W = 0,
        \\[0.1cm]
        \nonumber
        R_2' - \lambda \rho R_1\sin W = 0,
        \\[0.1cm]
        \nonumber
        W' - \frac{R_1^2+R_2^2}{R_1R_2} \, \lambda \rho \cos W - \frac{r(R_1^2+R_2^2)}{4\rho^2} = 0.
    \end{empheq}
\end{subequations}
The arguments $X_1$, $X_2$ can then be recovered via
\begin{equation}\label{eq-X_12-sol}
    X_1(s) = X_1(0) + \lambda \int_0^s \frac{R_2}{R_1} \rho \cos W, \qquad 
    X_2(s) = X_2(0) - \lambda \int_0^s \frac{R_1}{R_2} \rho \cos W.
\end{equation}

Now note that the equations for $R_1$ and $R_2$ imply that
\begin{equation*}
    (R_1^2)' = (R_2^2)', \quad\text{i.e.}\quad R_1^2 - R_2^2 \equiv \text{const} = |\QuatRe(\xi_0)|^2 - |\QuatIm(\xi_0)|^2 =: \delta_0^2,
\end{equation*}
so that it in fact suffices to consider only one of the variables $R_1$ and $R_2$.
Then it is natural to consider separately the cases $\delta_0 = 0$ and $\delta_0 > 0$.

\begin{itemize}[itemsep=0.1cm]
    \item The initial datum $\xi_0$ is such that $\delta_0 = 0$ if and only if $\QuatIm(\xi_0) = \overline{\QuatRe(\xi_0)} e^{iW_0}$.
    In this case it is convenient to denote $R := R_1 = R_2$ to simplify the equations.
    Note that then
    \begin{equation*}
        R(0)^2 = |\QuatRe(\xi_0)|^2 = |\QuatIm(\xi_0)|^2 = \frac12 |\xi_0|^2.
    \end{equation*}
    We see in this case also that
    \begin{equation*}
        X_1 + X_2 \equiv \text{const} = \arg\QuatRe(\xi_0) - \arg\QuatIm(\xi_0),
    \end{equation*}
    and in particular if $\arg\QuatRe(\xi_0) = \arg\QuatIm(\xi_0)$, then the solution is such that $\QuatRe(\xi) \equiv \QuatIm(\xi)$.
    \item If $\delta_0 > 0$, we set $R := R_2$, so that $R_1 = \sqrt{\delta_0^2+R^2}$. 
    Here it is convenient to further substitute $R = \delta_0 \sinh \frac{P}{2}$, so that $R_1 = \delta_0 \cosh \frac{P}{2}$.
    Note that then
    \begin{equation*}
        P(0) = 2\arsinh\frac{|\QuatIm(\xi_0)|}{\delta_0} = 2\arcosh\frac{|\QuatRe(\xi_0)|}{\delta_0}.
    \end{equation*}
\end{itemize}

Thus, we have proven the following.

\begin{lemma}[Polar form]
\label{lem-polar-form}
Let $(\rho_0, U_0, \xi_0) \in \mathscr{I}_0$ with $\QuatIm(\xi_0) \not= 0$ and put
\begin{align*}
    W_0 &= \arg\QuatRe(\xi_0) + \arg\QuatIm(\xi_0),\\
    \delta_0 &= \sqrt{|\QuatRe(\xi_0)|^2-|\QuatIm(\xi_0)|^2}.
\end{align*}
Then the Dirac-Yang-Mills pair $(z,\xi)$ with initial parameters $(\rho_0, U_0, \xi_0)$ and with respect to the metric $r^2(\diff s^2 + g_{\Sph^2})$ is given by the following:
\begin{itemize}[itemsep=0.1cm]
    \item If $\delta_0 = 0$, then necessarily $\QuatIm(\xi_0) = \overline{\QuatRe(\xi_0)}e^{iW_0}$, and
    \small
    \begin{align*}
        z(s) &= \rho(s)\exp\left( i\left[2\lambda \int_0^s \rho(t)\cos W(t) \,\diff t - W(s) + W_0 \right] \right),\\
        \xi(s) &= \xi_0 \exp\left( i\lambda \int_0^s \rho(t)e^{-iW(t)} \diff t \right) = \sqrt{2}\,\frac{\xi_0}{|\xi_0|} \,R(s) \exp\left(i\lambda \int_0^s \rho(t)\cos W(t) \diff t\right),
    \end{align*}
    \normalsize
    where $\rho$, $R$, and $W$ are real-valued functions satisfying the system
    \begin{subequations}
        \begin{empheq}[left=\empheqlbrace]{align}
            \label{eq-rho-polar-delta=0}
            \rho' - rH = 0,
            \\[0.1cm]
            \label{eq-H-polar-delta=0}
            H' + \frac{1}{r}\rho(1-\rho^2) - \frac{rR^4}{4\rho^3} + \lambda R^2\cos W = 0, 
            \\[0.1cm]
            \label{eq-R-polar-delta=0}
            R' - \lambda \rho R\sin W = 0, 
            \\[0.1cm]
            \label{eq-W-polar-delta=0}
            W'- 2\lambda\rho\cos W - \frac{rR^2}{2\rho^2} = 0, 
        \end{empheq}
    \end{subequations}
    with initial conditions
    \begin{equation*}
        \begin{array}{ll}
             \rho(0)= \rho_0, & R(0) = \frac{1}{\sqrt{2}}|\xi_0| = |\QuatRe(\xi_0)|, \\[0.2cm]
             H(0)= \rho_0U_0, & W(0) = W_0.
        \end{array}
    \end{equation*}
    \item If $\delta_0 > 0$, then
    \small
    \begin{align*}
        z(s) &= \rho(s)\exp\left( i\left[2\lambda \int_0^s \rho(t)\coth P(t) \cos W(t)\,\diff t - W(s) + W_0 \right] \right),\\
        \QuatRe(\xi)(s) &= \delta_0 \cosh \frac{P(s)}{2} \,\exp\left( i\left[\lambda \int_0^s \rho(t)\tanh \frac{P(t)}{2} \cos W(t) \,\diff t + \arg\QuatRe(\xi_0)\right] \right),\\
        \QuatIm(\xi)(s) &= \delta_0 \sinh \frac{P(s)}{2} \,\exp\left( i\left[\lambda \int_0^s \rho(t) \coth \frac{P(t)}{2} \cos W(t) \,\diff t + \arg\QuatIm(\xi_0)\right] \right),
    \end{align*}
    \normalsize
    where $\rho$, $P$, and $W$ are real-valued functions satisfying the initial value problem
    \begin{subequations}
    \begin{empheq}[left=\empheqlbrace]{align}
        \label{eq-rho-polar-delta>0}
        \rho' - rH = 0,
        \\[0.1cm]
        \label{eq-H-polar-delta>0}
        \frac{1}{\delta_0^2} H' + \frac{1}{\delta_0^2 r}\rho(1-\rho^2) - \frac{\delta_0^2 r\cosh^2 P}{16 \rho^3} + \frac{\lambda}{2} \sinh P \cos W = 0,
        \\[0.1cm]
        \label{eq-P-polar-delta>0}
        P' - 2\lambda \rho \sin W = 0,
        \\[0.1cm]
        \label{eq-W-polar-delta>0}
        W' - 2\lambda \rho \coth P \cos W - \frac{\delta_0^2 r\cosh P}{4\rho^2} = 0,
        \end{empheq}
    \end{subequations}
    with initial conditions
    \begin{equation*}
        \begin{array}{ll}
             \rho(0)= \rho_0, & P(0) = 2\arsinh\frac{|\QuatIm(\xi_0)|}{\delta_0}, \\[0.2cm]
             H(0)= \rho_0U_0, & W(0) = W_0.
        \end{array}
    \end{equation*}
\end{itemize}
\end{lemma}

\begin{remark}
    At first glance this may seem like it makes the equations more complicated, but the point of this result is that it "partially solves" the problem, in the sense that we reduce the complex/quaternionic system (\ref{eq-z-intro}--\ref{eq-xi-intro}) to a system of four real-valued equations.
    
    In the case $\delta_0>0$ one can rescale $r \mapsto r/\delta_0^2$ and $H \mapsto \delta_0^2 H$ in order to simplify the equations (i.e.\ this sets $\delta_0=1$ in the equations).
    However, this would make the statement somewhat more confusing, since it changes the metric which should be viewed as a fixed parameter, and also the relation to the original variables changes since $\xi$ depends on $r$ implicitly, cf.\ also the discussion at the end of \S \ref{subsec-initial-invariances}. Therefore we keep $\delta_0$ explicit in the statements.
\end{remark}

\section{Solutions with constant \texorpdfstring{$\rho$}{rho}}
\label{sec-constant-rho}

One of the difficulties with studying the Dirac-Yang-Mills (DYM) system is the fact that it is non-autonomous. Indeed, the metric coefficient $r$ should be viewed as a fixed externally given coefficient. This makes it difficult to study the system in general since, a priori, there are no constraints on how the function $r$ should look like.

In this section we flip the problem by instead fixing $\rho \equiv \rho_0$, which makes the system simple enough to study or even solve explicitly in certain cases. In doing so, we of course lose the additional degree of freedom to choose $r$, so that $r$ will in fact be determined by the other variables. On the other had, this allows us to reinsert the formula for $r$ into the system, which makes it autonomous and in particular allows for a complete analysis.

Throughout the section, we assume that $\rho \equiv \rho_0 > 0$ is constant.
In view of Lemma \ref{lem-polar-form}, we study the cases $\delta_0 = 0$ and $\delta_0 > 0$ separately.

\subsection{\texorpdfstring{The case $\delta_0 = 0$}{The case delta0 = 0}}
\label{subsec-constant-rho-delta0=0}

In this case, we show the following.

\begin{proposition}[The case $\delta_0 = 0$]
    \label{prop-constant-rho-delta=0}
    Using the notation of Proposition \ref{lem-polar-form}, assume that $\delta_0=0$ so that $\xi_0 = \QuatRe(\xi_0)(1+je^{iW_0})$, and $\rho \equiv \rho_0$ for a constant $\rho_0 > 0$, and define $\rho_\crit = (1+8\lambda^2)^{-\frac12}$.
    \smallskip
    \begin{enumerate}
        \item If $\rho_0 > 1$, the DYM pair is singular at some finite $s$.
        \item If $\rho_0 = 1$, then we have $W_0=0$ without loss of generality,%
        \footnote{Here, and throughout the rest of the statement, without loss of generality means up to translations of $s$, reflections $(s,W) \mapsto -(s,W)$, and translations $W \mapsto W+2\pi k$ for $k\in \Z$.}
        and the associated DYM pair is
        \begin{align*}
           z(s) &= \exp\left(-\frac{i}{2}\arctan\sinh(4\lambda s)\right),\\
           \xi(s) &= \xi_0 \cosh^{\frac14}(4\lambda s)  \exp\left(\frac{i}{4}\arctan\sinh(4\lambda s) \right),
        \end{align*}
        and is defined globally on $(\R \times \Sph^2)_r$ with
        \begin{equation*}
            r(s) = \frac{16\lambda}{|\xi_0|^2} \sech^{\frac32}(4\lambda s).
        \end{equation*}
        \item If $\rho_\crit < \rho_0 < 1$ then we have $W_0 \in \{0, W_\infty, \pi\}$ without loss of generality, where
        \begin{equation*}
            W_\infty = \arccos\left(-\frac{1}{2\sqrt{2}\lambda}\sqrt{\frac{1}{\rho_0^2}-1}\right).
        \end{equation*}
        The DYM pair is globally defined on $(\R\times\Sph^2)_r$ for all choices of $W_0$, and $z$ oscillates asymptotically.
        \smallskip
        \begin{itemize}[itemsep=0.1cm]
            \item If $W_0 = W_\infty$, the solution is explicitly
            \begin{align*}
                z(s) &= \rho_0 \exp\left(-i\left(\frac{s}{\sqrt{2}}\sqrt{1-\rho_0^2} + W_\infty\right)\right),\\
                \xi(s) &= \xi_0 \exp\left(\frac{s}{2\sqrt{2}} \sqrt{\frac{\rho_0^2}{\rho_\crit^2}-1} \right) \exp\left(-\frac{is}{2\sqrt{2}} \sqrt{1-\rho_0^2}\right),\\
                r(s) &= \frac{2\rho_0^2\sqrt{2(1-\rho_0^2)}}{|\xi_0|^2} \, \exp\left(-\frac{s}{\sqrt{2}}\sqrt{\frac{\rho_0^2}{\rho_\crit^2}-1}\right). 
            \end{align*}
            \item If $W_0 = 0$, the solution is such that $\xi \to \infty$ and $r \to 0$ as $s\to\pm\infty$.
            \item If $W_0 = \pi$, the solution is such that $\xi \to 0$ and $r \to \infty$ as $s\to\pm\infty$.
        \end{itemize}
        \item If $\rho_0 = \rho_\crit$, then we have $W_0 \in \{0, \pi\}$ without loss of generality.
        \smallskip
        \begin{itemize}[itemsep=0.1cm]
            \item If $W_0 = 0$, the solution is such that $\xi \to \infty$ and $r \to 0$ as $s\to\pm\infty$.
            \item If $W_0 = \pi$, if we define the coordinate $t = \lambda\rho_\crit s$, then
            \begin{equation*}
                (z,\xi)(t) = (\rho_\crit e^{-2it}, \, \xi_0 e^{-it}),
            \end{equation*}
            with respect to the parametrization $t\mapsto e^{it} \in \Sph^1$ by arc length, and $(z,\xi)$ defines a DYM pair on
            $\Sph^1(8\rho_\crit^2/|\xi_0|^2) \times \Sph^2(8\lambda \rho_\crit^3/|\xi_0|^2)$, where the arguments denote the radii of the spheres.
        \end{itemize}
        \item If $\rho_0 < \rho_\crit$, we have $W_0=0$ without loss of generality.
        The associated DYM pair is globally defined on $(\R \times \Sph^2)_r$ with $(z,\xi)$ and $r$ bounded.
        Furthermore, for countably many choices of $\rho_0$, the solution is periodic and thus induces a DYM pair on $(\Sph^1 \times \Sph^2)_r$.
    \end{enumerate}
\end{proposition}

\begin{remark}
    In particular, the DYM pairs with $\delta_0 = 0$ and constant $\rho \equiv \rho_0 \leq 1$ are all globally defined on $(\R \times \Sph^2)_r$, and there exists a countably infinite set of values $\rho_0 \leq \rho_\crit$ for which the DYM pairs descend to solutions on $(\Sph^1\times\Sph^2)_r$, with $r$ being constant if and only if $\rho_0 = \rho_\crit$.
\end{remark}

In the rest of the subsection, we prove Proposition \ref{prop-constant-rho-delta=0}.
We also provide some exemplary plots of the solutions that do not admit a closed form expression.

To start with, since $\rho \equiv \rho_0$, \eqref{eq-H-polar-delta=0} gives
\begin{equation}\label{eq-constant-rho-cosW-expr}
     2\lambda \rho_0 \cos W = \frac{rR^2}{2\rho_0^2} - \frac{2\rho_0^2}{r R^2}(1-\rho_0^2),
\end{equation}
which can also be rewritten as
\begin{equation}\label{eq-constant-rho-rR-quadratic}
    \left(\frac{rR^2}{2\lambda\rho_0^3} - \cos W\right)^2 = \cos^2 W - \frac{1}{\lambda^2}\left(1 - \frac{1}{\rho_0^2}\right).
\end{equation}
This completely determines the metric coefficient $r$ in terms of the variables $(R, W)$, so that in particular we can completely eliminate $r$ from the system, which makes the autonomous and simpler to study. 
It is natural to consider the cases $\rho_0 < 1$, $\rho_0=1$, and $\rho_0 > 1$ separately.

\subsubsection{\texorpdfstring{$\rho_0 > 1$}{rho0 > 1}}
In this case the function
\begin{equation}\label{eq-constant-rho-cosW-RHS}
    x \mapsto x - \frac{1}{x}(1-\rho_0^2)
\end{equation}
is clearly positive for $x > 0$, and has a minimum at $x=\sqrt{\rho_0^2-1}$, so that \eqref{eq-constant-rho-cosW-expr} implies that $\cos W$ is positive, and more particularly 
\begin{equation*}
    \cos W \in [\alpha, 1], \qquad \alpha = \frac{1}{\lambda}\sqrt{1-\frac{1}{\rho_0^2}} > 0.
\end{equation*}
Note that $\alpha < 1$ since $\lambda \geq 1$.
Then \eqref{eq-constant-rho-rR-quadratic} gives two branches of solutions
\begin{equation}\label{eq-constant-rho-rR-expr}
    \frac{r_\pm R_\pm^2}{2\rho_0^2} = \lambda\rho_0 \left( \cos W_\pm \pm \sqrt{\cos^2 W_\pm - \alpha^2}\right),
\end{equation}
and so
\eqref{eq-W-polar-delta=0} becomes
\begin{equation}\label{eq-constant-rho-W'}
    W_\pm' 
    = \lambda\rho_0 \left(3\cos W_\pm \pm \sqrt{\cos^2 W_\pm - \alpha^2}\right).
\end{equation}
Now a simple calculus argument shows that $f_\pm(x) = \lambda\rho_0(3x \pm \sqrt{x^2-\alpha^2})$ is uniformly positive for $x \in [\alpha, 1]$ for both choices of sign, and hence $W_\pm' = f_\pm(\cos W_\pm)$ implies that $W_\pm$ uniformly increases for both choices of sign. Consequently, $\cos W_\pm \to \alpha$ at some finite $s$ (both forwards and backwards).
Then \eqref{eq-constant-rho-W'} implies that
$W_\pm$ is never $C^2$-extendible across this point. 
On the other hand \eqref{eq-R-polar-delta=0} implies that $R_\pm$ remains bounded, so that from \eqref{eq-constant-rho-rR-expr} we get that $r_\pm$ remains lower bounded by a positive constant, as $\cos W_\pm \to \alpha$. It follows that this singularity is not a consequence of the geometry degenerating, but is a true singularity of the solutions to the equations, which are consequently not of interest to us.

\subsubsection{\texorpdfstring{$\rho_0=1$}{rho0 = 1}}
In this case \eqref{eq-constant-rho-cosW-expr} directly gives
\begin{equation}\label{eq-rho=1-W'}
    rR^2 = 4\lambda \cos W,
    \quad\text{so that}\quad
    W' 
    = 4\lambda \cos W .
\end{equation}
In particular the first identity implies that we need $\cos W > 0$, i.e.\ $|W| < \pi/2$.
Then \eqref{eq-rho=1-W'} can be solved by substituting $W = \arctan\sinh \gamma$, and the general solution is 
\begin{equation*}
    W(s) = \arctan \sinh( \gamma_0 + 4\lambda s), \qquad \gamma_0 = \arsinh\tan(W_0). 
\end{equation*}
Then it also follows that (recall we are setting $R_0 = \frac12 |\xi_0|$ since $\delta_0=0$)
\begin{equation*}
    R(s) = \frac12|\xi_0| \exp \int_0^s \lambda\sin W = \frac12|\xi_0| \cos^{\frac14}(W_0) \cosh^\frac{1}{4}(\gamma_0 + 4\lambda s),
\end{equation*}
as well as
\begin{equation}\label{eq-rho=1-r}
    r(s) = \frac{4\lambda \cos W}{R^2} = \frac{16\lambda}{|\xi_0|^2 \cos^{\frac12} W_0} \sech^{\frac32}(\gamma_0+4\lambda s).
\end{equation}
The corresponding solution of the Dirac-Yang-Mills system is then given by
\begin{align*}
   z(s) &= \exp\left(-\frac{i}{2}\left( \arctan\sinh(\gamma_0 + 4\lambda s) - W_0 \right)\right)\\
   \xi(s) &= \xi_0 \cos^{-\frac14}(W_0) \cosh^{\frac14}(\gamma_0 + 4\lambda s)  \exp\left(\frac{i}{4}\left( \arctan\sinh(\gamma_0 + 4\lambda s) - W_0 \right)\right).
\end{align*}
Note that $s\mapsto W$ defines a bijection $\R \to (-\pi/2, \pi/2)$ so that since \eqref{eq-rho=1-W'} is autonomous, all of these solutions can  be obtained from the one with $W_0 = 0$ by translating the coordinate $s$, which also does not change the structure of the metric.
These solutions have the interesting feature that the metric coefficient $r$, which describes the geometric radius of the spheres that the manifold is foliated by, tends to 0 at $\pm\infty$.

\subsubsection{\texorpdfstring{$\rho_0 < 1$}{rho0 < 1}}
\label{subsubsec-delta=0-constant-rho<1}

In this case, \eqref{eq-constant-rho-rR-quadratic} gives a unique positive solution
\begin{equation}\label{eq-rR-delta=0-rho<1}
    \frac{r R^2}{2\rho_0^2} = \lambda\rho_0 \left( \cos W + \sqrt{\cos^2 W + \frac{1}{\lambda^2}\left(\frac{1}{\rho_0^2} - 1\right)}\right),
\end{equation}
and \eqref{eq-W-polar-delta=0} becomes
\begin{equation}\label{eq-W-delta=0-rho<1}
    W' 
    = \lambda\rho_0 \left(3\cos W + \sqrt{\cos^2 W + \beta^2}\right), \qquad
    \beta = \frac{1}{\lambda}\sqrt{\frac{1}{\rho_0^2} - 1}.       
\end{equation}
The right-hand side of \eqref{eq-W-delta=0-rho<1} is smooth and uniformly upper and lower bounded, so the corresponding solution $W$ is global.
The associated Dirac-Yang-Mills pair
\begin{align*}
    z(s) &= \rho_0\exp\left( i\left[2\lambda\rho_0 \int_0^s \cos W(t) \,\diff t - W(s) + W_0 \right] \right),\\
    \xi(s) &= \xi_0 \exp\left(\lambda \rho_0 \int_0^s \sin W(t)\, \diff t\right)
    \exp\left(i\lambda \rho_0 \int_0^s \cos W(t)\, \diff t\right)
\end{align*}
as well as the corresponding metric coefficient (cf.\ \eqref{eq-rR-delta=0-rho<1})
\begin{equation*}
    r(s) = \frac{4\lambda\rho_0^3}{|\xi_0|^2} \left( \cos W(s) + \sqrt{\cos^2 W(s) + \frac{1}{\lambda^2}\left(\frac{1}{\rho_0^2} - 1\right)}\right) \exp\left(-2\lambda\rho_0 \int_0^s \sin W(t)\,\diff t\right)
\end{equation*}
depend only on $W$ and are then also globally defined.
In general, the equation \eqref{eq-W-delta=0-rho<1} does not admit a simple closed form solution,%
\footnote{
Provided that the right-hand side is non-zero we can implicitly write
\begin{equation*}
    \lambda \rho_0 s = \int_{W_0}^{W(s)} \frac{\diff W}{3\cos W + \sqrt{\cos^2 W + \beta^2}}
    = 2\int_{\tan \frac{W_0}{2}}^{\tan \frac{W(s)}{2}} \frac{\diff u}{3(1-u^2) + \sqrt{(1-u^2)^2 + \beta^2 (1+u^2)^2}}.
\end{equation*}
which, however, does not seem to admit a simple closed form in general.
}
so we do not try to make the formulas more explicit, but we can still qualitatively analyze the behaviour. To this end we let $\rho_\crit = (1+8\lambda^2)^{-\frac12}$ and consider three cases separately (see also Figure \ref{fig-constant-rho<1-delta=0}).

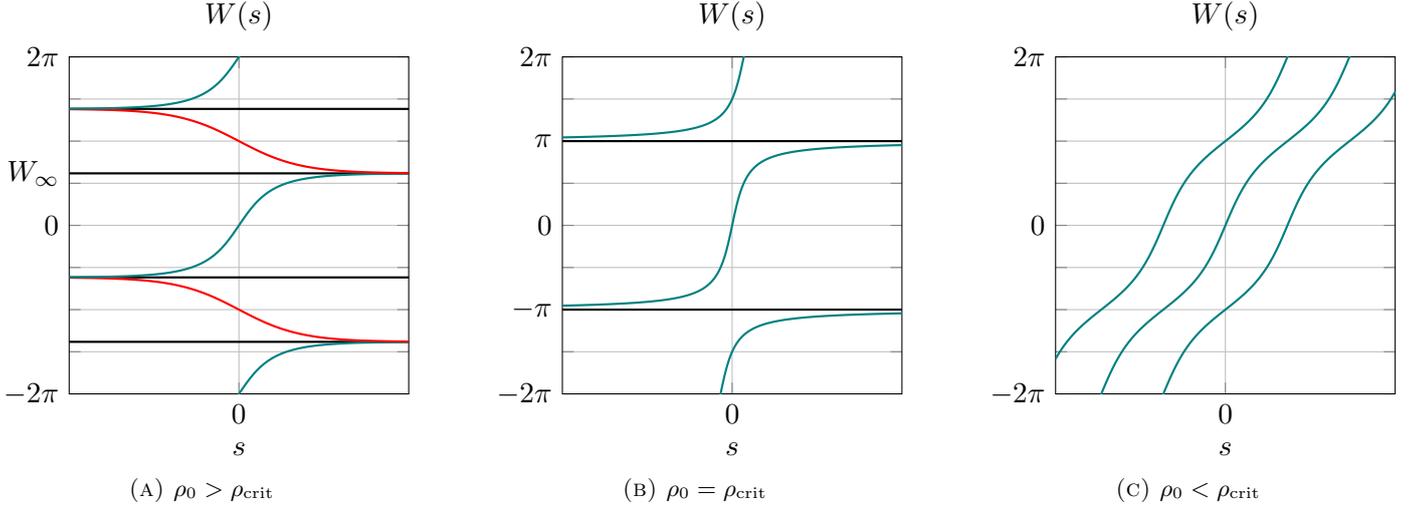
\begin{figure}[t]
\centerline{
\subfloat[$\rho_0 > \rho_\crit$]{
\begin{tikzpicture}

\pgfmathsetmacro{\lam}{1}
\pgfmathsetmacro{\rho}{0.7}
\pgfmathsetmacro{\bet}{(1/\lam)*sqrt(1/\rho^2-1)}
\pgfmathsetmacro{\Winf}{rad(acos(-\bet/sqrt(8)))}
\pgfmathsetmacro{\boxw}{6}
\pgfmathsetmacro{\boxh}{4*pi}

\begin{axis}
[
    title = {$W(s)$},
    grid=both,
    xmin = -\boxw/2, xmax = \boxw/2,
    ymin = -\boxh/2, ymax = \boxh/2,
    width = 0.38\textwidth,
    height = 0.38\textwidth,
    xlabel = {$s$},
    xtick = {0},
    ytick = {-2*pi, -3*pi/2, \Winf-2*pi, -pi, -\Winf, -pi/2, 0, pi/2, \Winf, pi, 2*pi-\Winf, 3*pi/2, 2*pi},
    yticklabels = {$-2\pi$, , , , , , $0$, , $W_\infty$ , , , , $2\pi$}
]

\addplot [domain=-\boxw/2:\boxw/2, samples=50,thick]({x},{\Winf});
\addplot [domain=-\boxw/2:\boxw/2, samples=50,thick]({x},{-\Winf});
\addplot [domain=-\boxw/2:\boxw/2, samples=50,thick]({x},{\Winf-2*pi});
\addplot [domain=-\boxw/2:\boxw/2, samples=50,thick]({x},{-\Winf+2*pi});

\addplot[color=teal,smooth,tension=0.7,thick] table [x index=0,y index=1,col sep=comma] {anc/W_deltapositive_rhosupercritical.csv};
\addplot[color=red,smooth,tension=0.7,thick] table [x index=0,y index=2,col sep=comma] {anc/W_deltapositive_rhosupercritical.csv};
\addplot[color=teal,smooth,tension=0.7,thick] table [x index=0,y index=3,col sep=comma] {anc/W_deltapositive_rhosupercritical.csv};
\addplot[color=red,smooth,tension=0.7,thick] table [x index=0,y index=4,col sep=comma] {anc/W_deltapositive_rhosupercritical.csv};
\addplot[color=teal,smooth,tension=0.7,thick] table [x index=0,y index=5,col sep=comma] {anc/W_deltapositive_rhosupercritical.csv};

\end{axis}
\end{tikzpicture}
}
\qquad
\subfloat[$\rho_0 = \rho_\crit$]{
\begin{tikzpicture}

\pgfmathsetmacro{\lam}{1}
\pgfmathsetmacro{\boxw}{30}
\pgfmathsetmacro{\boxh}{4*pi}

\begin{axis}
[
    title = {$W(s)$},
    grid=both,
    xmin = -\boxw/2, xmax = \boxw/2,
    ymin = -\boxh/2, ymax = \boxh/2,
    width = 0.38\textwidth,
    height = 0.38\textwidth,
    xlabel = {$s$},
    xtick = {0},
    ytick = {-2*pi, -3*pi/2, -pi, -pi/2, 0, pi/2, pi, 3*pi/2, 2*pi},
    yticklabels = {$-2\pi$,, $-\pi$, , $0$, , $\pi$, , $2\pi$}
]

\addplot [domain=-\boxw/2:\boxw/2, samples=50,thick]({x},{pi});
\addplot [domain=-\boxw/2:\boxw/2, samples=50,thick]({x},{-pi});

\addplot[color=teal,smooth,tension=0.7,thick] table [x index=0,y index=1,col sep=comma] {anc/W_deltapositive_rhocritical.csv};
\addplot[color=teal,smooth,tension=0.7,thick] table [x index=0,y index=2,col sep=comma] {anc/W_deltapositive_rhocritical.csv};
\addplot[color=teal,smooth,tension=0.7,thick] table [x index=0,y index=3,col sep=comma] {anc/W_deltapositive_rhocritical.csv};

\end{axis}
\end{tikzpicture}
}
\qquad
\subfloat[$\rho_0 < \rho_\crit$]{
\begin{tikzpicture}

\pgfmathsetmacro{\lam}{1}
\pgfmathsetmacro{\boxw}{20}
\pgfmathsetmacro{\boxh}{4*pi}

\begin{axis}
[
    title = {$W(s)$},
    grid=both,
    xmin = -\boxw/2, xmax = \boxw/2,
    ymin = -\boxh/2, ymax = \boxh/2,
    width = 0.38\textwidth,
    height = 0.38\textwidth,
    xlabel = {$s$},
    xtick = {0},
    ytick = {-2*pi, -3*pi/2, -pi, -pi/2, 0, pi/2, pi, 3*pi/2, 2*pi},
    yticklabels = {$-2\pi$,, , , $0$, , , , $2\pi$}
]

\addplot[color=teal,smooth,tension=0.7,thick] table [x index=0,y index=1,col sep=comma] {anc/W_deltapositive_rhosubcritical.csv};
\addplot[color=teal,smooth,tension=0.7,thick] table [x index=0,y index=2,col sep=comma] {anc/W_deltapositive_rhosubcritical.csv};
\addplot[color=teal,smooth,tension=0.7,thick] table [x index=0,y index=3,col sep=comma] {anc/W_deltapositive_rhosubcritical.csv};

\end{axis}
\end{tikzpicture}
}
}
\caption{Solutions $W(s)$ of \eqref{eq-W-delta=0-rho<1}.}
\label{fig-constant-rho<1-delta=0}
\end{figure}

\begin{enumerate}
    \item If $\rho_0 > \rho_\crit$, the equation \eqref{eq-W-delta=0-rho<1} has fixed points whenever $\cos W = -\frac{\beta}{2\sqrt{2}} < 0$.
    Define 
    \begin{equation}\label{eq-W-infty}
        W_\infty = \arccos\left(-\frac{\beta}{2\sqrt{2}}\right) = \pi - \arcsin\left( \frac{1}{2\sqrt{2}\lambda}\left(\frac{1}{\rho_\crit^2} - \frac{1}{\rho_0^2}\right) \right) \in \left(\frac{\pi}{2},\, \pi\right).
    \end{equation}
    The solutions with $W_0 = W_\infty$ and $W_0 = -W_\infty$ are then constant.
    Solutions with $|W_0| < W_\infty$ are increasing and have $W \to \pm W_\infty$ as $s \to \pm \infty$, while solutions with $W_\infty < |W_0| < \pi$ are decreasing and have $W \to \pm (W_\infty - (1 \mp \mathrm{sgn}\,W_0)\pi)$ as $s \to \pm\infty$.
    By the autonomy of the system and the fact that the solutions provide a bijection from $\R$ to an interval between the fixed points (explicitly $(-W_\infty, W_\infty)$ mod $2\pi$ for the increasing ones and $(W_\infty, 2\pi-W_\infty)$ mod $2\pi$ for the decreasing ones), we see that the all the increasing (resp.\ all the decreasing) solutions are related to each other by translations of $s \mapsto s+s_0$ and $W \mapsto W+2\pi k$ for $k\in\Z$. Furthermore, the system also has a reflection symmetry $(s,W) \mapsto -(s,W)$ and in particular one does not lose generality by only considering solutions with $W_0 \in \{0, W_\infty, \pi\}$, where $W_0 = W_\infty$ is of constant type, $W_0=0$ is of increasing type, and $W_0=\pi$ is of decreasing type.

    In the constant case $W_0 = W_\infty$, one can explicitly express the DYM pair as shown in Proposition \ref{prop-constant-rho-delta=0} (iii).%
    Thus we see that, the solutions have $\xi\to\infty$ and $r\to 0$ as $s\to \infty$, and $\xi\to 0$ and $r\to \infty$ as $s\to-\infty$, while the variable $z$ oscillates and has no limits.

    In the increasing case $W_0 = 0$, we see that $\sin W \to \pm\sin W_\infty > 0$ as $s\to \pm\infty$. It follows then that $\xi \to \infty$ and $r \to 0$ as $s\to\pm \infty$, while $z$ oscillates.

    Analogously, in the decreasing $W_0 = \pi$ case, we have that $\sin W \to \mp\sin W_\infty$ as $s\to \pm\infty$ and it follows that $\xi \to 0$ and $r\to\infty$ as $s\to\pm\infty$, while $z$ again oscillates.

    \item If $\rho_0 = \rho_\crit$ the equation \eqref{eq-W-delta=0-rho<1} has fixed points at $W = \pi + 2k\pi$ for $k\in \Z$. 
    Note that these fixed points are limits of two colliding fixed points from the previous case $\rho_0 > \rho_\crit$. In particular, the regions that previously admitted decreasing solutions close up, and one only has solutions of constant type or of increasing type. Then by translation and symmetries, as before it suffices to consider only $W_0 \in \{0, \pi\}$.

    The increasing type solutions with $W_0 = 0$ are now somewhat more interesting since $\sin W \to 0$ as $s\to\pm\infty$ so that one could expect that $\sin W$ is integrable. However, as $\cos W \to -1$, we can for (without loss of generality positive) large enough $s$ estimate
    \begin{equation*}
        W' = \frac{8\sin^2 W}{-3\cos W + \sqrt{\cos^2 W + 8}} \leq \sin^2 W,
    \end{equation*}
    which implies that $\sin W(s) \gtrsim \frac{1}{s}$ for large $s$, and hence is not integrable in this case either.
    Thus $\xi \to \infty$ and $r \to 0$, while $z$ oscillates, as $s \to \pm \infty$.
 
    On the other hand, the constant type solutions with $W_0 = \pi$ also have constant $r$ and $s\mapsto (z,\xi)$ is periodic, so that the solution descends to a DYM pair on $(\Sph^1\times \Sph^2)_r$.
    We can make this more precise geometrically by defining a new coordinate $t = \lambda(1+8\lambda^2)^{-\frac12}s = \lambda\rho_\crit s$, in terms of which
    \begin{equation*}
        (z,\xi)(t) = (\rho_\crit e^{-2it},\, \xi_0 e^{-it}),
    \end{equation*}
    with respect to the parametrization $t\mapsto e^{it} \in \Sph^1$ by arc length (as a submanifold of $\R^2$).
    The metric then becomes
    \begin{equation*}
        g = \frac{64}{|\xi_0|^4(1+8\lambda^2)^2}\,g_{\Sph^1} + \frac{64\lambda^2}{|\xi_0|^4(1+8\lambda^2)^{3}}g_{\Sph^2},
    \end{equation*}
    and we can interpret the solution as being defined on the Riemannian direct product of spheres $\Sph^1\left(8\rho_\crit^2/|\xi_0|^2\right) \times \Sph^2\left(8\lambda\rho_\crit^3/|\xi_0|^2\right)$
    where the arguments denote the radii.
    %
    \item If $\rho_0 < \rho_\crit$, we have $W' > 0$ uniformly for any choice of $W_0$, so that $W$ increases uniformly and $W \to \pm \infty$ as $s\to \pm\infty$.
    In particular, for any $W_0$, the corresponding solution $W$ defines a bijection $\R \to \R$, so that by the autonomy the equation and uniqueness of its solutions, we can always translate $s$ and assume without loss of generality that $W_0 = 0$.
    Then $W$ is odd since the equation \eqref{eq-W-delta=0-rho<1} is invariant under reflections $(s,W) \mapsto -(s,W)$.
    Furthermore, since the equation is also invariant modulo $2\pi$, we have         \begin{equation*}
        W(s+kT) = W(s) + 2\pi k, \qquad T=W^{-1}(2\pi),
    \end{equation*}
    for all $k \in \mathbb{Z}$.
    In particular, $s \mapsto e^{iW(s)}$ is periodic with period $T$.
    We now wish to inspect the periodicity of the associated Dirac-Yang-Mills pair $(z,\xi)$ and metric coefficient $r$.
    Since $W$ is odd, we see that $s\mapsto \cos W(s)$ is even, while $s \mapsto \sin W(s)$ is odd, and they are both periodic with period $T$.
    It follows then that
    \begin{equation*}
        \int_0^T \sin W(t)\,\diff t = \int_{-T/2}^{T/2} \sin(t)\,\diff t = 0
    \end{equation*}
    so that $s\mapsto \int_0^s \sin W$ is periodic with period $T$. 
    Thus $(z,\xi)$ and $r$ are periodic if and only if $s\mapsto \exp(i\lambda\rho_0\int_0^s \cos W)$ is periodic and has period equal to an integer multiple of $T$. This is the case if and only if
    \begin{equation*}
        \frac{\lambda\rho_0}{2\pi} \int_0^{T} \cos W(t)\,\diff t = \frac{\lambda\rho_0}{\pi} \int_0^{\frac{T}{2}} \cos W(t)\,\diff t =: \frac{p}{q} \in \Q,
    \end{equation*}
    in which case the minimal period is $qT$, assuming that $p,q \in \Z$ are coprime.    
    Now the solution is a continuous (even smooth) bijection $W : \R \to \R$ depending continuously on $\rho_0$, it follows that the period $T(\rho_0) = W^{-1}(2\pi, \rho_0)$ also depends continuously on $\rho_0$.
    In particular, the function 
    \begin{equation*}
        f : \rho_0 \mapsto \frac{\lambda\rho_0}{\pi} \int_0^{\frac{T(\rho_0)}{2}} \cos W(t)\,\diff t
    \end{equation*}
    is continuous, and hence will attain infinitely many rational values provided that it is non-constant.
    To see that it is indeed non-constant, note that as $\rho_0 \to 0$, the right-hand side of \eqref{eq-W-delta=0-rho<1} tends to infinity, so that in particular $T$ is small for small $\rho_0$. Thus, we can make $f(\rho_0)$ as small as we like, and in particular $f$ cannot be constant as it would then have to be identically zero. It follows that the constructed Dirac-Yang-Mills pair is periodic for countably many choices of $\rho_0$, and for these choices it descends to a solution on $(\Sph^1 \times \Sph^2)_r$.
\end{enumerate}

\subsection{\texorpdfstring{The case $\delta_0 > 0$}{The case delta0 > 0}}

The situation for $\delta_0 > 0$ is in many cases similar to that of $\delta_0 = 0$, but the cases with $\rho_0 \leq 1$ display some crucial differences.

\begin{proposition}[The case $\delta_0 > 0$]
    \label{prop-constant-rho-delta>0}
    Using the notation of Proposition \ref{lem-polar-form}, assume that $\delta_0 > 0$ and $\rho \equiv \rho_0$ for a constant $\rho_0 > 0$, and define $\rho_\crit = (1+8\lambda^2)^{-\frac12}$.
    \begin{enumerate}
        \item If $\rho_0 > 1$, the DYM pair is singular for finite $s$.
        \item If $\rho_0 = 1$, then without loss of generality $W_0=0$, the solution is given by
        \begin{align*}
            P(s) &= \frac12\arcosh\left( \cosh(2P_0)\cosh(4\lambda s)\right),
            \\
            W(s) &= \arcsin\frac{\sinh(4\lambda s)}{\sqrt{\cosh^2(4\lambda s)  - \sech^2(2P_0)}},
        \end{align*}
        and the associated DYM pair is defined globally on $(\R \times \Sph^2)_r$ with
        \begin{equation*}
            r(s) = \frac{8\sqrt{2}\lambda \sinh(2 P_0)}{\delta_0^2} \left[\cosh(2P_0)\cosh(4\lambda s + \gamma_0) + 1\right]^{-\frac32}.
        \end{equation*}
        \item If $\rho_\crit \leq \rho_0 < 1$, then there are distinguished orbits with $(P,W) \to (0,\pi/2)$ and $(P,W) \to (0,3\pi/2)$. All other orbits (i.e.\ not obtainable from the distinguished orbits by symmetries) are global and without loss of generality have $W_0 \in \{0, \pi\}$.
        \smallskip
        \begin{itemize}[itemsep=0.1cm]
            \item If $W_0 = 0$, the orbit is such that $P\to\infty$ and $W \to \pm W_\infty$ as $s\to\pm \infty$.
            \item If $W_0 = \pi$, the orbit is such that $W \to W_\infty$ as $s \to \infty$ and $W \to 2\pi - W_\infty$ as $s \to -\infty$, and $P \to \infty$ at both ends.
        \end{itemize}
        \item If $\rho_0 < \rho_\crit$, one can assume without loss of generality that $W_0 = \pi$. We also let
        \begin{equation*}
            P_\infty = P_\infty(\rho_0) = \arcoth\sqrt{1 + \frac{1}{4\lambda^2}\left(\frac{1}{\rho_0^2}-\frac{1}{\rho_\crit^2}\right)}.
        \end{equation*}
        Then there exists a number $P_\crit > P_\infty > 0$ with the following properties.
        \smallskip
        \begin{itemize}[itemsep=0.1cm]
            \item If $P_0 > P_\crit$ there exists $T > 0$ such that $P$ is $T$-periodic and $W(s + T) = W(s) + 2\pi$ for all $s \in \R$. 
            \item If $P_0 = P_\crit$ then there exists $0<s_\crit<\infty$ for which 
             $(P,W) \to (0, \pi\pm \frac{\pi}{2})$ as $s\to \pm s_\crit$.
            \item If $0<P_0<P_\crit$ then either $P_0=P_\infty$ and the solution is constant $(P,W) \equiv (P_\infty, \pi)$, or otherwise for $P_0 \neq P_\infty$ there exists $T > 0$ such that $P$ and $W$ are $T$-periodic. 
        \end{itemize}
        \smallskip
        In particular, there exist countably many pairs $(\rho_0,P_0)$ with $\rho_0 < \rho_\crit$ and \\ $P_0 \notin  \{P_\infty,P_\crit\}$, for which the associated DYM pair is periodic and thus descends to a solution on $(\Sph^1 \times \Sph^2)_r$.
        For the fixed point solution with $(P_0, W_0) = (P_\infty(\rho_0), \pi)$, the periodic DYM pairs are classified by rational numbers $p/q \in \Q$. Explicitly, if we define the coordinate $t = \frac{\lambda\rho_0}{\sqrt{pq}}\, s$, then the DYM pair is
        \begin{align*}
            z(t) &= \rho_0(\lambda,p/q) e^{-i(p+q)t}\\
             \QuatRe(\xi)(t) &= \delta_0\sqrt{\frac{q}{q-p}} \,e^{- ipt + i\arg\QuatRe(\xi_0)},\\
            \QuatIm(\xi)(t) &= -\delta_0\sqrt{\frac{p}{q-p}} \,e^{- iqt - i\arg\QuatRe(\xi_0)},
        \end{align*}
        for explicitly computable $\rho_0(\lambda, p/q)$, 
        and is defined the closed Riemannian manifold
        \begin{equation*}
            \Sph^1\left(\frac{2\sqrt{pq}}{\lambda\delta_0}\sqrt{1 - \frac{\rho_0^2}{\rho_\crit^2}} \right) \times \Sph^2\left( \frac{2\rho_0}{\delta_0}\sqrt{1 - \frac{\rho_0^2}{\rho_\crit^2}} \right).
        \end{equation*}
    \end{enumerate}
\end{proposition}

The remainder of the section is devoted to proving Proposition \ref{prop-constant-rho-delta>0}.
Let $\rho \equiv \rho_0$, so that also $H\equiv 0$, and \eqref{eq-H-polar-delta>0} gives
\begin{equation}\label{eq-constant-rho-cosW-delta>0}
    2\lambda \rho_0 \tanh P \cos W = \frac{\delta_0^2 r\cosh P}{4\rho_0^2} - \frac{4\rho_0^2}{\delta_0^2 r \cosh P}(1-\rho_0^2),
\end{equation}
or equivalently
\begin{equation}\label{eq-constant-rho-rR-quadratic-delta>0}
    \left( \frac{\delta_0^2 r\cosh P}{4\lambda\rho_0^3} - \tanh P \cos W \right)^2 = \tanh^2 P \cos^2 W - \frac{1}{\lambda^2}\left(1 - \frac{1}{\rho_0^2}\right).
\end{equation}
As before, we consider $\rho_0<1$, $\rho_0=1$, and $\rho_0>1$ separately.

\subsubsection{\texorpdfstring{$\rho_0 > 1$}{rho0 > 1}}
As before, properties of the function \eqref{eq-constant-rho-cosW-RHS} in this case imply that
\begin{equation*}
    \tanh P \cos W \in [\alpha, 1], \qquad \alpha = \frac{1}{\lambda}\sqrt{1-\frac{1}{\rho_0^2}} > 0,
\end{equation*}
so that in particular $\cos W > 0$, and $P$ cannot get too close to $0$.
Now we again get two branches of solutions of \eqref{eq-constant-rho-rR-quadratic-delta>0}, for which \eqref{eq-W-polar-delta>0} becomes
\begin{equation*}
    W_\pm' = \lambda \rho_0 (2\coth P_\pm + \tanh P_\pm) \cos W_\pm \pm \lambda\rho_0\sqrt{\tanh^2 P_\pm \cos^2 W_\pm - \frac{1}{\lambda^2}\left(1 - \frac{1}{\rho_0^2}\right)}.
\end{equation*}
As before, one can show that the right-hand side is uniformly positive for both choices of sign and these solutions are not of interest, so we omit the details.

\subsubsection{\texorpdfstring{$\rho_0 = 1$}{rho0 = 1}} 
Here \eqref{eq-constant-rho-cosW-delta>0} gives
\begin{equation}\label{eq-delta>0-r-relation}
    2\lambda \tanh P \cos W = \frac{1}{4}\,\delta_0^2 r\cosh P,
\end{equation}
and the system reduces to
\begin{equation*}
    \begin{cases}
        P' = 2\lambda\sin W,\\[0.1cm]
        W' = 4\lambda\coth 2P \cos W.
    \end{cases}
\end{equation*}
In view of \eqref{eq-delta>0-r-relation}, the only interesting sector is $\tanh P \cos W > 0$, since we need the metric coefficient $r$ to be positive. In particular, we demand $|W_0| < \pi/2$, so that $\cos W_0 > 0$.
A calculation shows that if $F = \sin W \tanh 2P$, then
\begin{equation*}
    F' = 4\lambda(1-F^2), \quad\text{i.e.}\quad F(s) = \tanh(4\lambda s + \gamma_0),
\end{equation*}
where $\gamma_0 = \artanh(\sin W_0 \tanh 2P_0)$.
It follows then that
\begin{equation*}
    P'\tanh 2P = 2\lambda\tanh(4\lambda s + \gamma_0)
\end{equation*}
i.e.\
\begin{align*}
    P(s) &= \frac12\arcosh\left(\frac{\cosh(2P_0)}{ \cosh(\gamma_0)} \cosh(4\lambda s + \gamma_0)\right)\\[0.1cm]
    &=
    \frac12\arcosh\left(\sqrt{1 + \cos^2(W_0) \sinh^2 (2P_0)} \,\cosh(4\lambda s + \gamma_0)\right).
\end{align*}
Then also
\begin{equation*}
    W(s) = \arcsin\frac{\sinh(4\lambda s + \gamma_0)}{\sqrt{\cosh^2(4\lambda s + \gamma_0)  - \frac{1}{1+\cos^2(W_0)\sinh^2 (2P_0)}}},
\end{equation*}
and the metric coefficient is
\begin{equation*}
    r(s) = \frac{8\sqrt{2}\lambda \cos(W_0) \sinh(2 P_0)}{\delta_0^2} \left(\sqrt{1 + \cos^2(W_0) \sinh^2 (2P_0)} \,\cosh(4\lambda s + \gamma_0) + 1\right)^{-\frac32}.
\end{equation*}
Since $W$ defines a bijection $\R \to (-\pi/2,\pi/2)$ for any choice of initial data, we can assume without loss of generality that $W_0=0$ by translating $s$, since the system is autonomous.
Now one could in principle insert these formulae for $P$ and $W$ into the formulas for $z$ and $\xi$ from Lemma \ref{lem-polar-form}, but this becomes too cumbersome to handle explicitly so we do not pursue this further, and instead we just note that these solutions are globally defined on $(\R \times \Sph^2)_r$.

\subsubsection{\texorpdfstring{$\rho_0 < 1$}{rho0 < 1}}
In this case, we get a unique positive branch
\begin{equation*}
    \frac{\delta_0^2 r\cosh P}{4\lambda\rho_0^3} = \tanh P \cos W + \sqrt{\tanh^2 P \cos^2 W + \frac{1}{\lambda^2}\left(\frac{1}{\rho_0^2} - 1\right)},
\end{equation*}
for which the system reduces to
\begin{equation}\label{eq-P-W-system-delta0>0-rho0<1}
    \begin{cases}
        P' = 2\lambda\rho_0 \sin W,\\
        W' = \lambda \rho_0 (2\coth P + \tanh P) \cos W + \lambda\rho_0\sqrt{\tanh^2 P \cos^2 W + \frac{1}{\lambda^2}\left(\frac{1}{\rho_0^2} - 1\right)}.
    \end{cases}
\end{equation}

We note that the system is invariant under translations $W \mapsto W + 2\pi k$ for any integer $k$, and the system also has a reflection symmetry $(s,W) \mapsto -(s, W)$.
Here it will be particularly convenient to study the solutions in the region $0 < W < 2\pi$, and note that this region is also reflection symmetric along $W=\pi$, in view of the symmetry $(s,W) \mapsto (-s, 2\pi-W)$ (cf.\ also Figure \ref{fig-P-W-phase-delta0>0}).

The system \eqref{eq-P-W-system-delta0>0-rho0<1} has a singularity at $P=0$.
To study the behaviour of the system in the vicinity of $P = 0$, we rescale the vector field associated to \eqref{eq-P-W-system-delta0>0-rho0<1} by $\frac{\tanh{P}}{\lambda \rho_0}$ to get a new vector field on $\R_{>0} \times \R$ given by
\begin{equation*}
    X =  
    \begin{bmatrix}
        2 \tanh P \sin W\\
         (2 + \tanh^2 P) \cos W + \tanh P \sqrt{\tanh^2 P \cos^2 W + \beta^2} 
    \end{bmatrix},
\end{equation*}
which extends smoothly across $P=0$ and whose integral curves coincide (as sets) with the integral curves of \eqref{eq-P-W-system-delta0>0-rho0<1} on $\R_{>0} \times \R$.

We note that $X$ restricted to $P=0$ satisfies $X(0,W) = (0, 2\cos W)$, so that the line $P=0$ is invariant. 
In particular, $X$ has fixed points at $(P,W) = (0, \pi \pm \frac{\pi}{2} + 2\pi k)$, connected by integral curves that stay on the line $P = 0$.
We study these fixed points in more detail.
Since by the translation invariance of the system we can without loss of generality assume $0 \leq W_0 \leq 2\pi$, it suffices to consider the fixed points at $W = \pi \pm \frac{\pi}{2}$.
The differential of $X$ satsifies
\begin{equation*}
    \diff X_{(0, \pi \pm \frac{\pi}{2})}
    =
    \begin{bmatrix}
        \mp 2 & 0\\
        \beta & \pm 2
    \end{bmatrix}
\end{equation*}
The corresponding eigenvalues and eigenvectors are
\begin{equation*}
    \left(\mp 2, \,\begin{bmatrix} 0\\ 1\end{bmatrix}\right)
    \quad\text{and}\quad
    \left(\pm 2,\, \begin{bmatrix} 1\\ \mp\frac{1}{4}\beta\end{bmatrix}\right).
\end{equation*}
Denote the $P > 0$ branch of the stable manifold entering $(0, \frac{3\pi}{2})$ in the direction $(1,-\frac{1}{4}\beta)$ by $\mathscr{S}_+$ and the $P > 0$ branch of the unstable manifold exiting $(0, \frac{\pi}{2})$ in the direction $(1,\frac{1}{4}\beta)$ by $\mathscr{U}_+$.
Thus, an orbit $(P,W)$ of \eqref{eq-P-W-system-delta0>0-rho0<1} with $P \to 0$ must satisfy $(P,W) \to (0, \pi \pm \frac{\pi}{2})$ and coincide with the stable (resp.\ unstable) manifold associated to this zero of $X$, while all other orbits stay away from $P=0$.

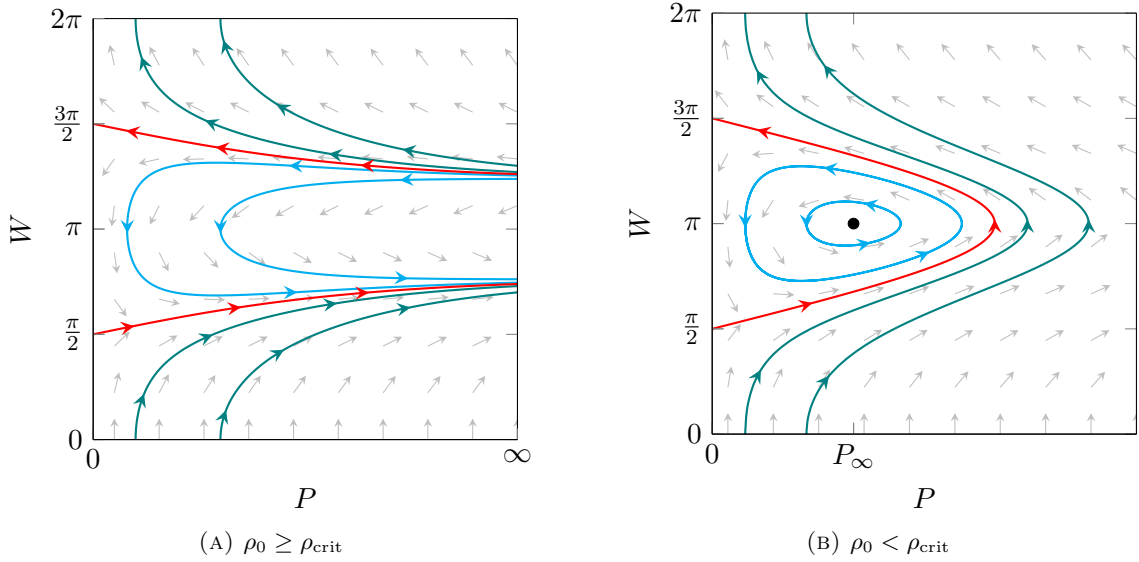
\begin{figure}[t]
\centerline{
\subfloat[$\rho_0 \geq \rho_\crit$]{
\begin{tikzpicture}

\pgfmathsetmacro{\lam}{1}
\pgfmathsetmacro{\rho}{1.4/(sqrt(1+8*\lam^2))}
\pgfmathsetmacro{\boxw}{2.5}
\pgfmathsetmacro{\boxh}{2*pi}

\begin{axis}
[
    xmin = 0, xmax = \boxw,
    ymin = 0, ymax = \boxh,
    width = 0.45\textwidth,
    height = 0.45\textwidth,
    xlabel = {$P$},
    ylabel = {$W$},
    xtick = {0, 2.5},
    xticklabels = {$0$, $\infty$},
    ytick = {-2*pi,-3*pi/2,-pi,-pi/2,0,pi/2,pi,3*pi/2,2*pi},
    yticklabels = {$-2\pi$, $-\frac{3\pi}{2}$, $-\pi$, $-\frac{\pi}{2}$, $0$, $\frac{\pi}{2}$, $\pi$, $\frac{3\pi}{2}$, $2\pi$},
    view = {0}{90}
]

\addplot3[
    quiver = {
        u = { (2*\lam*\rho*sin(deg(y))) / sqrt( (2*\lam*\rho*sin(deg(y)))^2/\boxw^2 + (\lam*\rho*((2/tanh(x)+tanh(x))*cos(deg(y)) + sqrt(tanh(x)^2*cos(deg(y))^2 + 1/\lam^2 * (1/\rho^2-1))))^2 /\boxh^2 )},
        v = {(\lam*\rho*((2/tanh(x)+tanh(x))*cos(deg(y)) + sqrt(tanh(x)^2*cos(deg(y))^2 + 1/\lam^2 * (1/\rho^2-1)))) / sqrt( (2*\lam*\rho*sin(deg(y)))^2/\boxw^2 + (\lam*\rho*((2/tanh(x)+tanh(x))*cos(deg(y)) + sqrt(tanh(x)^2*cos(deg(y))^2 + 1/\lam^2 * (1/\rho^2-1))))^2/\boxh^2 )},
        scale arrows = 0.05,
    },
    -stealth,
    domain = 1/8:\boxw,
    samples = 10,
    domain y = 0:\boxh,
    lightgray] 
{0};

%


\addplot[color=cyan,smooth,tension=0.7,thick] table [x index=0,y index=1,col sep=comma] {anc/PW_deltapositive_rhosupercritical.csv};
\addplot[cyan, quiver={
        u = { 2*\lam*\rho*sin(deg(y))) / sqrt( (2*\lam*\rho*sin(deg(y)))^2/(\boxw^2) + (\lam*\rho*((2/tanh(x)+tanh(x))*cos(deg(y)) + sqrt(tanh(x)^2*cos(deg(y))^2 + 1/\lam^2 * (1/\rho^2-1))))^2/\boxh^2 ) },
        v = {(\lam*\rho*((2/tanh(x)+tanh(x))*cos(deg(y)) + sqrt(tanh(x)^2*cos(deg(y))^2 + 1/\lam^2 * (1/\rho^2-1)))) / sqrt( (2*\lam*\rho*sin(deg(y)))^2/(\boxw^2) + (\lam*\rho*((2/tanh(x)+tanh(x))*cos(deg(y)) + sqrt(tanh(x)^2*cos(deg(y))^2 + 1/\lam^2 * (1/\rho^2-1))))^2/\boxh^2 ) },
        scale arrows = 0.01
      }, -stealth, very thick] 
      table[
        x index=0,
        y index=1,
        col sep=comma,
      ] {anc/PW_deltapositive_rhosupercritical_arrows.csv};

\addplot[color=cyan,smooth,tension=0.7,thick] table [x index=2,y index=3,col sep=comma] {anc/PW_deltapositive_rhosupercritical.csv};
\addplot[cyan, quiver={
        u = { 2*\lam*\rho*sin(deg(y))) / sqrt( (2*\lam*\rho*sin(deg(y)))^2/(\boxw^2) + (\lam*\rho*((2/tanh(x)+tanh(x))*cos(deg(y)) + sqrt(tanh(x)^2*cos(deg(y))^2 + 1/\lam^2 * (1/\rho^2-1))))^2/\boxh^2 ) },
        v = {(\lam*\rho*((2/tanh(x)+tanh(x))*cos(deg(y)) + sqrt(tanh(x)^2*cos(deg(y))^2 + 1/\lam^2 * (1/\rho^2-1)))) / sqrt( (2*\lam*\rho*sin(deg(y)))^2/(\boxw^2) + (\lam*\rho*((2/tanh(x)+tanh(x))*cos(deg(y)) + sqrt(tanh(x)^2*cos(deg(y))^2 + 1/\lam^2 * (1/\rho^2-1))))^2/\boxh^2 ) },
        scale arrows = 0.01
      }, -stealth, very thick] 
      table[
        x index=2,
        y index=3,
        col sep=comma,
      ] {anc/PW_deltapositive_rhosupercritical_arrows.csv};

\addplot[color=teal,smooth,tension=0.7,thick] table [x index=4,y index=5,col sep=comma] {anc/PW_deltapositive_rhosupercritical.csv};
\addplot[teal, quiver={
        u = { 2*\lam*\rho*sin(deg(y))) / sqrt( (2*\lam*\rho*sin(deg(y)))^2/(\boxw^2) + (\lam*\rho*((2/tanh(x)+tanh(x))*cos(deg(y)) + sqrt(tanh(x)^2*cos(deg(y))^2 + 1/\lam^2 * (1/\rho^2-1))))^2/\boxh^2 ) },
        v = {(\lam*\rho*((2/tanh(x)+tanh(x))*cos(deg(y)) + sqrt(tanh(x)^2*cos(deg(y))^2 + 1/\lam^2 * (1/\rho^2-1)))) / sqrt( (2*\lam*\rho*sin(deg(y)))^2/(\boxw^2) + (\lam*\rho*((2/tanh(x)+tanh(x))*cos(deg(y)) + sqrt(tanh(x)^2*cos(deg(y))^2 + 1/\lam^2 * (1/\rho^2-1))))^2/\boxh^2 ) },
        scale arrows = 0.01
      }, -stealth, very thick] 
      table[
        x index=4,
        y index=5,
        col sep=comma,
      ] {anc/PW_deltapositive_rhosupercritical_arrows.csv};

\addplot[color=teal,smooth,tension=0.7,thick] table [x index=6,y index=7,col sep=comma] {anc/PW_deltapositive_rhosupercritical.csv};
\addplot[teal, quiver={
        u = { 2*\lam*\rho*sin(deg(y))) / sqrt( (2*\lam*\rho*sin(deg(y)))^2/(\boxw^2) + (\lam*\rho*((2/tanh(x)+tanh(x))*cos(deg(y)) + sqrt(tanh(x)^2*cos(deg(y))^2 + 1/\lam^2 * (1/\rho^2-1))))^2/\boxh^2 ) },
        v = {(\lam*\rho*((2/tanh(x)+tanh(x))*cos(deg(y)) + sqrt(tanh(x)^2*cos(deg(y))^2 + 1/\lam^2 * (1/\rho^2-1)))) / sqrt( (2*\lam*\rho*sin(deg(y)))^2/(\boxw^2) + (\lam*\rho*((2/tanh(x)+tanh(x))*cos(deg(y)) + sqrt(tanh(x)^2*cos(deg(y))^2 + 1/\lam^2 * (1/\rho^2-1))))^2/\boxh^2 ) },
        scale arrows = 0.01
      }, -stealth, very thick] 
      table[
        x index=6,
        y index=7,
        col sep=comma,
      ] {anc/PW_deltapositive_rhosupercritical_arrows.csv};

\addplot[color=teal,smooth,tension=0.7,thick] table [x index=8,y index=9,col sep=comma] {anc/PW_deltapositive_rhosupercritical.csv};
\addplot[teal, quiver={
        u = { 2*\lam*\rho*sin(deg(y))) / sqrt( (2*\lam*\rho*sin(deg(y)))^2/(\boxw^2) + (\lam*\rho*((2/tanh(x)+tanh(x))*cos(deg(y)) + sqrt(tanh(x)^2*cos(deg(y))^2 + 1/\lam^2 * (1/\rho^2-1))))^2/\boxh^2 ) },
        v = {(\lam*\rho*((2/tanh(x)+tanh(x))*cos(deg(y)) + sqrt(tanh(x)^2*cos(deg(y))^2 + 1/\lam^2 * (1/\rho^2-1)))) / sqrt( (2*\lam*\rho*sin(deg(y)))^2/(\boxw^2) + (\lam*\rho*((2/tanh(x)+tanh(x))*cos(deg(y)) + sqrt(tanh(x)^2*cos(deg(y))^2 + 1/\lam^2 * (1/\rho^2-1))))^2/\boxh^2 ) },
        scale arrows = 0.01
      }, -stealth, very thick] 
      table[
        x index=8,
        y index=9,
        col sep=comma,
      ] {anc/PW_deltapositive_rhosupercritical_arrows.csv};

\addplot[color=teal,smooth,tension=0.7,thick] table [x index=10,y index=11,col sep=comma] {anc/PW_deltapositive_rhosupercritical.csv};
\addplot[teal, quiver={
        u = { 2*\lam*\rho*sin(deg(y))) / sqrt( (2*\lam*\rho*sin(deg(y)))^2/(\boxw^2) + (\lam*\rho*((2/tanh(x)+tanh(x))*cos(deg(y)) + sqrt(tanh(x)^2*cos(deg(y))^2 + 1/\lam^2 * (1/\rho^2-1))))^2/\boxh^2 ) },
        v = {(\lam*\rho*((2/tanh(x)+tanh(x))*cos(deg(y)) + sqrt(tanh(x)^2*cos(deg(y))^2 + 1/\lam^2 * (1/\rho^2-1)))) / sqrt( (2*\lam*\rho*sin(deg(y)))^2/(\boxw^2) + (\lam*\rho*((2/tanh(x)+tanh(x))*cos(deg(y)) + sqrt(tanh(x)^2*cos(deg(y))^2 + 1/\lam^2 * (1/\rho^2-1))))^2/\boxh^2 ) },
        scale arrows = 0.01
      }, -stealth, very thick] 
      table[
        x index=10,
        y index=11,
        col sep=comma,
      ] {anc/PW_deltapositive_rhosupercritical_arrows.csv};

\addplot[color=red,smooth,tension=0.7,thick] table [x index=12,y index=13,col sep=comma] {anc/PW_deltapositive_rhosupercritical.csv};
\addplot[red, quiver={
        u = { 2*\lam*\rho*sin(deg(y))) / sqrt( (2*\lam*\rho*sin(deg(y)))^2/(\boxw^2) + (\lam*\rho*((2/tanh(x)+tanh(x))*cos(deg(y)) + sqrt(tanh(x)^2*cos(deg(y))^2 + 1/\lam^2 * (1/\rho^2-1))))^2/\boxh^2 ) },
        v = {(\lam*\rho*((2/tanh(x)+tanh(x))*cos(deg(y)) + sqrt(tanh(x)^2*cos(deg(y))^2 + 1/\lam^2 * (1/\rho^2-1)))) / sqrt( (2*\lam*\rho*sin(deg(y)))^2/(\boxw^2) + (\lam*\rho*((2/tanh(x)+tanh(x))*cos(deg(y)) + sqrt(tanh(x)^2*cos(deg(y))^2 + 1/\lam^2 * (1/\rho^2-1))))^2/\boxh^2 ) },
        scale arrows = 0.01
      }, -stealth, very thick] 
      table[
        x index=12,
        y index=13,
        col sep=comma,
      ] {anc/PW_deltapositive_rhosupercritical_arrows.csv};

\addplot[color=red,smooth,tension=0.7,thick] table [x index=14,y index=15,col sep=comma] {anc/PW_deltapositive_rhosupercritical.csv};
\addplot[red, quiver={
        u = { 2*\lam*\rho*sin(deg(y))) / sqrt( (2*\lam*\rho*sin(deg(y)))^2/(\boxw^2) + (\lam*\rho*((2/tanh(x)+tanh(x))*cos(deg(y)) + sqrt(tanh(x)^2*cos(deg(y))^2 + 1/\lam^2 * (1/\rho^2-1))))^2/\boxh^2 ) },
        v = {(\lam*\rho*((2/tanh(x)+tanh(x))*cos(deg(y)) + sqrt(tanh(x)^2*cos(deg(y))^2 + 1/\lam^2 * (1/\rho^2-1)))) / sqrt( (2*\lam*\rho*sin(deg(y)))^2/(\boxw^2) + (\lam*\rho*((2/tanh(x)+tanh(x))*cos(deg(y)) + sqrt(tanh(x)^2*cos(deg(y))^2 + 1/\lam^2 * (1/\rho^2-1))))^2/\boxh^2 ) },
        scale arrows = 0.01
      }, -stealth, very thick] 
      table[
        x index=14,
        y index=15,
        col sep=comma,
      ] {anc/PW_deltapositive_rhosupercritical_arrows.csv};

\end{axis}
\end{tikzpicture}
}
\qquad
\subfloat[$\rho_0 < \rho_\crit$]{
\begin{tikzpicture}

\pgfmathsetmacro{\lam}{1}
\pgfmathsetmacro{\rho}{1/(sqrt(4*\lam^2*(1/tanh(0.3)^2+1)+1))}
\pgfmathsetmacro{\boxw}{0.9}
\pgfmathsetmacro{\boxh}{2*pi}

\begin{axis}
[
    xmin = 0, xmax = \boxw,
    ymin = 0, ymax = \boxh,
    width = 0.45\textwidth,
    height = 0.45\textwidth,
    xlabel = {$P$},
    ylabel = {$W$},
    xtick = {0, 0.3, 1},
    xticklabels= {0, $P_\infty$, },
    ytick = {-2*pi,-3*pi/2,-pi,-pi/2,0,pi/2,pi,3*pi/2,2*pi},
    yticklabels = {$-2\pi$, $-\frac{3\pi}{2}$, $-\pi$, $-\frac{\pi}{2}$, $0$, $\frac{\pi}{2}$, $\pi$, $\frac{3\pi}{2}$, $2\pi$},
    view = {0}{90}
]

\addplot3[
    quiver = {
        u = { 2*\lam*\rho*sin(deg(y))) / sqrt( (2*\lam*\rho*sin(deg(y)))^2/(\boxw^2) + (\lam*\rho*((2/tanh(x)+tanh(x))*cos(deg(y)) + sqrt(tanh(x)^2*cos(deg(y))^2 + 1/\lam^2 * (1/\rho^2-1))))^2/\boxh^2 ) },
        v = {(\lam*\rho*((2/tanh(x)+tanh(x))*cos(deg(y)) + sqrt(tanh(x)^2*cos(deg(y))^2 + 1/\lam^2 * (1/\rho^2-1)))) / sqrt( (2*\lam*\rho*sin(deg(y)))^2/(\boxw^2) + (\lam*\rho*((2/tanh(x)+tanh(x))*cos(deg(y)) + sqrt(tanh(x)^2*cos(deg(y))^2 + 1/\lam^2 * (1/\rho^2-1))))^2/\boxh^2 ) },
        scale arrows = 0.05
    },
    -stealth,
    domain = 1/30:\boxw,
    samples = 10,
    domain y = 0:\boxh,
    lightgray] 
{0};

%


\filldraw ({artanh(1/sqrt(1 + 1/(4*\lam^2)*(1/\rho^2 - (1+8*\lam^2))))}, pi) circle (2pt);
\filldraw ({artanh(1/sqrt(1 + 1/(4*\lam^2)*(1/\rho^2 - (1+8*\lam^2))))}, -pi) circle (2pt);

\addplot[color=cyan,smooth,tension=0.7,thick] table [x index=0,y index=1,col sep=comma] {anc/PW_deltapositive_rhosubcritical.csv};
\addplot[cyan, quiver={
        u = { 2*\lam*\rho*sin(deg(y))) / sqrt( (2*\lam*\rho*sin(deg(y)))^2/(\boxw^2) + (\lam*\rho*((2/tanh(x)+tanh(x))*cos(deg(y)) + sqrt(tanh(x)^2*cos(deg(y))^2 + 1/\lam^2 * (1/\rho^2-1))))^2/\boxh^2 ) },
        v = {(\lam*\rho*((2/tanh(x)+tanh(x))*cos(deg(y)) + sqrt(tanh(x)^2*cos(deg(y))^2 + 1/\lam^2 * (1/\rho^2-1)))) / sqrt( (2*\lam*\rho*sin(deg(y)))^2/(\boxw^2) + (\lam*\rho*((2/tanh(x)+tanh(x))*cos(deg(y)) + sqrt(tanh(x)^2*cos(deg(y))^2 + 1/\lam^2 * (1/\rho^2-1))))^2/\boxh^2 ) },
        scale arrows = 0.01
      }, -stealth, very thick] 
      table[
        x index=0,
        y index=1,
        col sep=comma,
      ] {anc/PW_deltapositive_rhosubcritical_arrows.csv};

\addplot[color=cyan,smooth,tension=0.7,thick] table [x index=2,y index=3,col sep=comma] {anc/PW_deltapositive_rhosubcritical.csv};
\addplot[cyan, quiver={
        u = { 2*\lam*\rho*sin(deg(y))) / sqrt( (2*\lam*\rho*sin(deg(y)))^2/(\boxw^2) + (\lam*\rho*((2/tanh(x)+tanh(x))*cos(deg(y)) + sqrt(tanh(x)^2*cos(deg(y))^2 + 1/\lam^2 * (1/\rho^2-1))))^2/\boxh^2 ) },
        v = {(\lam*\rho*((2/tanh(x)+tanh(x))*cos(deg(y)) + sqrt(tanh(x)^2*cos(deg(y))^2 + 1/\lam^2 * (1/\rho^2-1)))) / sqrt( (2*\lam*\rho*sin(deg(y)))^2/(\boxw^2) + (\lam*\rho*((2/tanh(x)+tanh(x))*cos(deg(y)) + sqrt(tanh(x)^2*cos(deg(y))^2 + 1/\lam^2 * (1/\rho^2-1))))^2/\boxh^2 ) },
        scale arrows = 0.01
      }, -stealth, very thick] 
      table[
        x index=2,
        y index=3,
        col sep=comma,
      ] {anc/PW_deltapositive_rhosubcritical_arrows.csv};

\addplot[color=teal,smooth,tension=0.7,thick] table [x index=4,y index=5,col sep=comma] {anc/PW_deltapositive_rhosubcritical.csv};
\addplot[teal, quiver={
        u = { 2*\lam*\rho*sin(deg(y))) / sqrt( (2*\lam*\rho*sin(deg(y)))^2/(\boxw^2) + (\lam*\rho*((2/tanh(x)+tanh(x))*cos(deg(y)) + sqrt(tanh(x)^2*cos(deg(y))^2 + 1/\lam^2 * (1/\rho^2-1))))^2/\boxh^2 ) },
        v = {(\lam*\rho*((2/tanh(x)+tanh(x))*cos(deg(y)) + sqrt(tanh(x)^2*cos(deg(y))^2 + 1/\lam^2 * (1/\rho^2-1)))) / sqrt( (2*\lam*\rho*sin(deg(y)))^2/(\boxw^2) + (\lam*\rho*((2/tanh(x)+tanh(x))*cos(deg(y)) + sqrt(tanh(x)^2*cos(deg(y))^2 + 1/\lam^2 * (1/\rho^2-1))))^2/\boxh^2 ) },
        scale arrows = 0.01
      }, -stealth, very thick] 
      table[
        x index=4,
        y index=5,
        col sep=comma,
      ] {anc/PW_deltapositive_rhosubcritical_arrows.csv};

\addplot[color=teal,smooth,tension=0.7,thick] table [x index=6,y index=7,col sep=comma] {anc/PW_deltapositive_rhosubcritical.csv};
\addplot[teal, quiver={
        u = { 2*\lam*\rho*sin(deg(y))) / sqrt( (2*\lam*\rho*sin(deg(y)))^2/(\boxw^2) + (\lam*\rho*((2/tanh(x)+tanh(x))*cos(deg(y)) + sqrt(tanh(x)^2*cos(deg(y))^2 + 1/\lam^2 * (1/\rho^2-1))))^2/\boxh^2 ) },
        v = {(\lam*\rho*((2/tanh(x)+tanh(x))*cos(deg(y)) + sqrt(tanh(x)^2*cos(deg(y))^2 + 1/\lam^2 * (1/\rho^2-1)))) / sqrt( (2*\lam*\rho*sin(deg(y)))^2/(\boxw^2) + (\lam*\rho*((2/tanh(x)+tanh(x))*cos(deg(y)) + sqrt(tanh(x)^2*cos(deg(y))^2 + 1/\lam^2 * (1/\rho^2-1))))^2/\boxh^2 ) },
        scale arrows = 0.01
      }, -stealth, very thick] 
      table[
        x index=6,
        y index=7,
        col sep=comma,
      ] {anc/PW_deltapositive_rhosubcritical_arrows.csv};

\addplot[color=red,smooth,tension=0.7,thick] table [x index=8,y index=9,col sep=comma] {anc/PW_deltapositive_rhosubcritical.csv};
\addplot[red, quiver={
        u = { 2*\lam*\rho*sin(deg(y))) / sqrt( (2*\lam*\rho*sin(deg(y)))^2/(\boxw^2) + (\lam*\rho*((2/tanh(x)+tanh(x))*cos(deg(y)) + sqrt(tanh(x)^2*cos(deg(y))^2 + 1/\lam^2 * (1/\rho^2-1))))^2/\boxh^2 ) },
        v = {(\lam*\rho*((2/tanh(x)+tanh(x))*cos(deg(y)) + sqrt(tanh(x)^2*cos(deg(y))^2 + 1/\lam^2 * (1/\rho^2-1)))) / sqrt( (2*\lam*\rho*sin(deg(y)))^2/(\boxw^2) + (\lam*\rho*((2/tanh(x)+tanh(x))*cos(deg(y)) + sqrt(tanh(x)^2*cos(deg(y))^2 + 1/\lam^2 * (1/\rho^2-1))))^2/\boxh^2 ) },
        scale arrows = 0.01
      }, -stealth, very thick] 
      table[
        x index=8,
        y index=9,
        col sep=comma,
      ] {anc/PW_deltapositive_rhosubcritical_arrows.csv};

\end{axis}
\end{tikzpicture}
}
}
\caption{Phase portrait of \eqref{eq-P-W-system-delta0>0-rho0<1}.}
\label{fig-P-W-phase-delta0>0}
\end{figure}

We now let $\rho_\crit = (1+8\lambda^2)^{-\frac12}$ and consider two cases separately (cf.\ also Figure \ref{fig-P-W-phase-delta0>0}).

\begin{enumerate}[itemsep=0.1cm]
\item If $\rho_\crit \leq \rho_0 < 1$, then the system has no fixed points with $P>0$. 
Consider the unstable manifold $\mathscr{U}_+$ emanating from $(0,\pi/2)$. Note that it stays in the region $W > \pi/2$ in the forward direction, since $W' = \lambda\rho_0\beta > 0$ whenever $W = \frac{\pi}{2}$. 
In view of the symmetry $(s,W) \mapsto (-s, 2\pi - W)$, the stable manifold $\mathscr{S}_+$ must then stay in the region $W < 3\pi/2$.
If $\mathscr{S}_+$ and $\mathscr{U}_+$ intersect, then they must do so at $W=\pi$ again by the reflection symmetry, and in this case they would necessarily coincide as sets by the autonomy of the system. However, this curve together with the axis $P=0$ would then enclose a compact set, and in particular the vector field $X$ restricted to this set would necessarily have a zero with $P>0$, by a combination of the Poincaré-Bendixson and the Brouwer fixed point theorem, which is impossible for $\rho_0 \geq \rho_\crit$. 
It follows that $\mathscr{S}_+$ and $\mathscr{U}_+$ do not intersect, and in fact $\mathscr{U}_+$ stays (in the forward direction) in the region $\pi/2 < W < \pi$ while $\mathscr{S}_+$ stays (in the backward direction) in the region $\pi < W < 3\pi/2$.

As for the other orbits, none of them can intersect $\mathscr{S}_+$ or $\mathscr{U}_+$, nor can they have $P\to0$.
It follows that orbits with $\pi/2 < W_0 < 3\pi/2$ lying between $\mathscr{S}_+$ and $\mathscr{U}_+$ must all intersect $W=\pi$ at some point, while orbits with $0 < W_0 < \pi/2$ lying below $\mathscr{U}_+$ resp.\ orbits with $3\pi/2 < W_0 < 2\pi$ lying above $\mathscr{S}_+$ must intersect the line $W=2\pi k$ for some integer $k$.
In particular, we can without loss of generality (up to translation of $s$ and $W$ mod $2\pi$) assume that $W_0 \in \{0,\pi\}$.
By the reflection symmetry it also suffices to study them only in the forward direction.
In fact, $W' < 0$ at $W = 0$ and $W' > 0$ at $W = \pi$, so that all such orbits must stay in the region $0 < W < \pi$ for $s > 0$.
In this region $P' > 0$, so that necessarily $P \to \infty$ as the orbit cannot stop existing at any given finite $P$ since $W$ is bounded and there are no singularities for $P>0$. For large $P$, both $\coth P$ and $\tanh P$ are close to 1, so that $W$ is well-approximated by the solutions from the $\delta_0=0$ case. In particular, a similar analysis as in \ref{subsubsec-delta=0-constant-rho<1} shows that all forward orbits (including $\mathscr{U}_+$) asymptotically have $W \to W_\infty$ as $s\to\infty$, where $W_\infty$ is the same as in \eqref{eq-W-infty} for $\rho_0 > \rho_\crit$, and $W_\infty = \pi$ for $\rho_0 = \rho_\crit$.

\item If $\rho_0 < \rho_\crit$, the system has a fixed point at
\begin{equation}\label{eq-W-P-explicit-delta0>0-rho0<1}
    (P,W) = (P_\infty, \pi), \qquad 
    P_\infty = \arcoth\sqrt{1 + \frac{1}{4\lambda^2}\left(\frac{1}{\rho_0^2}-\frac{1}{\rho_\crit^2}\right)}.
\end{equation}
This provides an interesting Dirac-Yang-Mills pair which we will come back to later.

The aforementioned reflection symmetry of the system about $W = \pi$ maps $\mathscr{S}_+$ onto $\mathscr{U}_+$ and vise-versa. We now show that  $\mathscr{U}_+$ intersects $W = \pi$ in $(P_\crit, \pi)$ for some $P_\crit > P_\infty$ and hence $\mathscr{S}_+ = \mathscr{U}_+$.
To see this consider the orbit of \eqref{eq-P-W-system-delta0>0-rho0<1} with $P_0 > P_\infty$ and $0 < W_0 < \pi$. Then while $W < \pi$ we have $C_0 < W'$ for some constant $C_0 > 0$ depending on $P_0$ and $0 < P' < 2\lambda \rho_0$. Hence $W(s) = \pi$ for some finite $s>0$. 
Now $\mathscr{U}_+$ cannot intersect $W = \pi$ at a point with $P < P_\infty$ since at these points $P'<0$. Furthermore clearly $P_\crit \neq P_\infty$. Hence $\mathscr{U}_+$ must intersect the region $(P_\infty, \infty) \times (0,\pi)$. Thus, any orbit starting in $\mathscr{U}_+$ intersects $W = \pi$ and hence, so does $\mathscr{U}_+$.
A straightforward argument shows that any orbit of \eqref{eq-P-W-system-delta0>0-rho0<1} starting in $\mathscr{U}_+$ must hit $P = 0$ at finite $s$ in both the backward and forward direction.

The curve $\mathscr{U}_+$ divides the strip $\R_{P>0} \times [0,2\pi]$ into a bounded component $D_b$ containing $(0,P_\crit)\times \{\pi\}$ and an unbounded component $D_u$ containing $(P_\crit,\infty) \times \{\pi\}$. Any orbit starting in $D_u$ must be contained in $D_u$ and likewise for $D_b$.

For orbits starting in $D_u$, $P>0$ and, since $D_b$ contains the set of points with $W'\leq 0$, $W'(s)>C_0>0$ for all $s$ for some $C_0$ depending on $(P_0,W_0)$. It follows that $W(s) = \pi$ for some $s$. Assuming $W_0 = \pi$ a similar argument shows $W(s_0) =2\pi$ for some $s_0$. The reflection symmetry about $W = \pi$ then gives $W(s + T) = W(s) + 2\pi$ and $P(s+T) = P(s)$ with $T = 2s_0$ and for all $s \in \R$.

For an orbit $(P,W)$ starting in $D_b$ an application of the Poincaré-Bendixson Theorem and the reflection symmetry about $W = \pi$ gives that either $(P,W) \equiv (P_\infty,\pi)$ or $(P,W)$ is a periodic orbit circling $(P_\infty,\pi)$.

The associated Dirac-Yang-Mills pair $(z,\xi)$ is periodic if and only if
\begin{equation}
\label{eq-f1-f2-rational-periodic}
\begin{split}
    f_1(\rho_0,P_0) := \frac{\lambda \rho_0}{2\pi}\int_0^T \tanh\frac{P(t)}{2} \cos W(t) dt = \frac{p_1}{q_1},\\
    f_2(\rho_0,P_0) := \frac{\lambda \rho_0}{2\pi}\int_0^T \coth\frac{P(t)}{2} \cos W(t) dt = \frac{p_2}{q_2},
\end{split}
\end{equation}
for some rational numbers $\frac{p_1}{q_1}$ and $\frac{p_2}{q_2}$.
We shall show that this holds for countably many pairs $(\rho_0,P_0)$ with $0<\rho_0<\rho_\crit$ and $P_0 \neq P_{\infty}, P_\crit$.

We first show that, away from $P_0 \in \{P_{\infty}, P_\crit\}$, $f_1$ and $f_2$ depend continuously on $P_0$ and $\rho_0$. 
In the case $P_0 > P_\crit$ this follows by similar arguments as in the $\delta_0 = 0$ case. In the case $P_0<P_\crit$, we note that it is sufficient to show that the period $T = T(P_0,\rho_0)$ is continuous in $P_0$ and $\rho_0$. For fixed $0<\rho_0<\rho_\crit$ we denote by $y_{{P_0},\rho_0}(s) = (P_{{P_0},\rho_0}(s),W_{{P_0},\rho_0}(s))$ the orbit of the system \eqref{eq-P-W-system-delta0>0-rho0<1} with $P(0) = {P_0}$ and $W(0) = \pi$. For $0 < {P_0} < P_\crit$ we may without loss of generality take $0<{P_0}<P_\infty$. The orbit $y_{{P_0},\rho_0}$ intersects the line $W = \pi$ in exactly two points, one of which is $({P_0},\pi)$ and the other we denote by $(c({P_0},\rho_0),\pi)$.

An application of the principle of continuous dependence on parameters yields that $c({P_0},\rho_0)$ is continuous in ${P_0}$ and $\rho_0$ on $\{({P_0},\rho_0) \ | \ 0 < \rho_0 < \rho_\crit, 0 < {P_0} \leq P_\infty(\rho_0)\}$. Denote by $T({P_0},\rho_0)$ the minimal period of $(P_{{P_0},\rho_0}(s),W_{{P_0},\rho_0}(s))$. Then $\restr{y_{{P_0},\rho_0}}{(0,T({P_0},\rho_0))}$ is a homeomorphism onto its image and hence
\begin{equation*}
    (\restr{y_{{P_0},\rho_0}}{(0,T({P_0},\rho_0))})^{-1}(c({P_0},\rho_0),\pi) = \frac{T({P_0},\rho_0)}{2}    
\end{equation*}
depends continuously on ${P_0}$ and $\rho_0$.

For the periodicity conditions \eqref{eq-f1-f2-rational-periodic}, we will prove they are satisfied for countably many pairs $(P_0,\rho_0)$ with $P_0 \in D_u$. For the case $P_0 \in D_b$ the plots in Figure \ref{fig: periodicity in bounded case} together with continuity of $f_1$ and $f_2$ provides strong numerical evidence that the periodicity conditions are satisfied also for countably many pairs $(\rho_0,P_0)$ with $P_0 \in D_b$.

If $P_0 \in D_u$ we can without loss of generality instead take the inital values to be $W_0 = 0$ and $P_0 > 0$. In particular, along an orbit the minimum and maximum values of $P$ are $P_\mathrm{min} = P_0$ and $P_\mathrm{max} = P(\frac{T}{2}) > P_\crit$, respectively.

We have
\begin{equation*}
    f_2 - f_1 = \frac{\lambda \rho_0}{\pi} \int_0^T \csch P(t) \cos W(t)\, \diff t
\end{equation*}
so that
\begin{equation*}
    \frac{\lambda \rho_0}{\pi} \csch (P_0) \left|\int_0^T \cos W(t)\, \diff t \right| \leq  \lvert f_2 - f_1\rvert \leq \frac{\lambda \rho_0}{\pi} \csch (P_\mathrm{max}) \left|\int_0^T \cos W(t)\, \diff t \right|.
\end{equation*}
On the other hand, 
\begin{equation*}
    \frac{\lambda \rho_0}{2\pi} \int_0^T \cos W(t)\, \diff t \to \frac{\lambda \rho_0}{2\pi} \int_0^{\widetilde{T}} \cos \widetilde{W}(t)\, \diff t,
    \quad\text{as}\quad
    P_0 \to \infty,
\end{equation*}
where $\widetilde{W}(t)$ is the solution to \eqref{eq-constant-rho-W'} with $W_0 = 0$ and $\widetilde{T} = \widetilde{T}(\rho_0)$ the associated period.
Then $|f_2 -f_1| \to 0$ as $P_0 \to \infty$, while for each finite $P_0$ we have $|f_2 -f_1| > 0$.

Meanwhile $f_1$ (and hence also $f_2$) satsifies 
\begin{equation*}
    f_1(\rho_0,P_0) \to  \frac{\lambda\rho_0}{2\pi} \int_{0}^{\widetilde{T}} \cos \widetilde{W}(t) \, \diff t \quad\textrm{as}\quad P_0 \to \infty.
\end{equation*}
By the same arguments as in the $\delta_0 = 0$ case this limit is a continouos non-constant function of $\rho_0$. This is enough to conclude that the image of the set of admissible $(\rho_0,P_0)$ under the map $(f_1,f_2)$ contains some open subset of $\R^2$. Since $\Q^2$ is dense in $\R^2$ there are then countably many choices of $(\rho_0,P_0)$ for which $(f_1,f_2) \in \Q^2$.

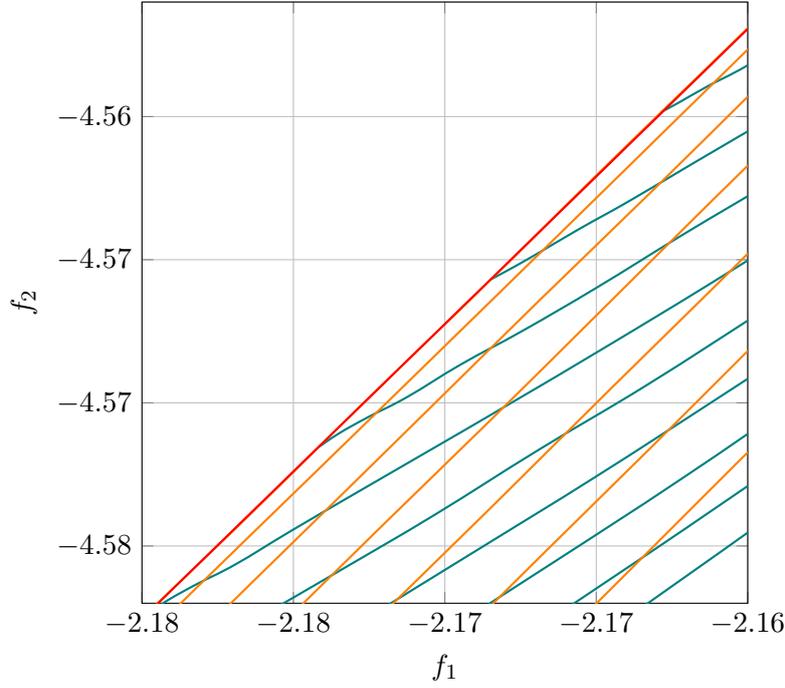
\begin{figure}[t]

\begin{tikzpicture}
\begin{axis}
[
    grid=both,
    xmin = -2.18, xmax =-2.160 ,
    ymin = -4.577, ymax = -4.556,
    width = 0.6\textwidth,
    height = 0.6\textwidth,
    xlabel = {$f_1$},
    ylabel = {$f_2$},
    xtick = {-2.18,-2.175,-2.170,-2.165,-2.16},
    view = {0}{90}
] 

\addplot[color=teal,smooth,tension=0.7,thick] table [x index=0,y index=1,col sep=comma] {anc/f_curves.csv};
\addplot[color=teal,smooth,tension=0.7,thick] table [x index=2,y index=3,col sep=comma] {anc/f_curves.csv};
\addplot[color=teal,smooth,tension=0.7,thick] table [x index=4,y index=5,col sep=comma] {anc/f_curves.csv};
\addplot[color=teal,smooth,tension=0.7,thick] table [x index=6,y index=7,col sep=comma] {anc/f_curves.csv};
\addplot[color=teal,smooth,tension=0.7,thick] table [x index=8,y index=9,col sep=comma] {anc/f_curves.csv};
\addplot[color=teal,smooth,tension=0.7,thick] table [x index=10,y index=11,col sep=comma] {anc/f_curves.csv};
\addplot[color=teal,smooth,tension=0.7,thick] table [x index=12,y index=13,col sep=comma] {anc/f_curves.csv};
\addplot[color=teal,smooth,tension=0.7,thick] table [x index=14,y index=15,col sep=comma] {anc/f_curves.csv};
\addplot[color=teal,smooth,tension=0.7,thick] table [x index=16,y index=17,col sep=comma] {anc/f_curves.csv};
\addplot[color=teal,smooth,tension=0.7,thick] table [x index=18,y index=19,col sep=comma] {anc/f_curves.csv};
\addplot[color=teal,smooth,tension=0.7,thick] table [x index=20,y index=21,col sep=comma] {anc/f_curves.csv};
\addplot[color=teal,smooth,tension=0.7,thick] table [x index=22,y index=23,col sep=comma] {anc/f_curves.csv};
\addplot[color=orange,smooth,tension=0.7,thick] table [x index=24,y index=25,col sep=comma] {anc/f_curves.csv};
\addplot[color=orange,smooth,tension=0.7,thick] table [x index=26,y index=27,col sep=comma] {anc/f_curves.csv};
\addplot[color=orange,smooth,tension=0.7,thick] table [x index=28,y index=29,col sep=comma] {anc/f_curves.csv};
\addplot[color=orange,smooth,tension=0.7,thick] table [x index=30,y index=31,col sep=comma] {anc/f_curves.csv};
\addplot[color=orange,smooth,tension=0.7,thick] table [x index=32,y index=33,col sep=comma] {anc/f_curves.csv};
\addplot[color=orange,smooth,tension=0.7,thick] table [x index=34,y index=35,col sep=comma] {anc/f_curves.csv};
\addplot[color=orange,smooth,tension=0.7,thick] table [x index=36,y index=37,col sep=comma] {anc/f_curves.csv};
\addplot[color=orange,smooth,tension=0.7,thick] table [x index=38,y index=39,col sep=comma] {anc/f_curves.csv};
\addplot[color=red,smooth,tension=0.7,thick] table [x index=40,y index=41,col sep=comma] {anc/f_curves.csv};
\end{axis}
\end{tikzpicture}
\caption{$f_1$ and $f_2$ as functions of $\rho_0$ for fixed $P_0$ (orange curves) and as functions of  $\rho_0$ for fixed $P_0$ (teal). The red line shows the limiting case $P_0 = P_\infty$ with the period equal to that of the linearized system.}
\label{fig: periodicity in bounded case}
\end{figure}

\end{enumerate}

Finally, we study the fixed point $(P,W) \equiv (P_\infty, \pi)$ more explicitly. In this case, we have $\arg\QuatIm(\xi_0) = \pi - \arg\QuatRe(\xi_0)$, and the associated Dirac-Yang-Mills pair is given by
\begin{align*}
    z(s) &= \rho_0 e^{-2i\lambda\rho_0 \coth P_0\, s},\\
     \QuatRe(\xi)(s) &= \delta_0\cosh \frac{P_0}{2} \,\exp\left( i\left[- \lambda \rho_0 \tanh \frac{P_0}{2} \, s + \arg\QuatRe(\xi_0)\right] \right),\\
    \QuatIm(\xi)(s) &= -\delta_0\sinh \frac{P_0}{2} \,\exp\left( i\left[ - \lambda \rho_0 \coth \frac{P_0}{2} \, s - \arg\QuatRe(\xi_0)\right] \right),
\end{align*}
with respect to the metric coefficient
\begin{equation*}
    r \equiv \frac{4\rho_0^3}{\delta_0^2} \sqrt{ \frac{1}{\rho_0^2} - \frac{1}{\rho_\crit^2}}.
\end{equation*}
Now $(z,\xi)$ will descend to a solution on $(\Sph^1 \times \Sph^2)_r$ if and only if $z$, $\QuatRe(\xi)$, and $\QuatIm(\xi)$ have a common period, which is the case if and only if there exists a non-zero real number $a \in \R$ such that 
\begin{equation*}
    2a\coth P_0, \; a\tanh\frac{P_0}{2}, \; a\coth\frac{P_0}{2} \in \mathbb{Z}.
\end{equation*}
Since $2\coth P_0 = \coth\frac{P_0}{2}+\tanh\frac{P_0}{2}$, it suffices to ensure that $a\tanh \frac{P_0}{2} = p$ and $a\coth \frac{P_0}{2} = q$ are integers, which we assume to be coprime in order to extract the minimal period.
This is the case if and only if $\tanh^2 \frac{P_0}{2} = \frac{p}{q}$,%
\footnote{This is clearly necessary since $ax = p$ and $a/x = q$ together imply that $x^2 = \frac{p}{q} \in \mathbb{Q}$. Conversely, if $x^2 = \frac{p}{q}$ for coprime integers $p$ and $q$, then it is sufficient to set $a = q x = p/x = \sqrt{pq}$, which recovers $ax = p$ and $a/x = q$. In fact, $a=\sqrt{pq}$ is then the smallest positive real number with the property that $ax$ and $a/x$ are coprime integers.}
and from \eqref{eq-W-P-explicit-delta0>0-rho0<1} we see that this is equivalent to
\begin{equation*}
    \frac{1}{\rho_0^2} - \frac{1}{\rho_\crit^2} = \frac{2\lambda^2}{3}\left( \frac{q}{p} - 3 \pm \frac12\sqrt{\frac{q^2}{p^2}-12} \right)
\end{equation*}
provided that $\frac{p}{q} \leq \frac{1}{2\sqrt{3}}$, and for the minus-branch we additionally need $\frac{p}{q} > \frac14$.
Thus we need
\begin{equation*}
    \rho_0 = \rho_0(\lambda,p/q) = \left( 1 + \frac{2\lambda^2}{3}\left( \frac{q}{p} + 9 \pm \frac12\sqrt{\frac{q^2}{p^2}-12} \right) \right)^{-\frac12}.
\end{equation*}
If we define the coordinate
\begin{equation*}
    t = \frac{\lambda\rho_0}{q \tanh \frac{P_0}{2}}\, s = \frac{\lambda\rho \tanh\frac{P_0}{2}}{p}\, s = \frac{\lambda\rho_0}{\sqrt{pq}}\, s,
\end{equation*}
then the solution can be written as
\begin{align*}
    z(t) &= \rho_0(\lambda,p/q) e^{-i(p+q)t}\\
     \QuatRe(\xi)(t) &= \delta_0\sqrt{\frac{q}{q-p}} \,e^{- ipt + i\arg\QuatRe(\xi_0)},\\
    \QuatIm(\xi)(t) &= -\delta_0\sqrt{\frac{p}{q-p}} \,e^{- iqt - i\arg\QuatRe(\xi_0)},
\end{align*}
with $t\mapsto e^{it}$ being the standard coordinate on $\Sph^1$, and we can view this as a solution on the closed Riemannian manifold
\begin{equation*}
    \Sph^1\left(\frac{2\sqrt{pq}}{\lambda\delta_0}\sqrt{1 - \frac{\rho_0^2}{\rho_\crit^2}} \right) \times \Sph^2\left( \frac{2\rho_0}{\delta_0}\sqrt{1 - \frac{\rho_0^2}{\rho_\crit^2}} \right).
\end{equation*}
Here, if desired, one can choose $\delta_0$ so that one of the sphere factors has unit radius.

\section{Solutions with constant \texorpdfstring{$W$}{W}}
\label{sec-constant-W}

We also consider solutions with constant $W \equiv W_0$, which display a different behaviour.

\subsection{\texorpdfstring{The case $\delta_0=0$}{The case delta0 = 0}}
\label{subsec-constant-W-delta=0}

In this case, we prove the following.

\begin{proposition}\label{prop-constant-W-delta=0}
    Using the notation of Proposition \ref{lem-polar-form}, let $\delta_0 = 0$ so that $\xi_0 = \QuatRe(\xi_0)(1+je^{iW_0})$, and assume $W \equiv W_0$ for a constant $W_0$. 
    Then without loss of generality $\pi/2 < W_0 \leq \pi$, and the DYM pair together with the corresponding metric coefficient $r$ is given by
    \begin{align*}
        z(s) &= \rho(s) \exp\left(2i \lambda\cos(W_0)\int_0^s \rho(t)\, \diff t\right), \\ 
        \xi(s) &= \xi_0 \exp\left(i\lambda e^{-iW_0} \int_0^s \rho(t)\, \diff t\right),\\
         r(s) &= - \frac{8\lambda\cos(W_0)}{|\xi_0|^2} \rho(s)^3 \exp\left(-2\lambda \sin(W_0) \int_0^s \rho(t)\,\diff t\right),
    \end{align*}
    where $\rho$ belongs one to the following classes (denoting $\rho_\crit = (1+8\lambda^2)^{-\frac12}$):
    \begin{enumerate}
        \item If $\rho_0 = \rho_\crit$ and $U_0=0$, then the solution has constant $\rho \equiv \rho_\crit$ and corresponds to the one already constructed in Proposition \ref{prop-constant-rho-delta=0}.
        \item There exist distinguished orbits such that $\rho \to \rho_\crit$ as $s \to \infty$ (resp.\ $s \to -\infty$) and $\rho \to 0$ exponentially as $s \to -\infty$ (resp.\ $s\to \infty$).
        \item The solutions with initial data $0 < \rho_0 < \rho_\crit$ and $U_0 = 0$ are global and have $\rho \to 0$ exponentially as $s\to\pm \infty$. In particular, $\rho$ is integrable so the DYM pair $(z,\xi)$ has finite limits as $s\to\pm\infty$ (in fact $z \to 0$), and the corresponding metric coefficient also has $r \to 0$ exponentially as $s \to \pm\infty$.
        \item All other orbits (i.e.\ not coinciding with any of the previous cases) have $\rho \to \infty$ at some finite $s$ (either forwards or backwards).
    \end{enumerate}
\end{proposition}

\begin{remark}
    Here, (iii) is the arguably most interesting case, since it has $r \to 0$ exponentially while the DYM pair has finite limits. This suggests that these solutions might be extendible to global solutions on a closed manifold topologically equivalent to $\Sph^3$.
\end{remark}

To prove the result, assume $W \equiv W_0$, so that \eqref{eq-W-polar-delta=0} reads
\begin{equation}\label{eq-W-fixed-W}
    \frac{rR^2}{4\rho^3} = - \lambda \cos(W_0),
\end{equation}
which forces $\cos(W_0) < 0$, i.e.\ $\pi/2 < |W_0| \leq \pi$ (as we are excluding the case $R_0=0$).
We also see from \eqref{eq-R-polar-delta=0} that
\begin{equation*}
    R(s) = \frac{1}{\sqrt{2}}|\xi_0|\exp\left(\lambda \sin(W_0) \int_0^s \rho\right),
\end{equation*}
so that \eqref{eq-W-fixed-W} implies $r$ and $\rho$ are related by
\begin{equation}\label{eq-r-constant-W-delta0=0}
    r(s) = - \frac{8\lambda\cos(W_0)}{|\xi_0|^2} \rho(s)^3 \exp\left(-2\lambda \sin(W_0) \int_0^s \rho\right)
\end{equation}
On the other hand, by (\ref{eq-rho-polar-delta=0}, \ref{eq-H-polar-delta=0}) we get
\begin{align}
    \nonumber
    0 &= \rho'' - \frac{r'}{r} \rho' + \rho(1-\rho^2) + 2\lambda\cos(W_0) rR^2 \\
    \label{eq-rho-delta=0-constant-W}
    &= \rho'' - \left(\frac{3\rho'}{\rho} - 2\lambda\sin(W_0) \rho \right) \rho' + \rho\left[1-(1+8\lambda^2\cos^2(W_0))\rho^2\right],
\end{align}
which is now an autonomous second order equation for $\rho$.
The associated Dirac-Yang-Mills pair can then be expressed as
\begin{equation*}
    z(s) = \rho(s) \exp\left(2i \lambda\cos(W_0)\int_0^s \rho(t)\, \diff t\right), \qquad 
    \xi(s) = \xi_0 \exp\left(i\lambda e^{-iW_0} \int_0^s \rho(t)\, \diff t\right),
\end{equation*}
so it suffices to study the equation \eqref{eq-rho-delta=0-constant-W} for $\rho$.
We can simplify the equation by letting $x = \rho^{-2}$, $\beta = \lambda\sin(W_0)$, and $\alpha = 1+8\lambda^2\cos^2(W_0) = 1+8(\lambda^2-\beta^2) > 1$, so that
\begin{equation}\label{eq-beta-delta=0}
    x'' + 2\beta x^{-\frac12}x' - 2(x - \alpha) = 0.
\end{equation}
Observe that this has a constant solution $x \equiv \alpha$ for any choice of $W_0$ satisfying $\pi/2 < |W_0| \leq \pi$. This solution corresponds to the one from Proposition \ref{prop-constant-rho-delta=0} (iii--iv).
Furthermore, we note that reversing the coordinate $s \mapsto -s$ in \eqref{eq-beta-delta=0} has the same effect as reversing the sign of $W_0$. Therefore it suffices to consider initial data with $\pi/2 < W_0 \leq \pi$.

\medskip

We begin by noting that if $W_0=\pi$, we have $\beta=0$ and $\alpha=1+8\lambda^2$, and the equation \eqref{eq-beta-delta=0} can be solved explicitly to get
\begin{equation*}
    x(s) = (x_0-\alpha) \cosh(\sqrt{2}s) + \frac{x_0'}{\sqrt{2}} \sinh(\sqrt{2}s) + \alpha
\end{equation*}
or
\begin{equation*}
    \rho(s) = \rho_0\left[\left(1-\frac{\rho_0^2}{\rho_\crit^2}\right) \cosh(\sqrt{2}s) - \frac{r_0U_0}{\sqrt{2}\rho_0} \sinh(\sqrt{2}s) + \frac{\rho_0^2}{\rho_\crit^2}\right]^{-\frac12},
\end{equation*}
where we set $\rho_\crit = (1+8\lambda^2)^{-\frac12}$ and we recall that $\rho_0'=r_0U_0$, cf.\ \S \ref{subsec-initial-invariances}.
It is easy to see that $\rho$ is globally defined if and only if
\begin{equation*}
    r_0 |U_0| \leq \sqrt{2} \rho_0 \left(1-\frac{\rho_0^2}{\rho_\crit^2}\right).
\end{equation*}
Here it is noteworthy to separate three cases.
\begin{enumerate}
    \item The boundary case $\rho_0 = \rho_\crit$ forces $U_0=0$, and in particular this case corresponds to a special case of the aforementioned constant solution $x \equiv 1+8\lambda^2$, corresponding to Proposition \ref{prop-constant-rho-delta=0} (iv).
    \item In the boundary case $\frac{r_0U_0}{\sqrt{2}\rho_0} = \pm(1-\frac{\rho_0^2}{\rho_\crit^2})$ we can write
    \begin{equation*}
        \rho(s) = \rho_0\left[\left(1-\frac{\rho_0^2}{\rho_\crit^2}\right) e^{\mp\sqrt{2}s} + \frac{\rho_0^2}{\rho_\crit^2}\right]^{-\frac12}.
    \end{equation*}
    These solutions have $\rho \to 0$ as $s\to\mp\infty$ and $\rho \to \rho_\crit$ as $s \to \pm \infty$.
    In particular they correspond to the unstable (resp.\ stable) manifold of the fixed point $(x,x') = (1+8\lambda^2, 0)$.
    One can also express $(z,\xi)$ and $r$ explicitly but the formulas are not particularly illuminating so we omit them.

    \item Finally, if $0 < 1-\frac{\rho_0^2}{\rho_\crit^2} < r_0|U_0|$ we can more simply express $\rho$ as
    \begin{equation*}
        \rho(s) = \rho_0\left[\cosh(\sqrt{2}s + \gamma_0) + \frac{\rho_0^2}{\rho_\crit^2}\right]^{-\frac12},
        \quad
        \gamma_0 = \artanh \frac{r_0U_0}{\sqrt{2} \rho_0 \left(1-\frac{\rho_0^2}{\rho_\crit^2}\right)}.
    \end{equation*}
    One can then express $r$ and $(z,\xi)$ in terms of elliptic integrals of the first kind but the formulas become somewhat unwieldy.
    More importantly, we observe in this case that $\rho$ is integrable, so that $z$ and $\xi$ have well-defined finite limits as $s \to \infty$, and the metric coefficient $r \to 0$ with decay rate $\exp(-\frac{3}{\sqrt{2}}|s|)$.
\end{enumerate}


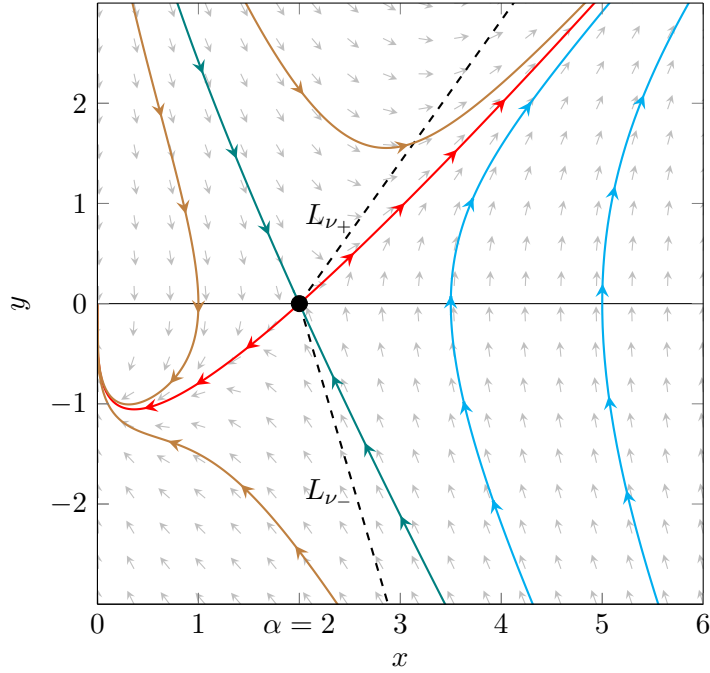
\begin{figure}[t]
\begin{tikzpicture}

\pgfmathsetmacro{\lam}{1}
\pgfmathsetmacro{\K}{2}
\pgfmathsetmacro{\kap}{sqrt(1-(\K-1)/(8*\lam^2))}
\pgfmathsetmacro{\muone}{-\lam*\kap/sqrt(\K) + sqrt((\lam*\kap)^2/\K + 2)}
\pgfmathsetmacro{\mutwo}{-\lam*\kap/sqrt(\K) - sqrt((\lam*\kap)^2/\K + 2)}
\pgfmathsetmacro{\boxw}{6}

\begin{axis}
[
    xmin = 0, xmax = \boxw,
    ymin = -\boxw/2, ymax = \boxw/2,
    width = 0.6\textwidth,
    height = 0.6\textwidth,
    xlabel = {$x$},
    ylabel = {$y$},
    xtick = {0, 1, 2, 3, 4, 5, 6},
    xticklabels = {$0$, $1$, $\alpha=2$, $3$, $4$, $5$, $6$},
    ytick = {-2, -1, 0, 1, 2},
    view = {0}{90}
]

\addplot3[
    quiver = {
        u = {y/(sqrt(y^2+4*(-\lam*\kap*y/sqrt(x) + x - \K)^2))},
        v = {2*(-\lam*\kap*y/sqrt(x) + x - \K)/(sqrt(y^2+4*(-\lam*\kap*y/sqrt(x) + x - \K)^2))},
        scale arrows = 0.15
    },
    -stealth,
    domain = 0.01:\boxw,
    samples = 18 ,
    domain y = -\boxw/2:\boxw/2,
    lightgray] 
{0};


\addplot [
    domain=\K:\boxw, 
    samples=100, 
    color=black,
    dashed,
    thick
]
{sqrt(2)*(x-\K)};

\addplot [
    domain=\K:\boxw, 
    samples=100, 
    color=black,
    dashed,
    thick
]
{-(2+sqrt(2))*(x-\K)};

\addplot[mark=none,black,domain=0:\boxw] {0};

\addplot[color=red,smooth,tension=0.7,thick] table [x index=0,y index=1,col sep=comma] {anc/constant_W_delta=0.csv};
\addplot[red, quiver={
        u = {y/(sqrt(y^2+4*(-\lam*\kap*y/sqrt(x) + x - \K)^2))},
        v = {2*(-\lam*\kap*y/sqrt(x) + x - \K)/(sqrt(y^2+4*(-\lam*\kap*y/sqrt(x) + x - \K)^2))},
        scale arrows = 0.05
      }, -stealth, very thick] 
      table[
        x index=0,
        y index=1,
        col sep=comma,
      ] {anc/constant_W_delta=0_arrows.csv};

\addplot[color=teal,smooth,tension=0.7,thick] table [x index=2,y index=3,col sep=comma] {anc/constant_W_delta=0.csv};
\addplot[teal, quiver={
        u = {y/(sqrt(y^2+4*(-\lam*\kap*y/sqrt(x) + x - \K)^2))},
        v = {2*(-\lam*\kap*y/sqrt(x) + x - \K)/(sqrt(y^2+4*(-\lam*\kap*y/sqrt(x) + x - \K)^2))},
        scale arrows = 0.05
      }, -stealth, very thick] 
      table[
        x index=2,
        y index=3,
        col sep=comma,
      ] {anc/constant_W_delta=0_arrows.csv};

\addplot[color=cyan,smooth,tension=0.7,thick] table [x index=4,y index=5,col sep=comma] {anc/constant_W_delta=0.csv};
\addplot[cyan, quiver={
        u = {y/(sqrt(y^2+4*(-\lam*\kap*y/sqrt(x) + x - \K)^2))},
        v = {2*(-\lam*\kap*y/sqrt(x) + x - \K)/(sqrt(y^2+4*(-\lam*\kap*y/sqrt(x) + x - \K)^2))},
        scale arrows = 0.01
      }, -stealth, very thick] 
      table[
        x index=4,
        y index=5,
        col sep=comma,
      ] {anc/constant_W_delta=0_arrows.csv};

\addplot[color=cyan,smooth,tension=0.7,thick] table [x index=6,y index=7,col sep=comma] {anc/constant_W_delta=0.csv};
\addplot[cyan, quiver={
        u = {y/(sqrt(y^2+4*(-\lam*\kap*y/sqrt(x) + x - \K)^2))},
        v = {2*(-\lam*\kap*y/sqrt(x) + x - \K)/(sqrt(y^2+4*(-\lam*\kap*y/sqrt(x) + x - \K)^2))},
        scale arrows = 0.05
      }, -stealth, very thick] 
      table[
        x index=6,
        y index=7,
        col sep=comma,
      ] {anc/constant_W_delta=0_arrows.csv};

\addplot[color=brown,smooth,tension=0.7,thick] table [x index=8,y index=9,col sep=comma] {anc/constant_W_delta=0.csv};
\addplot[brown, quiver={
        u = {y/(sqrt(y^2+4*(-\lam*\kap*y/sqrt(x) + x - \K)^2))},
        v = {2*(-\lam*\kap*y/sqrt(x) + x - \K)/(sqrt(y^2+4*(-\lam*\kap*y/sqrt(x) + x - \K)^2))},
        scale arrows = 0.05
      }, -stealth, very thick] 
      table[
        x index=8,
        y index=9,
        col sep=comma,
      ] {anc/constant_W_delta=0_arrows.csv};

\addplot[color=brown,smooth,tension=0.7,thick] table [x index=10,y index=11,col sep=comma] {anc/constant_W_delta=0.csv};
\addplot[brown, quiver={
        u = {y/(sqrt(y^2+4*(-\lam*\kap*y/sqrt(x) + x - \K)^2))},
        v = {2*(-\lam*\kap*y/sqrt(x) + x - \K)/(sqrt(y^2+4*(-\lam*\kap*y/sqrt(x) + x - \K)^2))},
        scale arrows = 0.05
      }, -stealth, very thick] 
      table[
        x index=10,
        y index=11,
        col sep=comma,
      ] {anc/constant_W_delta=0_arrows.csv};

\addplot[color=brown,smooth,tension=0.7,thick] table [x index=12,y index=13,col sep=comma] {anc/constant_W_delta=0.csv};
\addplot[brown, quiver={
        u = {y/(sqrt(y^2+4*(-\lam*\kap*y/sqrt(x) + x - \K)^2))},
        v = {2*(-\lam*\kap*y/sqrt(x) + x - \K)/(sqrt(y^2+4*(-\lam*\kap*y/sqrt(x) + x - \K)^2))},
        scale arrows = 0.05
      }, -stealth, very thick] 
      table[
        x index=12,
        y index=13,
        col sep=comma,
      ] {anc/constant_W_delta=0_arrows.csv};

\filldraw (\K,0) circle (3pt);

\draw (2.3,-1.9) node {$L_{\nu_-}$};
\draw (2.3,0.85) node {$L_{\nu_+}$};

\end{axis}
\end{tikzpicture}
\caption{Phase portrait of \eqref{eq-autonomous-kappa-pos} for $\lambda=1$, $\beta=\sqrt{7/8}$, $\alpha=2$. The green (resp.\ red) curve represents the stable (resp.\ unstable) manifold at the fixed point $(\alpha,0)$. The blue curves are the global solutions lying in the region $D_+ \cup D_-$. The brown curves are some prototypical orbits that do not lie entirely in $D_+ \cup D_-$, they all have $x \to 0$ at some finite $s$ (either forwards or backwards or both).}
\label{fig-beta-delta0=0-phase-portrait}
\end{figure}

More generally, if $\pi/2 < W_0 <\pi$ we have $\beta > 0$, and we rewrite \eqref{eq-beta-delta=0} as the first order autonomous system
\begin{equation}\label{eq-autonomous-kappa-pos}
    \begin{cases}
        x' = y,\\
        y' = -2\beta x^{-\frac12}y + 2(x-\alpha).
    \end{cases}
\end{equation}
We display the phase portrait of \eqref{eq-autonomous-kappa-pos} in Figure \ref{fig-beta-delta0=0-phase-portrait}, and the remainder of the section is devoted to studying its properties.
Let $F$ be the vector field associated to \eqref{eq-autonomous-kappa-pos}
The differential of $F$ is
\begin{equation*}
    \diff F
    =
    \begin{bmatrix}
        0 & 1\\
        \beta x^{-\frac32}y +2 & -2 \beta x^{-\frac12}
    \end{bmatrix}.
\end{equation*}
The system has a single fixed point at $(\alpha,0)$.
The characteristic equation and
eigenvalues of $\diff F$ at $(\alpha, 0)$ are respectively given by
\begin{equation*}
    \mu^2 + 2\beta\alpha^{-\frac12}\mu - 2 = 0,
    \qquad
    \mu_\pm = -\frac{\beta}{\sqrt{\alpha}} \pm \sqrt{\frac{\beta^2}{\alpha} + 2}, 
\end{equation*}
implying that $(\alpha,0)$ is a saddle point as $\mu_- < 0 < \mu_+$.
The stable manifold enters $(\alpha,0)$ in the direction of $(1,\mu_-)$, while the unstable manifold exits $(\alpha,0)$ in the direction of $(1,\mu_+)$.
We denote the $y>0$ branch of the stable (resp.\ unstable) manifold by $\mathscr{S}_+$ (resp.\ $\mathscr{U}_+$), and the $y < 0$ branch by $\mathscr{S}_-$ (resp.\ $\mathscr{U}_-$). 

Fix $\nu_+ =\sqrt{2} > \mu_+$ and any $\nu_- < \mu_-$, 
and define the auxiliary regions (cf.\ Figure \ref{fig-beta-delta0=0-phase-portrait})
\begin{align*}
    D_+ &= \{ (x,y) \in \R^2 \,\mid\, x\geq \alpha, \; 0 \leq y \leq \nu_+(x-\alpha) \},
    \\
    D_- &= \{ (x,y) \in \R^2 \,\mid\, x\geq \alpha, \; \nu_-(x-\alpha) \leq y \leq 0 \}.
\end{align*}
We will show that a solution starting in $D_+$ (resp.\ $D_-$) stays in $D_+$ (resp.\ $D_-$), is defined forward globally (resp.\ backward globally), and also that the $\mathscr{S}_-$ branch of the stable manifold (resp.\ $\mathscr{U}_+$ branch of the unstable manifold) is contained in $D_+$ (resp.\ $D_-$).

We first note that the Frobenius norm of the differential satisfies
\begin{equation}\label{eq-dF-est}
    |\diff F|^2 = 1 +(\beta x^{-\frac32}y+2)^2 +4\beta^2 x^{-1},
\end{equation}
which is clearly uniformly bounded on $D_+ \cup D_-$, so by the Picard-Lindelöf theorem the solution continues existing (forward or backward) as long as it stays in this region. 
To show that this is the case, we note that the boundary of $D_\pm$ consists of the fixed point $(\alpha,0)$ and rays of the form
\begin{equation*}
    L_\nu = \{(x,y)\,\mid\, x>\alpha,\, y=\nu(x-\alpha)\}    
\end{equation*}
for fixed $\nu \in \R$.
More precisely, $\partial D_\pm = \{(\alpha,0)\} \cup L_0 \cup L_{\nu_\pm}$, cf.\ Figure \ref{fig-beta-delta0=0-phase-portrait}.
Since no orbit can escape $D_\pm$ via the fixed point $(\alpha,0)$ it suffices to show that they cannot escape via the rays.
Along $L_\nu$ we can define a normal vector field via $n_\nu = (\nu,-1)$.
A calculation then shows that
\begin{equation}\label{eq-constant-W-delta0=0-vector-field-normal-part}
    \langle F, n_\nu \rangle|_{L_\nu} = (\nu^2 + 2\beta x^{-\frac12}\nu - 2)(x-\alpha).
\end{equation}
Thus:
\begin{itemize}[itemsep=0.1cm]
    \item Along $L_{\nu_+}$, we have $\langle F, n_{\nu_+}\rangle = 2\sqrt{2}\beta x^{-\frac12}(x-\alpha) > 0$ since $\nu_+ = \sqrt{2}$.
    As $n_{\nu_+}$ is inward-pointing for $D_+$, we see that no forward orbit can escape $D_+$ through $L_{\nu_+}$.
    \item Along $L_{\nu_-}$, we have $\langle F, n_{\nu_-}\rangle > (\nu_-^2 + 2\beta\alpha^{-\frac12}\nu_- - 2)(x-\alpha) > 0$ since $\nu_- < \mu_- < 0$. As $n_{\nu_-}$ is outward-pointing for $D_-$, we see that no backward orbit can escape $D_-$ through $L_{\nu_-}$.
    \item Along $L_0$, we have $\langle F, n_0 \rangle = -2(x-\alpha) < 0$. Since $n_0$ is outward-pointing for $D_+$, we see that no forward orbit can escape $D_+$ through $L_0$. Analogously, no backward orbit can escape $D_-$ through $L_0$.
\end{itemize}
We conclude that the region $D_+$ (resp.\ $D_-$) is forward- (resp.\ backward-) invariant, as desired.
In particular, we also conclude that the solutions emanating from the ray $D_+\cap D_- = [\alpha,\infty)\times\{0\}$ are global (both forwards and backwards).
Furthermore, since the $\mathscr{S}_+$ manifold exits (resp.\ $\mathscr{U}_-$ enters) the fixed point in the direction of $(1,\mu_+)$ (resp.\ $(1,\mu_-)$), we see that it lies in $D_+$ (resp.\ $D_-$) close to the fixed point, and consequently they stay in this region.

The curves $\mathscr{S}_-$ and $\mathscr{U}_+$ together with the fixed point $(K,0)$ thus enclose a region of good orbits, which are globally defined in both directions.
All other orbits can be shown to hit the singularity $x=0$ at some finite $s$, either in the forward or the backward direction (or both), cf.\ Figure \ref{fig-beta-delta0=0-phase-portrait}. 
This corresponds to finite $s$ blow up of the Dirac-Yang-Mills pair and the metric, so we do not pursue these orbits further.

Finally, we study the integrability properties of $\rho = x^{-\frac12}$.
To this end, we first note that the non-constant forward orbits in $D_+$ (resp. $D_-$) have $x' > 0$ (resp.\ $x'<0$) uniformly, so that $x \to \infty$ as $s \to \infty$ (resp.\ $s \to -\infty$). 
On the other hand, when $x$ is sufficiently large we have
\begin{equation*}
    2(1-\epsilon)x \leq y' \leq 2(1+\epsilon)x,
\end{equation*}
so that by a comparison argument we have:
\begin{itemize}
    \item $x \gtrsim e^{\sqrt{2(1-\epsilon)} s}$ for forward orbits in $D_+$ and sufficiently positive $s$,
    \item $x \gtrsim e^{-\sqrt{2(1+\epsilon)}s}$ for backward orbits in $D_-$ and sufficiently negative $s$.
\end{itemize}
It follows $x^{-\alpha}$ is forwards integrable in $D_+$ (resp.\ backwards integrable in $D_-$) for any $\alpha > 0$, and hence so is $\rho = x^{-\frac12}$.

\subsection{\texorpdfstring{The case $\delta_0>0$}{The case delta0 > 0}}

Finally, we study the solutions with constant $W$ and initial data with $\delta_0 > 0$.
This case is considerably more involved than its counterpart from the previous section due to the dependence of the equations on $P$. In particular, here we are unable to decouple the equation for $P$ from the equation for $\rho$, which ultimately yields a three-dimensional dynamical system, which additionally has a singularity at $P=0$.
Due to this complexity, we content ourselves with the following partial result.

\begin{proposition}
\label{prop-constant-W-delta>0}
    Using the notation of Proposition \ref{lem-polar-form}, let $\delta_0 > 0$ and assume $W \equiv W_0$ for a constant $\pi/2 < W_0 \leq \pi$. We set $\beta = \lambda\sin(W_0)$ and $\alpha = 1+4(\lambda^2-\beta^2)$. For simplicity, we consider only initial data with $U_0 = 0$, i.e.\ $\rho_0' = 0$. 
    \begin{enumerate}[itemsep=0.1cm]
        \item If
        \begin{equation*}
            \rho_0 > \left(\sqrt{2\alpha-1 + 2\beta^2 \coth^2 P_0} - \sqrt{2}\beta \coth P_0\right)^{-1},    
        \end{equation*}
        then $\rho \to \infty$ at some finite $s > 0$.
        \item Let $\mu = -\tan(W_0) > 0$. If $\mu < 1$ and
        \begin{equation*}
            \rho_0 \leq \left( 2\alpha-1 + \frac{(\mu^2 + \mu\sqrt{\mu^2+8} + 2)(\alpha-1)}{2(1-\mu^2)\sinh^2P_0} \right)^{-\frac12},
        \end{equation*}
        then $\rho$ is globally defined and tends to $0$ as $s \to \pm \infty$ exponentially. In particular, the associated metric coefficient $r \to 0$ and the DYM pair $(z,\xi)$ has finite limits as $s \to \pm \infty$.
    \end{enumerate}
\end{proposition}

\begin{remark}
    Note that even in the case $\mu < 1$, there is still a non-empty interval $\rho_0$ for which our result gives no information. In fact, our numerical analysis suggests that there exist global solutions for any $\mu > 0$, and in particular the bound for $\rho_0$ in (ii) is not sharp.
\end{remark}

To prove the result, assume $W \equiv W_0$ and $\delta_0 > 0$.
Then \eqref{eq-P-polar-delta>0} implies that
\begin{equation}\label{eq-P-sol-delta>0}
    P(s) = P_0 + 2\lambda \sin (W_0) \int_0^s \rho,
\end{equation}
and \eqref{eq-W-polar-delta>0} implies 
\begin{equation*}
    r(s) = -\frac{8\lambda\cos (W_0)}{\delta_0^2} \rho(s)^3 \csch\left(P_0 + 2\lambda \sin (W_0) \int_0^s \rho\right).
\end{equation*}
In particular we see that we need that $P > 0$ and $\pi/2 < |W_0| \leq \pi$ in order for $r$ to be positive.
We can also express
\begin{equation*}
    \delta_0^2 r\sinh P = -8\lambda \cos (W_0)\rho^3 , \qquad \delta_0^4r^2 \cosh^2 P = 64 \lambda^2 \cos^2 (W_0) \rho^6 \coth^2 P
\end{equation*}
The metric coefficient satisfies
\begin{equation*}
    \frac{r'}{r} = \frac{3\rho'}{\rho} - 2\lambda\sin (W_0) \rho \coth P,
\end{equation*}
so that the Yang-Mills equation becomes 
\begin{equation*}
    \rho'' - \left(\frac{3\rho'}{\rho} - 2\lambda\sin(W_0) \,\rho \coth P\right)\rho' + \rho\left[1- \left(1+4\lambda^2 \cos^2(W_0) (1+\coth^2 P)\right) \rho^2\right] = 0 
\end{equation*}
Defining $x = \rho^{-2}$ and setting, as before, $\beta = \lambda\sin(W_0)$ and $\alpha = 1 + 4 (\lambda^2-\beta^2)\geq 1$ turns the Yang-Mills equation into
\begin{equation}\label{eq-constant-W-delta>0-x}
    x'' + 2\beta x^{-\frac{1}{2}}x'\coth(P) + (2\alpha -2)\coth^2(P) -2x +2\alpha = 0.
\end{equation}
We note that reflecting $s\mapsto -s$ has the same effect as sending $\beta \mapsto -\beta$, so that we do not lose generality by assuming $\beta \geq 0$, i.e.\ $\pi/2 < W_0 \leq \pi$.

\medskip

Firstly, in the particular case $\beta=0$, i.e.\ $W_0 = \pi$, we see that $P \equiv P_0$ is constant and the equation \eqref{eq-constant-W-delta>0-x} reduces to
\begin{equation*}
    x'' - 2x + 2(1+4\lambda^2(1+\coth^2(P_0))) = 0,
\end{equation*}
which can be solved explicitly to get
\begin{align*}
    \rho(s) = \rho_0\bigg[&\left(1 - \rho_0^2(1+4\lambda^2(1+\coth^2 P_0))\right)\cosh(\sqrt{2}s) \\
    &- \frac{r_0U_0}{\sqrt{2}\rho_0} \sinh(\sqrt{2}s) + \rho_0^2(1+4\lambda^2 (1+\coth^2 P_0))\bigg]^{-\frac12}.
\end{align*}
By the same token as in the $\delta_0 = 0$ case (cf.\ \S \ref{subsec-constant-W-delta=0}), this provides a global DYM pair if and only if
\begin{equation*}
    r_0|U_0| \leq \sqrt{2}\rho_0 (1 - \rho_0^2(1+4\lambda^2(1+\coth^2 P_0))),
\end{equation*}
and the special case $\rho_0 = (1+4\lambda^2(1+\coth^2 P_0))^{-\frac{1}{2}}$ corresponds to the fixed point $(\rho,P,W) \equiv (\rho_0, P_\infty, \pi)$, cf.\ Proposition \ref{prop-constant-rho-delta>0} (iv).

\medskip

For the rest of the section we assume $\beta>0$.
We begin by noting that if
\begin{equation}\label{eq-constantW-delta>0-H}
    L = x - (2\alpha-1) + \frac{1}{\sqrt{2}}x' + 2\sqrt{2} \beta x^{\frac12} \coth P,
\end{equation}
then \eqref{eq-P-sol-delta>0} and \eqref{eq-constant-W-delta>0-x} imply that
\begin{equation*}
    L' = \sqrt{2} L - 4\left(\beta x^{\frac12}\coth P + \frac{\sqrt{2}\lambda^2}{\sinh^2 P}\right) < \sqrt{2}L,
\end{equation*}
so that $e^{-\sqrt{2}s} L(s)$ decreases.
In particular, then (for $s > 0$)
\begin{equation*}
    x - (2\alpha-1) + \frac{1}{\sqrt{2}}x' \leq L \leq L_0e^{\sqrt{2}s},
\end{equation*}
which implies that
\begin{equation*}
    x(s) \leq (2\alpha-1) + (x_0-(2\alpha-1)) e^{-\sqrt{2}s} + L_0 \cosh (\sqrt{2}s).
\end{equation*}
This inequality holds regardless of the choice of $P_0 > 0$, and for all $s > 0$ for which the orbit exists.
If $L_0 < 0$, then the right-hand side of the inequality is eventually negative, implying that $x\to 0$ at some finite $s > 0$ (note that $w$ and $P$ stay bounded as long as $x$ is bounded and away from zero), and the orbit is singular.
For $x_0' = 0$, the condition $L_0 < 0$ is equivalent to the one from Proposition \ref{prop-constant-W-delta>0} (i), which proves that part of the result.
This also shows that only initial data with $L_0 \geq 0$ are of potential interest to us, and in fact we see that, for such initial data, $x$ cannot escape to infinity for finite $s > 0$, and thus neither can $x'$ nor $P$, by \eqref{eq-constantW-delta>0-H} and \eqref{eq-P-sol-delta>0}.
Therefore, orbits with $L_0 \geq 0$ are forward global.
Since we are additionally assuming that $x'_0 = 0$ in Proposition \ref{prop-constant-W-delta>0}, we see in particular that orbits with $x_0 \geq 2\alpha-1$ are forward global for any $P_0 > 0$.

\medskip

Now define
\begin{equation*}
    x_\crit(P) = \alpha + (\alpha-1)\coth^2 P = 2\alpha - 1 + \frac{\alpha-1}{\sinh^2 P},
\end{equation*}
and
\begin{equation*}
    Q = \frac12 x' \sinh P, \qquad \zeta = (x-x_\crit(P))\sinh P.
\end{equation*}
Then the system becomes
\begin{equation*}
    \begin{cases}
        Q' = \zeta,\\[0.1cm]
        \zeta' = 2Q + 2\beta x^{-\frac12}\coth P \left( \zeta + \frac{2(\alpha-1)}{\sinh P} \right),\\[0.1cm]
        P' = 2\beta x^{-\frac12}.
    \end{cases}
\end{equation*}


We are only interested in orbits with $\zeta_0 > 0$, $Q_0 = 0$.
Note that if $\zeta_0 > 0$, $Q_0 \geq 0$, then $L_0 \geq 0$, so that these orbits are forward global and have $\zeta > 0$, $Q \geq 0$ for all $s > 0$ since $Q' > 0$ and $\zeta' > 0$ in that region. 
In fact, for all such initial data, $P$ has a finite limit $s\to\infty$, or equivalently $x^{-\frac12}$ is integrable for $s > 0$.
This can be be shown by contradiction using a comparison argument, since if $P\to\infty$, then $\coth P \to 1$ and the equation \eqref{eq-constant-W-delta>0-x} asymptotically has a similar form as the one from \S \ref{subsec-constant-W-delta=0}.
We omit the details for brevity.

In the region $\zeta \geq 0$, we have $x\geq x_\crit(P)$, and therefore we can estimate
\begin{equation*}
    x^{-\frac12} \leq \frac{1}{\sqrt{\alpha + (\alpha-1)\coth^2P}} \leq \frac{\tanh P}{\sqrt{\alpha-1}}.
\end{equation*}
If we set $\mu = 2\beta/\sqrt{\alpha-1} = -\tan(W_0) > 0$, we thus see that
\begin{equation*}
    \begin{cases}
        Q' = \zeta,\\[0.1cm]
        \zeta' \leq 2Q + \mu \zeta + \frac{2\mu(\alpha-1)}{\sinh P},\\[0.1cm]
        P' \leq \mu \tanh P.
    \end{cases}
\end{equation*}
The final inequality implies that, for $s < 0$,
\begin{equation*}
    \sinh P(s) \geq \sinh P_0 \, e^{\mu s},
\end{equation*}
and therefore
\begin{equation}\label{eq-zeta'-bound-s<0}
    \zeta' \leq 2Q + \mu \zeta + \frac{2\mu(\alpha-1)}{\sinh P_0} e^{-\mu s} \quad \text{for} \quad s < 0.
\end{equation}

The inequality (\ref{eq-zeta'-bound-s<0}) allow for a comparison with a simpler system.
More particularly, consider the solution $(\widetilde{Q}, \widetilde{\zeta})$ of the system
\begin{equation*}
    \begin{cases}
        \widetilde{Q}' = \widetilde{\zeta},\\[0.1cm]
        \widetilde \zeta' = 2\widetilde{Q} + \mu \widetilde\zeta + \frac{2\mu(\alpha-1)}{\sinh P_0} e^{-\mu s},
    \end{cases}
\end{equation*}
with the same initial data $(\widetilde{Q}, \widetilde \zeta)(0) = (Q_0,\zeta_0)$ as the original solution.
A standard comparison argument then shows that, as long as $\zeta \geq 0$, we have $\zeta(s) \geq \widetilde\zeta(s)$ for all $s < 0$.
If we put $\kappa = \sqrt{\mu^2 + 8}$ then the solution is explicitly given by
\begin{equation*}
    \widetilde \zeta(s) = 
    \begin{cases}
        \frac{\kappa+\mu}{2}c_1e^{\frac{\kappa+\mu}{2}s} - \frac{\kappa-\mu}{2} c_2 e^{-\frac{\kappa-\mu}{2}s}
        - \frac{\mu^2(\alpha-1)}{(\mu^2-1)\sinh P_0} e^{-\mu s}, & \mu\not=1,
        \\[0.2cm]
        2c_1 e^{2s} - c_2 e^{-s}
        - \frac{2(\alpha-1)}{3\sinh P_0} (1-s)e^{-s}, & \mu=1.
    \end{cases}
\end{equation*}
where
\begin{align*}
    c_1 &= \frac{\kappa-\mu}{2\kappa} Q_0 + \frac{1}{\kappa}\zeta_0 - \frac{\mu(\kappa-3\mu)(\alpha-1)}{2\kappa(\mu^2-1)\sinh P_0},\\[0.1cm]
    c_2 &= \frac{\kappa+\mu}{2\kappa} Q_0 - \frac{1}{\kappa}\zeta_0 - \frac{\mu(\kappa+3\mu)(\alpha-1)}{2\kappa(\mu^2-1)\sinh P_0}.
\end{align*}
Thus it suffices to find conditions for $\widetilde\zeta \geq 0$ to hold for $s \leq 0$.
It is simple to show that $\widetilde\zeta$ always attains negative values for $\mu \geq 1$, so the comparison gives no information in this case.
On the other hand, if $\mu < 1$, then $\widetilde\zeta \geq 0$ for $s \leq 0$ if and only if $\zeta_0 \geq 0$ and $c_2 \leq 0$, or equivalently
\begin{equation*}
    \zeta_0  \geq \max\left\{0, \; \frac{\kappa+\mu}{2} Q_0 + \frac{\mu(\kappa+3\mu)(\alpha-1)}{2(1-\mu^2)\sinh P_0}\right\}.
\end{equation*}
In particular, orbits of the original system with such initial data are backwards global.
Here, $\zeta$ tends to $\infty$ exponentially as $s\to-\infty$.
Since $P$ increases, it cannot be unbounded as $s\to-\infty$, and hence we have that $x$ also tends to $\infty$ exponentially, and in particular $\rho = x^{-\frac12} \to 0$ and is integrable for $s < 0$.
This proves part (ii) of Proposition \ref{prop-constant-W-delta>0}.

\begin{remark}
    Numerical analysis suggests that for the global orbits from Proposition \ref{prop-constant-W-delta>0} (ii), $P$ also has a strictly positive limit as $s \to -\infty$, but we were not able to show this. Note however that even if $P\to0$, the variables $(z,\xi)$ still have a finite limit as $s \to -\infty$, since $\QuatRe(\xi)$ has a well-defined limit by the integrability of $\rho$, and only the polar angles $\arg(z)$ and $\arg(\QuatIm(\xi))$ blow up, while the norms $|z|\to 0$ and $|\QuatIm(\xi)| \to 0$. On the other hand, if one would like to show that the derivatives of $(z,\xi)$ also have well-defined limits, then one should show that $P$ indeed has a strictly positive limit.
\end{remark}

\section{Lifting of solutions}\label{sec-lifting}
This section is devoted to proving Theorem \ref{thm-existence-any-dim}. We will first show how to lift DYM pairs on an odd-dimensional manifold to an even-dimensional product manifold of one dimension higher. Then we will show how how to combine a DYM pair on an odd-dimensional manifold with a parallel spinor on an even dimensional manifold to yield a DYM pair on the product. 

\subsection{The even case}
Let $(M,g)$ be a $(2m - 1)$-dimensional Riemannian spin manifold, $G$ a compact Lie group with Lie algebra $\g$ and $G\to P \to M$ a principal fibre bundle. Let $E \times_\chi V$ be an associated Hermitian vector bundle for some unitary representation $\chi: G \to \U(V)$. Let $M$ be a one-dimensional connected manifold (i.e.\ either an open interval  $I = (a,b) \subset \R$, $\R$ or $\Sph^1$) endowed with the trivial spin structure and let $f \in C^\infty(B)$.
We now describe a procedure for lifting DYM pairs on $M$ to yield DYM pairs on the $2m$-dimensional warped product $B \times_f M = (B \times M,\, \diff s^2 + f^2 g)$. Denote by $\mathrm{pr}_2:B \times M \to M$ the projection onto the second factor and by 
\begin{equation*}
    \widetilde{P} := \mathrm{pr}_2^*P = \{ ((s,x),p) \in (B\times M) \times P \ | \ \pi_{P}(p) =  x \}
\end{equation*}
and 
\begin{equation*}
    \widetilde{E} := \mathrm{pr}_2^*E = \{ ((s,x),v) \in (B\times M) \times P \ | \ \pi_{E}(p) =  x \} \cong \widetilde{P}\times_\chi V
\end{equation*}
the pullbacks of $P$ and $E$, respectively, to $B \times M$. We now summarize how to relate the spinor bundle of $B\times_f M$ to that of $M$. For details see \cite{Baum2} and \cite{Baer}.

For $s \in B$ write $\Sigma^{f(s)^2} M$ for the spinor bundle of $(M,f(s)^2g)$. Then there is a canonical isometry $\Sigma M \to \Sigma^{f(s)^2} M$ , $\Psi \mapsto \widetilde\Psi$ satisfying $X\cdot \widetilde\Psi = f(s)\, \widetilde{X\cdot \Psi}$ or equivalently $\widetilde{X}\cdot\widetilde{\Psi} = \widetilde{X\cdot \Psi}$ where $\widetilde{X} = f(s)^{-1}X$ for all $X \in \Gamma(TM)$ and $\Psi \in \Gamma(\Sigma M)$. 

By slight abuse of notation we write $\Sigma M$ for $\mathrm{pr}_2^* \Sigma M$ and similarily for $\Sigma^{f(s)^2} M$. Then for each $s \in B$
\begin{equation*}
    \restr{\Sigma(B \times_f M)}{\{s\}\times M} \cong  \Sigma^{f(s)^2} M \oplus \Sigma^{f(s)^2} M \cong \Sigma M \oplus \Sigma M.
\end{equation*}
so that
\begin{equation*}
    \Sigma(B \times_f M) \cong \Sigma M \oplus \Sigma M
\end{equation*}
where the two copies of $\Sigma M$ correspond to the chiral subspaces of $\Sigma(B \times M)$. 
For $X \in \Gamma(TM)$ set $\widetilde{X} = f^{-1}\mathrm{pr}_2^*X \in \Gamma(T(B\times_f M))$. 
Clifford multiplication is then given by 
\begin{equation}\label{eq-clifford-on-prod}
    \widetilde{X} \cdot 
    \begin{bmatrix}
    \Psi^+\\
    \Psi^-
    \end{bmatrix}
    = 
    i\begin{bmatrix}
        X\cdot \Psi^-\\
        -X\cdot \Psi^+
    \end{bmatrix}, \quad 
    \nu \cdot 
    \begin{bmatrix}
    \Psi^+\\
    \Psi^-
    \end{bmatrix}
    = 
    -i\begin{bmatrix}
        \Psi^-\\
        \Psi^+
    \end{bmatrix}
\end{equation}
for $X \in \Gamma(TM)$ and $\nu$ is a vertical unit vector field (i.e. $\diff (\mathrm{pr}_2)(\nu) = 0$) which we can explicitly take to be $\nu = \frac{\partial}{\partial s}$ so that $\nu(\alpha)(s) = \alpha'(s)$ for $\alpha \in C^1(B)$.

The Hermitian metric on the spinor bundle is then simply given by
\begin{equation}\label{eq-metric-on-prod}
    \left\langle \begin{bmatrix}
        \Psi^+\\
        \Psi^-
    \end{bmatrix},
    \begin{bmatrix}
        \Phi^+\\
        \Phi^-
    \end{bmatrix}
    \right\rangle = \mathrm{pr}_2^*(\langle \Psi^+,\Phi^+\rangle + \langle \Psi^-,\Phi^-\rangle)
\end{equation}
for all $\Psi, \Phi \in \Sigma(B\times_f M)$.
The relations \eqref{eq-clifford-on-prod} and \eqref{eq-metric-on-prod} remain true for $\Sigma(B\times_f M)\otimes \widetilde E \cong \Sigma M \otimes \widetilde E \oplus \Sigma M \otimes \widetilde E$ provided we take the induced metric from the lifted metric on $\widetilde E$ and the spinor metric. A pair $\Psi^+,\Psi^- \in \Gamma(\Sigma M)$ lifts to a section $ \Psi = \Psi^+\oplus \Psi^-$ in $\Gamma(\Sigma(B\times_f M))$ via $\restr{\Psi}{\{s\}\times M} = {\Psi}^+ \oplus {\Psi}^- \in \Gamma(\Sigma M) \oplus \Gamma(\Sigma M)$.  The spinorial Levi-Civita connection $\nabla$ on $\Gamma(\Sigma(B\times_f M))$ satisfies
\begin{equation*}
    (\nabla\Psi)(\widetilde X) = 
    \frac{1}{f}\begin{bmatrix}
        (\nabla^{\Sigma B}\Psi^+)(X)\\
        (\nabla^{\Sigma B}\Psi^-)(X)
    \end{bmatrix} + \frac{f'(s)}{2f(s)} \begin{bmatrix}
        -X\cdot \Psi^+\\
        X\cdot \Psi^-
    \end{bmatrix},\quad (\nabla\Psi)(\nu) = 0.
\end{equation*}
for $\Psi^+,\Psi^- \in \Gamma(\Sigma M)$ and $X \in \Gamma(TM)$.
A connection $\omega \in \Omega^1(P,\g)$ on $P$ pulls back to a connection $\widetilde\omega = \mathrm{pr}_2^*\omega \in \Omega^1(\widetilde P,\g)$ on $\widetilde P$ whose curvature form satisfies $F_{\widetilde \omega} = \mathrm{pr}_2^* F_\omega \in \Omega^2(B\times_f M; \mathrm{pr}_2^*\Ad(P))$ (for details see \cite{Baum}). One checks that
\begin{equation}\label{eq-YM-pullback}
    \diff^*_{\widetilde \omega} F_{\widetilde \omega} = \frac{1}{f^2}\mathrm{pr}_2^*(\diff^*_{\omega} F_\omega) 
\end{equation}
holds. The connection $\nabla_{\widetilde \omega}$ on $\Gamma(\Sigma(B\times_f M) \otimes \widetilde E)$ induced by $\widetilde \omega$ and the spinorial Levi-Civita connection then satisfies
\begin{equation}\label{eq-connection-on-prod}
    (\nabla_{\widetilde\omega} \Psi)(\widetilde X)  
    = 
    \frac{1}{f}\begin{bmatrix}
        (\nabla_\omega\Psi^+)(X)\\
        (\nabla_\omega\Psi^-)(X)
    \end{bmatrix} + \frac{f'(s)}{2f(s)} \begin{bmatrix}
        -X\cdot \Psi^+\\
        X\cdot \Psi^-
    \end{bmatrix},\quad (\nabla_{\widetilde\omega}\Psi)(\nu) = 0.
\end{equation}
for $\Psi^+,\Psi^- \in \Gamma(\Sigma M \otimes E)$ and $X \in \Gamma(TM)$.

\begin{lemma}\label{lem-dirac-on-prod}
    Let $(\omega,\Psi)$ be a DYM pair on $(M,g)$. Set $\widetilde{\omega} = \mathrm{pr}_2^*\omega$. For a pair of functions $\alpha,\beta \in C^\infty(B;\C)$ set
    \begin{equation*}
        \Psi_{\alpha\beta} = 
        \begin{bmatrix}
        \alpha \Psi\\
        \beta \Psi
        \end{bmatrix} \in \Gamma(\Sigma(B\times_f M)).
    \end{equation*}
    Then the current and Dirac operator satisfy:
    \begin{align*}
        \mathfrak{J}[\Psi_{\alpha\beta}] &= 2f\, \Im(\alpha\bar{\beta})\,\mathrm{pr}_2^*(\diff_\omega^*F_\omega) + 2i\Re(\alpha\bar{\beta})\sum_{k=1}^{\dim \g} \mathrm{pr}_2^*\langle \Psi, \chi_*(\tau_k)\Psi\rangle \, \nu^{\flat} \otimes \tau_k\\
        \dirac_{\widetilde{\omega}} \Psi_{\alpha\beta} &= -\frac{(2m-1)if'}{2f}
        \begin{bmatrix}
        \beta \Psi\\
        \alpha \Psi
        \end{bmatrix} 
        - i\begin{bmatrix}
            \beta'\Psi \\
            \alpha'\Psi
        \end{bmatrix}
    \end{align*}
    where $\{\tau_k\}$ is the lift of a local orthonormal frame for $\Ad(P)$.
\end{lemma}
\begin{proof}
    Given a local orthonormal frame $\{e_1,\dots,e_{2m-1}\}$ of $M$ the set $\{\widetilde{e}_1,\dots,\widetilde{e}_{2m-1}, \nu\}$ is a local orthonormal frame on $B\times_f M$. Using \eqref{eq-connection-on-prod} we compute
    \begin{align*}
        \dirac_{\widetilde\omega} \Psi_{\alpha\beta} &= \sum_{i=1}^{2m-1} \widetilde{e}_i \cdot (\nabla_{\widetilde\omega}\Psi_{\alpha\beta})(\widetilde{e}_i) + \nu\cdot (\nabla_{\widetilde\omega}\Psi_{\alpha\beta})(\nu)\\
        &=\sum_{i=1}^{2m-1}\left( \frac{i}{f}
        \begin{bmatrix*}[r]
            e_i \cdot \beta(\nabla_\omega \Psi)(e_i)\\
            -e_i \cdot \alpha(\nabla_{\omega}\Psi)(e_i)
        \end{bmatrix*}
        - \frac{if'}{2f} 
        \begin{bmatrix*}
        \beta \Psi\\
        \alpha \Psi
        \end{bmatrix*}\right)
        -i
        \begin{bmatrix*}
        \beta'\Psi\\
        \alpha' \Psi
        \end{bmatrix*}\\
        &= \frac{i}{f}
        \begin{bmatrix*}[r]
            \beta\dirac_\omega \Psi\\
            -\alpha\dirac_\omega \Psi
        \end{bmatrix*}
        -
        \frac{i(2m-1)f'}{2f}
        \begin{bmatrix*}
            \beta \Psi\\
            \alpha \Psi
        \end{bmatrix*}
        - i\begin{bmatrix*}
            \beta'\Psi\\
            \alpha' \Psi
        \end{bmatrix*}\\
        &=-
        \frac{i(2m-1)f'}{2f}
        \begin{bmatrix*}
            \beta \Psi\\
            \alpha \Psi
        \end{bmatrix*}
        - i\begin{bmatrix*}
            \beta'\Psi\\
            \alpha' \Psi
        \end{bmatrix*}.
    \end{align*}
    For the current we have, with respect to this frame,
    \begin{align*}
        \mathfrak{J}[\Psi_{\alpha\beta}] &= \sum_{k= 1}^{\dim \g}\left(\langle \nu\cdot \Psi_{\alpha\beta},\chi_*(\tau_k)\Psi_{\alpha\beta} \rangle\, \nu^\flat \otimes \tau_k + \sum_{i=1}^n \langle \widetilde{e}_i\cdot \Psi_{\alpha\beta},\chi_*(\tau_k)\Psi_{\alpha\beta }\rangle \, \widetilde{e}^i\otimes \tau_k\right)\\
        &=\sum_{k= 1}^{\dim \g}\left(i(\beta\bar{\alpha} + \alpha\bar{\beta})\,\mathrm{pr}_2^*\langle  \Psi,\chi_*(\tau_k)\Psi \rangle \nu^\flat \otimes \tau_k + \sum_{i=1}^n i(\beta\bar{\alpha} -\alpha\bar{\beta})\,\mathrm{pr}_2^*(\langle e_i\cdot \Psi,\chi_*(\tau_k)\Psi\rangle) \widetilde{e}^i\otimes \tau_k\right)\\
        &=\sum_{k= 1}^{\dim \g}2i\Re(\alpha\bar{\beta})\,\mathrm{pr}_2^*\langle  \Psi,\chi_*(\tau_k)\Psi \rangle\, \nu^\flat \otimes \tau_k +  2f\,\Im(\alpha\bar{\beta})\,\mathrm{pr}_2^*\mathfrak{J}[\Psi]\\
        &= 2f\,\Im(\alpha\bar{\beta})\,\mathrm{pr}_2^*(\diff_\omega^*F_\omega) + 2i\Re(\alpha\bar{\beta})\sum_{k=1}^{\dim \g} \mathrm{pr}_2^*\langle \Psi, \chi_*(\tau_k)\Psi\rangle \, \nu^{\flat} \otimes \tau_k
    \end{align*}
    where we have used the relation \eqref{eq-clifford-on-prod} and the fact that $(\omega,\Psi)$ is a DYM pair.
\end{proof}
The following then constructs out of each DYM pair on $M$ a one-parameter family of DYM pairs on $B\times_f M$, with $f > 0$ and smooth but otherwise arbitrary if $\dim M  = 3$ and constant otherwise.
\begin{proposition}
    Let $\dim M = 2m-1$ with $m \geq 2$ and $(\omega, \Psi)$ be a DYM pair on $M$. Take $f \in C^{\infty}(M; \R_{>0})$ arbitrary if $m=2$ and $f \equiv 1$ if $m > 2$. Set $\widetilde \omega  = \mathrm{pr}_2^*\omega$ and define  $\Psi_{c} \in \Gamma(\Sigma(B\times_f M)\otimes \widetilde E)$ by
    \begin{equation*}
        \Psi_c =
        cf^{-\frac{2m-1}{2}}\begin{bmatrix*}[c]
            \Psi\\
            -\frac{i}{2|c|^2}\Psi
        \end{bmatrix*}
    \end{equation*}
    for $c \in \C \setminus \{0\}$. Then $(\widetilde \omega, \Psi_c)$ is a DYM pair on $B \times_f M$.
\end{proposition}
\begin{proof}
    For $\alpha, \beta \in C^\infty(B;\C)$ set 
    \begin{equation*}
        \Psi_{\alpha\beta} = 
        \begin{bmatrix}
            \alpha \Psi\\
            \beta \Psi
        \end{bmatrix}
    \end{equation*}
    Note that $f,\alpha$ and $\beta$ are either smooth periodic functions with the same period if $B = \Sph^1$ or else smooth functions on an open interval if $B = I \subset \R$.
    From Lemma \ref{lem-dirac-on-prod} and the formula  \ref{eq-YM-pullback} we have that $(\widetilde{\omega},\Psi_{\alpha\beta})$ is a DYM pair if and only if $f,\alpha$ and $\beta$ satisfy
    \begin{equation*}
        \begin{cases}
            \alpha' = -\frac{2m-1}{2}\frac{f'}{f}\alpha\\
            \beta' = -\frac{2m-1}{2}\frac{f'}{f}\beta\\
            \alpha\bar{\beta} = \frac{i}{2f^3}
        \end{cases}
    \end{equation*}
    The first pair of equations give
    \begin{equation*}
        \begin{cases}
            \alpha = c_{\alpha}f^{-\frac{2m-1}{2}}\\
            \beta = c_\beta f^{-\frac{2m-1}{2}}
        \end{cases}
    \end{equation*}
    for $c_\alpha,c_\beta \in \C$. Inserting into the third equation yields the condition
    \begin{equation}\label{eq-lifting cond}
        \frac{c_\alpha \bar{c}_\beta}{f^{2m-1}} = \frac{i}{2f}.
    \end{equation}
    If $m = 2$ this simplifies to $c_\alpha\bar{c}_\beta = \frac{i}{2}$ which is satisfied by taking $c_\alpha =: c \in \C\setminus\{0\}$ and $c_\beta = -\frac{ic}{2|c|^2}$. If $m > 2$ there are no solutions to \eqref{eq-lifting cond} unless $f$ is constant. If $f$ is assumed constant we may without loss of generality set $f \equiv 1$. Taking $c_\alpha,c_\beta$ as in the $m=2$ case then yields the desired solutions.
\end{proof}

\begin{remark}
    Since $\dim B = 1$, the warped product $B\times_f M$ is conformally equivalent to the Riemannian product $\widetilde{B}\times M$ where $\widetilde{B}$ is another (possibly different) one-dimensional Riemannian manifold. Therefore, the existence of lifts with non-constant warping function only for $\dim M = 3$ can be viewed as a consequence of the conformal invariance of the DYM system in dimension 4.
\end{remark}

\begin{remark}
    We suspect this construction to be readily adaptable to produce coupled DYM pairs on base manifolds of Lorentzian signature of the form $(B\times M, -\diff s^2 +f^2g)$ where $f \in C^\infty(B;\R_{>0})$ and $\dim M =3$.
\end{remark}

\subsection{The odd case}
Let $(M,g), P$ and $E$ be as before with $\dim M = 2m-1$ and let $(B,h)$ be a $2n$-dimensional Riemannian spin manifold. We again denote the pullbacks of $P$ and $E$ by $\widetilde{P}$ and $\widetilde{E}$, respectively. We again briefly recall how to relate the spinor bundle on the product to those on the factors. For a more detailed and thorough exposition we recommend the note \cite{Klinker}. The spinor bundle of the Riemannian product $(B\times M,h + g)$ is given by 
\begin{equation*}
    \Sigma(B \times M) = \Sigma B \otimes \Sigma M,
\end{equation*}
where these are again the pullbacks of the bundles on the factors.
For $(x,y) \in B \times M$, $X \in T_x B \subset T_{(x,y)}(B\times M)$, $Y \in T_y M \subset T_{(x,y)} (B\times M)$, $\Phi = \Phi^+ + \Phi^- \in \Sigma_{x}^+B \oplus \Sigma_{x}^-B$ and $\Psi \in \Sigma_y M$ we have
\begin{align}\label{eq-clifford-on-prod-odd}
    \begin{split}
        X \cdot (\Phi \otimes \Psi) &= (X\cdot \Phi) \otimes \Psi\\
        Y \cdot (\Phi \otimes \Psi) &= \Phi^+\otimes (Y \cdot \Psi) - \Phi^- \otimes (Y\cdot \Psi).
    \end{split}
\end{align}
for Clifford multiplication. The Hermitian metric on the spinor bundle is
\begin{equation}\label{eq-metric-on-prod-odd}
    \langle \Phi_1 \otimes \Psi_1 , \Phi_2 \otimes \Psi_2 \rangle_{(x,y)} = \langle\Phi_1,\Phi_2\rangle_x\langle\Psi_1,\Psi_2\rangle_y
\end{equation}
where $\Phi_1,\Phi_2 \in \Sigma_x B$ and $\Psi_1, \Psi_2 \in \Sigma_y M$. Both formulae remain true for $\Sigma (B\times M) \otimes \widetilde{E} \cong \Sigma B \otimes (\Sigma M \otimes \widetilde{E})$ provided that one takes the metric induced by the pullback metric on $\widetilde{E}$. The spinorial Levi-civita connection $\nabla$ on $\Sigma (B\times M)$ is the tensor product of the pullbacks of the spinorial Levi-Civita connections on $\Sigma B$ and $\Sigma M$, i.e.
\begin{equation*}
    \nabla (\Phi \otimes \Psi) = \nabla^{\Sigma B} \Phi \otimes \Psi + \Phi \otimes \nabla^{\Sigma M} \Psi
\end{equation*}
for $\Psi \in \Gamma(\mathrm{pr}_2^* \Sigma M)$ and $\Phi \in \Gamma(\mathrm{pr}_1^* \Sigma B)$. 

A connection $\omega \in \Omega^1(P,\g)$ on $P$ pulls back to a connection $\widetilde\omega = \mathrm{pr}_2^*\omega \in \Omega^1(\widetilde P,\g)$ on $\widetilde P$ whose curvature form satisfies $F_{\widetilde \omega} = \mathrm{pr}_2^* F_\omega \in \Omega^2(B\times_f M; \mathrm{pr}_2^*\Ad(P))$ and
\begin{equation}\label{eq-YM-pullback-odd}
    \diff^*_{\widetilde \omega} F_{\widetilde \omega} = \mathrm{pr}_2^*(\diff^*_{\omega} F_\omega). 
\end{equation}
The induced connection $\nabla_{\omega}$ on $\Gamma(\Sigma(B\times M) \otimes \widetilde{E})$ satisfies 
\begin{equation}\label{eq-connection-on-prod-odd}
    \nabla_{\widetilde{\omega}}(\Phi \otimes \Psi) = (\nabla^{\Sigma B}\Phi) \otimes \Psi + \Phi \otimes (\nabla_{\omega}\Psi),
\end{equation}
for $\Phi \in \Gamma(\mathrm{pr}_1^* \Sigma B)$ and $\Psi \in \Gamma(\mathrm{pr}_2^*(\Sigma M \otimes E))$. 

A pair $\Phi \in \Gamma(\Sigma B), \Psi \in \Gamma(\Sigma M \otimes E)$ lift to a section $\Phi \otimes \Psi := \mathrm{pr}_1^*\Phi \otimes \mathrm{pr}_2^*\Psi$ in $\Gamma(\Sigma(B\times M) \otimes \widetilde{E})$. 
If $(\omega,\Psi)$ is a DYM pair on $M$ the following gives necessary and sufficient condtions for when $\Psi \otimes \Phi$ is a DYM pair on $M\times B$.
\begin{proposition}\label{prop-nec&suf-odd-case}
    Let $\dim M = 2m - 1$ with $m \geq 2$ and let $(\omega,\Psi)$ be a DYM pair on $M$. Let $(B,h)$ be a Riemannian spin manifold with $\dim B = 2n$, $n \geq 1$ and let $\Phi = \Phi^+ +\Phi^- \in \Gamma(\Sigma M)$  Then $(\widetilde{\omega}, \Phi \otimes \Psi)$ and is a DYM pairs on $B\times M$, where $\widetilde{\omega} = \mathrm{pr}_2^*\omega$ if and only if $\Phi$ satisfies
    \begin{equation*}
        \begin{cases}
            \slashed{\partial}\Phi = 0\\
            |\Phi^+|^2-|\Phi^-|^2 \equiv 1
        \end{cases}
    \end{equation*}
    and at least one of the following conditions holds:
    \begin{enumerate}
        \item $\langle\Psi,\chi_*(\tau) \Psi\rangle = 0$ for all $\tau \in \Gamma(\Ad(P))$, or
    \item $\langle \Phi, Y \cdot \Phi\rangle = 0$  for all $ Y \in \Gamma(TB)$.
    \end{enumerate}
\end{proposition}

\begin{proof}
    A simple calculation using the formulae \eqref{eq-clifford-on-prod-odd} and \eqref{eq-metric-on-prod-odd} yields
    \begin{equation*}
        \mathfrak{J}[\Phi\otimes\Psi] = (|\Phi^+|^2-|\Phi^-|^2)\mathrm{pr}_2^*\mathfrak{J}[\Psi] + \sum_{j=1}^{2n}\sum_{k=1}^{\dim \g} \mathrm{pr}_1^*\langle\Phi,\nu_i\cdot\Phi\rangle\mathrm{pr}_2^*\langle\Psi,\chi_*(\tau_k)\Psi\rangle \nu^i\otimes \tau_k
    \end{equation*}
    from which the formula for the current follows as by assumption $\mathfrak{J}[\Psi] = \diff_\omega^*F_\omega$ and equation \eqref{eq-YM-pullback-odd}. For the Dirac operator one calculates, using relations \eqref{eq-clifford-on-prod-odd} and \eqref{eq-connection-on-prod-odd},
    \begin{equation*}
        \dirac_{\widetilde{\omega}} (\Phi \otimes \Psi) = (\slashed{\partial}\Phi)\otimes \Psi + (\Phi^+-\Phi^-)\otimes \dirac_\omega \Psi = (\slashed{\partial}\Phi)\otimes \Psi
    \end{equation*}
    since $\dirac_\omega \Psi = 0$. A sufficient condition is then clearly that $\Phi$ satisfies
    \begin{equation*}
        \begin{cases}
            \slashed{\partial}\Phi = 0\\
            |\Phi^+|^2-|\Phi^-|^2 \equiv 1\\
            \langle \Phi, Y \cdot \Phi\rangle = 0 \quad \text{for all } Y \in \Gamma(TB).
        \end{cases}
    \end{equation*}
    If, for all $k = 1, \dots, \dim \g$, $\langle\Psi,\chi_*(\tau_k)\Psi\rangle = 0$, then the $\nu^i$-components of the current automatically vanish so we may drop the otherwise necessary final condition above and only need $\Phi$ to satisfy
    \begin{equation*}
        \begin{cases}
            \slashed{\partial}\Phi = 0\\
            |\Phi^+|^2-|\Phi^-|^2 \equiv 1.
        \end{cases}
    \end{equation*}
\end{proof}
\begin{corollary}
    Let $M,B$ and $(\omega, \Psi)$ be as in \ref{prop-nec&suf-odd-case} and suppose in addition $B$ admits a parallel spinor $\Phi^+ \in \Gamma(\Sigma^+B)$ of positive chirality. Then $(\Psi, \frac{\Phi^+}{|\Phi^+|})$ is a DYM pair on $(M\times B, g\oplus h)$
\end{corollary}
A straigthforward family of spin manifolds $B$ possesing parallel spinors is given by the even dimensional flat torii $\mathbb{T}^{2n}$ equiped with the trivial spin structure. In this case $\Sigma B \cong \mathbb{T}^{2n}\times \Sigma^{2n}$ and parallel spinors are simply constant sections. If $B$ is simply connected and irreducible, the condition of possesing a parallel spinor is equivalent to conditions on the Holonomy on $B$ and a complete classification can be found in \cite{Wang}. 
We note that for the spherically symmetric setting with $\dim M = 3$ that was considered throughout the previous stages of the paper, the spherically symmetric spinor fields of the form \eqref{eq-spinor-ansatz} satisfy
\begin{equation*}
    \langle \Psi, \chi_*(\tau_1) \Psi \rangle = \langle \Psi, \chi_*(\tau_2) \Psi \rangle = 0,
    \qquad
    \langle \Psi, \chi_*(\tau_3) \Psi \rangle = \frac{i}{2}(|\psi_1|^2 - |\psi_2|^2) = \frac{i\delta_0}{2r},
\end{equation*}
where $\delta_0$ is as in Lemma \ref{lem-polar-form}, 
as can easily be verified using Table \ref{tab:invariant-actions}.
In particular, the DYM pairs with $\delta_0 = 0$ that were constructed in Proposition \ref{prop-constant-rho-delta=0} and Proposition \ref{prop-constant-W-delta=0} satisfy condition (i) of Proposition \ref{prop-nec&suf-odd-case} above.

\bibliography{refs}

@BOOK{KnappLieGroups,
  title     = "Lie groups {B}eyond an {I}ntroduction",
  author    = "Knapp, Anthony W",
  publisher = "Birkh{\"a}user",
  series    = "Progress in Mathematics",
  year      =  2013
}

@article{Kunzle,
title = {Spherically symmetric {E}instein–{Y}ang–{M}ills–{H}iggs fields for general compact gauge groups},
journal = {Nonlinear Analysis: Theory, Methods and Applications},
volume = {63},
number = {5},
pages = {473-480},
year = {2005},
doi = {https://doi.org/10.1016/j.na.2005.02.086},
author = {H.P. Künzle and Todd A. Oliynyk},
}

@article{Brodbeck,
title = {On symmetric gauge fields for arbitrary gauge and symmetry
groups},
author = {O. Brodbeck},
journal = {Helv. Phys. Acta},
year = {1996},
pages = {321-324}
}

@article {Li,
    AUTHOR = {Li, Wei},
     TITLE = {Removable singularities for solutions of coupled
              {Y}ang-{M}ills-{D}irac equations},
   JOURNAL = {J. Math. Phys.},
  FJOURNAL = {Journal of Mathematical Physics},
    VOLUME = {47},
      YEAR = {2006},
    NUMBER = {10},
     PAGES = {103502, 13},
MRREVIEWER = {Thomas\ H.\ Otway},
       DOI = {10.1063/1.2354330},
       URL = {https://doi.org/10.1063/1.2354330},
}

@article {Otway,
    AUTHOR = {Otway, Thomas H.},
     TITLE = {Removable singularities in coupled {Y}ang-{M}ills-{D}irac
              fields},
   JOURNAL = {Comm. Partial Differential Equations},
  FJOURNAL = {Communications in Partial Differential Equations},
    VOLUME = {12},
      YEAR = {1987},
    NUMBER = {9},
     PAGES = {1029--1070},
       DOI = {10.1080/03605308708820517},
       URL = {https://doi.org/10.1080/03605308708820517},
}

@article {Isobe,
    AUTHOR = {Isobe, Takeshi},
     TITLE = {Regularity and energy quantization for the
              {Y}ang-{M}ills-{D}irac equations on 4-manifolds},
   JOURNAL = {Differential Geom. Appl.},
  FJOURNAL = {Differential Geometry and its Applications},
    VOLUME = {28},
      YEAR = {2010},
    NUMBER = {4},
     PAGES = {359--375},
       DOI = {10.1016/j.difgeo.2010.05.001},
       URL = {https://doi.org/10.1016/j.difgeo.2010.05.001},
}

@article {Parker,
    AUTHOR = {Parker, Thomas H.},
     TITLE = {Gauge theories on four-dimensional {R}iemannian manifolds},
   JOURNAL = {Comm. Math. Phys.},
  FJOURNAL = {Communications in Mathematical Physics},
    VOLUME = {85},
      YEAR = {1982},
    NUMBER = {4},
     PAGES = {563--602},
      ISSN = {0010-3616,1432-0916},
       URL = {http://projecteuclid.org/euclid.cmp/1103921548},
}

@misc {A,
    title={Uncoupled {D}irac-{Y}ang-{M}ills {P}airs on {C}losed {R}iemannian {S}pin {M}anifolds},
    note = {arXiv:2601.22886},
    author={Adam Lindström},
    year={2026},
    eprint={2601.22886},
    archivePrefix={arXiv},
    primaryClass={math.DG},
    URL={https://arxiv.org/abs/2601.22886}, 
}

@article {Wang,
    AUTHOR = {Wang, McKenzie Y.},
     TITLE = {Parallel spinors and parallel forms},
   JOURNAL = {Ann. Global Anal. Geom.},
  FJOURNAL = {Annals of Global Analysis and Geometry},
    VOLUME = {7},
      YEAR = {1989},
    NUMBER = {1},
     PAGES = {59--68},
      ISSN = {0232-704X},
   MRCLASS = {53C25 (53A50)},
  MRNUMBER = {1029845},
MRREVIEWER = {S.\ M.\ Salamon},
       DOI = {10.1007/BF00137402},
       URL = {https://doi.org/10.1007/BF00137402},
}

@misc{Klinker,
      title={The spinor bundle of {R}iemannian products}, 
      note = {arXiv:math/0212058},
      author={Frank Klinker},
      year={2003},
      eprint={math/0212058},
      archivePrefix={arXiv},
      primaryClass={math.DG},
      url={https://arxiv.org/abs/math/0212058}, 
}

@book {Hamilton,
    AUTHOR = {Hamilton, Mark J. D.},
     TITLE = {Mathematical gauge theory},
    SERIES = {Universitext},
      NOTE = {With applications to the standard model of particle physics},
 PUBLISHER = {Springer, Cham},
      YEAR = {2017},
     PAGES = {xviii+657},
       DOI = {10.1007/978-3-319-68439-0},
       URL = {https://doi.org/10.1007/978-3-319-68439-0},
}

@article {Jost,
    AUTHOR = {Jost, J\"{u}rgen and Ke{\ss}ler, Enno and Wu, Ruijun and Zhu,
              Miaomiao},
     TITLE = {Geometric analysis of the {Y}ang-{M}ills-{H}iggs-{D}irac
              model},
   JOURNAL = {J. Geom. Phys.},
  FJOURNAL = {Journal of Geometry and Physics},
    VOLUME = {182},
      YEAR = {2022},
     PAGES = {104669, 24},
       DOI = {10.1016/j.geomphys.2022.104669},
       URL = {https://doi.org/10.1016/j.geomphys.2022.104669},
}

@article{Harnad,
    author = {Harnad, J. and Shnider, S. and Vinet, Luc},
    title = {Group actions on principal bundles and invariance conditions for gauge fields},
    journal = {Journal of Mathematical Physics},
    volume = {21},
    number = {12},
    pages = {2719-2724},
    year = {1980},
    doi = {10.1063/1.524389},
    url = {https://doi.org/10.1063/1.524389}
}

@article {Donaldson,
    AUTHOR = {Donaldson, S. K.},
     TITLE = {Connections, cohomology and the intersection forms of
              {$4$}-manifolds},
   JOURNAL = {J. Differential Geom.},
  FJOURNAL = {Journal of Differential Geometry},
    VOLUME = {24},
      YEAR = {1986},
    NUMBER = {3},
     PAGES = {275--341},
      ISSN = {0022-040X,1945-743X},
   MRCLASS = {57R55 (32C10 32G05 32J15 53C05 58G10)},
  MRNUMBER = {868974},
MRREVIEWER = {Ronald\ J.\ Stern},
       URL = {http://projecteuclid.org/euclid.jdg/1214440551},
}

@book {DoKr,
    AUTHOR = {Donaldson, S. K. and Kronheimer, P. B.},
     TITLE = {The geometry of four-manifolds},
    SERIES = {Oxford Mathematical Monographs},
      NOTE = {Oxford Science Publications},
 PUBLISHER = {The Clarendon Press, Oxford University Press, New York},
      YEAR = {1990},
     PAGES = {x+440},
      ISBN = {0-19-853553-8},
   MRCLASS = {57R57 (57N13 57R55 58D27 58G05)},
  MRNUMBER = {1079726},
MRREVIEWER = {Ronald\ J.\ Stern},
}

@article{GodCor,
    AUTHOR = {Corrigan, Edward F. and Goddard, Peter},
    TITLE = {Construction of Instanton and Monopole Solutions and Reciprocity},
    JOURNAL = {Annals of Physics},
    VOLUME = {154},
    YEAR = {1984},
    NUMBER ={1},
    PAGES = {253-279},
    ISSN = {0003-4916},
    URL = {https://doi.org/10.1016/0003-4916(84)90145-3},
}

@article {ADHM,
    AUTHOR = {Atiyah, M. F. and Hitchin, N. J. and Drinfeld, V. G. and Manin, Yu. I.},
     TITLE = {Construction of instantons},
   JOURNAL = {Phys. Lett. A},
  FJOURNAL = {Physics Letters. A},
    VOLUME = {65},
      YEAR = {1978},
    NUMBER = {3},
     PAGES = {185--187},
      ISSN = {0375-9601,1873-2429},
   MRCLASS = {81E10 (14F05 32L05 53C05 57R25)},
  MRNUMBER = {598562},
MRREVIEWER = {P.\ E.\ Newstead},
       DOI = {10.1016/0375-9601(78)90141-X},
       URL = {https://doi.org/10.1016/0375-9601(78)90141-X},
}

@article {AtSi,
    AUTHOR = {Atiyah, M. F. and Singer, I. M.},
     TITLE = {The index of elliptic operators. {III}},
   JOURNAL = {Ann. of Math. (2)},
  FJOURNAL = {Annals of Mathematics. Second Series},
    VOLUME = {87},
      YEAR = {1968},
     PAGES = {546--604},
      ISSN = {0003-486X},
   MRCLASS = {57.50},
  MRNUMBER = {236952},
MRREVIEWER = {F.\ Hirzebruch},
       DOI = {10.2307/1970717},
       URL = {https://doi.org/10.2307/1970717},
}

@book {Baum,
    AUTHOR = {Baum, Helga},
     TITLE = {Eichfeldtheorie},
  SUBTITLE = {Eine {E}inführung in die {D}ifferentialgeometrie auf {F}aserbündeln},
    SERIES = {Masterclass},
      NOTE = {},
 PUBLISHER = {Springer Spektrum Berlin, Heidelberg},
      YEAR = {2014},
     PAGES = {XIV + 380},
      ISBN = {978-3-642-38539-1},
   MRCLASS = {},
  MRNUMBER = {},
MRREVIEWER = {},
}

@article {Baum2,
    AUTHOR = {Baum, Helga},
     TITLE = {Complete {R}iemannian manifolds with imaginary {K}illing
              spinors},
   JOURNAL = {Ann. Global Anal. Geom.},
  FJOURNAL = {Annals of Global Analysis and Geometry},
    VOLUME = {7},
      YEAR = {1989},
    NUMBER = {3},
     PAGES = {205--226},
      ISSN = {0232-704X},
   MRCLASS = {58G25 (53C21 58G30)},
  MRNUMBER = {1039119},
MRREVIEWER = {Oussama\ Hijazi},
       DOI = {10.1007/BF00128299},
       URL = {https://doi.org/10.1007/BF00128299},
}

@article {Baer,
    AUTHOR = {B\"ar, Christian},
     TITLE = {Extrinsic bounds for eigenvalues of the {D}irac operator},
   JOURNAL = {Ann. Global Anal. Geom.},
  FJOURNAL = {Annals of Global Analysis and Geometry},
    VOLUME = {16},
      YEAR = {1998},
    NUMBER = {6},
     PAGES = {573--596},
      ISSN = {0232-704X,1572-9060},
   MRCLASS = {58G25},
  MRNUMBER = {1651379},
MRREVIEWER = {Jean-Louis\ Milhorat},
       DOI = {10.1023/A:1006550532236},
       URL = {https://doi.org/10.1023/A:1006550532236},
}
\bibliographystyle{amsplain}

\end{document}